\documentclass[11pt,reqno,english]{amsart}
\usepackage{amsfonts,amsmath,latexsym,verbatim,amscd,mathrsfs,color,array}
\usepackage[colorlinks=true]{hyperref}

\usepackage{amsmath,amssymb,amsthm,amsfonts,graphicx,color}
\usepackage{amssymb}
\usepackage{pdfsync}
\usepackage{epstopdf}
\usepackage[colorlinks=true]{hyperref}

\makeatletter
\def\@currentlabel{2.1}\label{e:dispaa}
\def\@currentlabel{2.21}\label{e:dispau}
\def\@currentlabel{2.22}\label{e:dispav}
\def\@currentlabel{2.23}\label{e:dispaw}
\def\@currentlabel{2.24}\label{e:dispax}
\def\theequation{\thesection.\@arabic\c@equation}
\makeatother

\newcommand{\tg}{{\tt g}}
\newcommand{\inn}{{\quad\hbox{in } }}

\newcommand{\onn}{{\quad\hbox{on } }}
\newcommand{\ttt}{\tilde }

\newcommand{\TT}{{\mathcal T}  }

\newcommand{\nn}{ {\nabla}  }

\newcommand{\py}{ {\tt y } }
\newcommand{\hh}{ {\tt h } }
\newcommand{\uu}{ {\bf u } }
\newcommand{\vv}{ {\bf v} }
\newcommand{\pp}{ {\partial} }

\newcommand{\A}{\alpha }
\newcommand{\vp}{\varphi}
\newcommand{\B}{\beta }
\newcommand{\h}{\hat  }

\newcommand{\JJ}{{\mathcal J}}

\newcommand{\UU}{ {\mathcal U}}

\newcommand{\NN}{ {\mathcal N}}

\newcommand{\R} {\mathbb R}
\newcommand{\cuad}{{\sqcap\kern-.68em\sqcup}}

\newcommand{\BB}{{\tilde B}}

\newcommand{\rr }{{\tt r}}

\newcommand{\foral}{\quad\mbox{for all}\quad}
\newcommand{\ve}{\varepsilon}

\newcommand{\be}{\begin{equation}}
\newcommand{\ee}{\end{equation}}

\newcommand{\la}{\lambda}

\newcommand{\equ}[1]{(\ref{#1})}

\renewcommand{\theequation}{\thesection.\arabic{equation}}
 
 \newtheorem{lemma}{Lemma}[section]

\newtheorem{definition}{Definition}
\newtheorem{teo}{Theorem}

\newtheorem{proposition}{Proposition}[section]
\newtheorem{corollary}{Corollary}[section]
\newtheorem{remark}{Remark}[section]
\newcommand{\bremark}{\begin{remark} \em}
\newcommand{\eremark}{\end{remark} }

\long\def\red#1{{\color{black}#1}}
\long\def\blue#1{{\color{black}#1}}

\title{Serrin's overdetermined problem and constant mean curvature surfaces}

\author[M. del Pino]{Manuel del Pino}
\address{\noindent   Departamento de Ingenier\'{\i}a  Matem\'atica and Centro de Modelamiento Matem\'atico (UMI 2807 CNRS), Universidad de Chile, Casilla 170 Correo 3, Santiago, Chile.}
\email{delpino@dim.uchile.cl}

\author[F. Pacard]{Frank Pacard}
\address{\noindent  - Centre de Math\'ematiques Laurent Schwartz UMR-CNRS 7640, \'Ecole polytechnique, Palaiseau 91128, France.}
\email {frank.pacard@math.polytechnique.fr}

\author[J. Wei]{Juncheng Wei}
\address{\noindent  Department of Mathematics, University of British Columbia, Vancouver, BC V6T 1Z2, Canda, and Department of Mathematics, Chinese University of Hong Kong, Shatin, NT, Hong Kong
} \email{jcwei@math.ubc.ca}

\begin{document}

\begin{abstract}

For all $N \geq 9$, we find smooth entire epigraphs in $\R^N$, namely smooth domains  of the form
$\Omega : = \{  x\in \R^N\ /  \  x_N > F (x_1,\ldots, x_{N-1})\}$, which are not half-spaces and in which a problem of the form
 $\Delta u + f(u) = 0 $ in $\Omega$ has a  positive, bounded solution with $0$ Dirichlet boundary data and constant Neumann boundary data on $\partial \Omega$. This answers negatively for large dimensions a question by Berestycki, Caffarelli and Nirenberg \cite{bcn2}. In 1971, Serrin \cite{serrin} proved that a bounded domain where such an overdetermined problem is solvable must be a ball, in analogy to a famous result by Alexandrov that states that an embedded compact surface with constant mean curvature (CMC) in Euclidean space must be a sphere. In lower dimensions we succeed in providing examples
 for domains whose boundary is close to large dilations of a given CMC surface where Serrin's overdetermined problem is solvable.
\end{abstract}

\date{}\maketitle







\setcounter{equation}{0}
\section{Introduction and statement of the main results}

Let $\Omega$ be a domain in $\R^N$ with smooth boundary, and $\nu$ its inner normal. This paper deals with the {\em overdetermined} boundary value problem
\be
\Delta u + f(u) = 0,\ u>0  \inn \Omega, \quad u\in L^\infty (\Omega),
\label{1}\ee
\be
u=0, \quad \frac{\pp u}{\pp \nu} = constant  \onn \pp\Omega
\label{2}\ee
where $f$ is a locally Lipschitz function. The question we want to analyze in this paper is what type of domains are admissible for this problem to have a solution.

  \medskip
  In 1971, Serrin \cite{serrin} established the following  result:

  \medskip
  {\em If $\Omega$ is {\em bounded} and Problem $\equ{1}$-$\equ{2}$ has a solution,
  then  $\Omega$ must necessarily be an Euclidean ball.}

    \medskip
    Serrin's proof was based on  the {\em Alexandrov reflection principle}, introduced in 1956 by Alexandrov \cite{alexandrov}  to prove the following famous result:

  \medskip
   {\em A compact, connected, embedded hypersurface in $\R^N$ whose mean curvature is constant, must necessarily be an Euclidean sphere.}

     \medskip
     The reflection  maximum principle based procedure
  was used in 1979 by Gidas Ni and Nirenberg \cite{GNN}  to derive radial symmetry results for positive solutions of semilinear equations. The reflection principle, named after \cite{GNN} as the {\em moving plane method}, has become a standard and powerful tool for the analysis of symmetries of solutions of nonlinear elliptic equations.

  \medskip
  Serrin had a clever insight into the geometric structure of Problem \equ{1}-\equ{2} to prove his result as an analog of Alexandrov's.
  The purpose of this paper is to  further explore the parallel between Alexandrov's and Serrin's statements. The underlying question is: how do
  (non-compact) embedded constant mean curvature (CMC) surfaces relate with (unbounded) domains where Serrin's problem \equ{1}-\equ{2} is solvable?

  \medskip
  A natural class of unbounded domains to be considered are epigraphs, namely domains $\Omega$ of the form
  \be\label{epi}
  \Omega = \{ x \in \R^{N} \ /\ x_N > \vp (x_1,\ldots, x_{N-1}) \}
  \ee
  where $\vp: \R^{N-1}\to \R$ is a smooth function. In 1997, Berestycki, Caffarelli and Nirenberg \cite{bcn2} proved the following result: If  $\vp$ is uniformly Lipschitz and {\em asymptotically flat} at infinity, and Problem \equ{1}-\equ{2} is solvable, then $\vp$ must be a linear function, in other words $\Omega$ must be a {\em half-space}. This result was improved by Farina and Valdinoci \cite{FV}, by lifting the asymptotic flatness condititon, under the dimension constraint $N\le 3$.

  \medskip
  In \cite{bcn2} the following question was raised: is it true that an unbounded domain $\Omega$ where \equ{1}-\equ{2} is solvable must be either

  $\bullet $ a half-space, or

  $\bullet$ a cylinder $\Omega = B_k\times \R^{N-k}$, where $B_k$ is a $k$-dimensional Euclidean ball, or

  $\bullet$ the complement of a cylinder?

  \medskip
  In particular, the question is whether or not an epigraph \equ{epi} where Serrin's problem is solvable must be a half-space, under no constraints for the smooth function $\vp$. Our first result, Theorem \ref{teo1} below,  establishes that {\bf this is not the case} if $N\ge 9$.

\medskip
In all what follows we shall consider a {\em monostable} nonlinearity $f$
 for which \equ{1}-\equ{2} is indeed solvable in a half-space.
We assume that $f$ is a smooth function such that
\be
f(0) = 0=f(1), \quad f(s) > 0\foral s\in (0,1), \quad f'(1)<0.
\label{f}\ee
Under these conditions, there exists a unique positive solution $w(t)$, which is also increasing, to the problem
\begin{equation}
\label{eqnforw}
w'' + f(w) = 0  \inn (0,\infty), \quad w(0) = 0, \quad w(+\infty) = 1 ,
\end{equation}
which is implicitly defined by the formula
$$
t = \int_{0}^{w(t)} \frac{ ds}{ \sqrt{ 2 \int^1_s f(\tau)d\tau } }  .
$$
Conditions \equ{f} are satisfied by the standard Fisher-Kolmogorov and Allen-Cahn nonlinearities,
$$
f(s) =  s(1-s) ,\quad f(s) = s(1-s^2)  .
 $$
In the latter case, we explicitly have $w(t) = \tanh \left ( t/{\sqrt{2}} \right )$.
Let us observe that the function
 $u(x) = w(x_N)$ solves \equ{1}-\equ{2} in the half-space
  $\Omega = \{ x \in \R^{N} \ /\ x_N > 0 \} $. Our first main result is the following.

  \begin{teo}\label{teo1} Let $f$ satisfy conditions $\equ{f}$. If $N\ge 9$, there exists an epigraph domain $\Omega$ of the form $\equ{epi}$,
  which is not a half-space, such that Problem $\equ{1}$-$\equ{2}$ is solvable.
  \end{teo}

Let us roughly describe the epigraph of Theorem 1. In 1969, Bombieri, De Giorgi and Giusti \cite{BDG} found an example of an entire
function
 in $\R^8$ whose graph $\Gamma$ is a minimal surface in $\R^9$ and it is not a hyperplane (the BDG minimal graph). Let us call $\Omega_{bdg}$ its epigraph.
 Then, for a sufficiently small $\ve>0$,  the epigraph $\Omega$ in Theorem 1 lies in a $O(\ve)$-neighborhood
of $\ve^{-1}\Omega_{bdg}$.
The solution $u$ will be at main order given by $u(x) = w(z) + O(\ve)$ where $z$ designates the normal inner coordinate to  $\pp\Omega$.

\medskip
The result in \cite{BDG} is a counterexample in large dimensions to {\em Bernstein's conjecture},   which asserts that all entire minimal graphs in $\R^N$ must be hyperplanes.  This statement holds true in dimensions $N\le 8$, see \cite{simons} and its references, so that in analogy, it is natural to think that the question in \cite{bcn2} for epigraphs may have an affirmative answer in low dimensions, but this not even known  in dimension $N=2$. Another PDE analogue of Bernstein's problem is De Giorgi's conjecture (1978) \cite{dg}, which states that entire solutions, monotone in one direction must have level sets which are parallel hyperplanes. This is true in dimensions $N=2,3$ \cite{gg,ac}, and under a certain additional condition for $4\le N\le 8$ \cite{savin}. This statement is indeed false for $N\ge 9$ as proven in \cite{dkwdg} by the construction of an example of a monotone solution whose
level sets resemble largely dilated BDG minimal graphs.
Serrin's epigraph question in \cite{bcn2} seems to be much harder.

\medskip
The principle behind Theorem 1 applies, more generally, to domains enclosed by {\bf a large dilation of an embedded CMC surface}, provided
that sufficient information about the surface (such as nondegeneracy) is available.

\medskip
Our second results exhibits two such examples, consisting of non-cylindrical domains of revolution in $\R^3$ where \equ{1}-\equ{2} is solvable for $f$ satisfying \equ{f}. Let us consider first the solid region enclosed by the catenoid $r=\cosh z$,

\be
\Omega_{c}  =  \{  (r\cos \theta ,  r\sin \theta, z) \ /\  0\le r < \cosh z , \  z\in \R \}.
\label{c}\ee

\medskip
\begin{teo}\label{teo2} For each $\ve>0$ sufficiently small
there exists a domain of revolution $\Omega$, which lies within a $\ve$-neighborhood of the dilated solid catenoid
$\ve^{-1}\Omega_c$, such that Problem $\equ{1}$-$\equ{2}$
with $f$ satisfying $\equ{f}$ is solvable.
\end{teo}

The boundary of $\Omega_c$ is a minimal surface. This result is a part of a more general statement regarding {\em embbeded finite-total curvature
minimal surfacesin $\R^3$} a class that includes for instance the Costa and Costa-Hoffmann-Meeks surfaces, which we shall discuss in the next section.

\medskip
On the other hand, a statement similar to Theorem 2 holds for the classical {\em Delaunay surfaces},  a one parameter family of constant mean curvature surfaces of revolution in ${\R }^3$ which are periodic along one axis which, up to a rigid motion, can be taken to be the  $x_3$-axis. These surfaces, which are called Delaunay surfaces and denoted by $\mathcal D_\tau$, are the boundary of a smooth domain $U_\tau$ and can be parameterized by
\begin{equation}
\label{Delaunay1}
X_\tau (s, \theta) : =  (\varphi (s) \cos \theta , \varphi(s) \sin \theta, \psi (s)) ,
\end{equation}
where the function $\varphi$ is a non constant smooth solution of
\begin{equation}
\label{Delaunay2}
\dot \varphi^2 + (\varphi^2+\tau)^2 = \varphi^2 ,
\end{equation}
and where the function $\psi$ is obtained from
\[
\dot \psi = \varphi^2 + \tau ,\ \ \mbox{with}\  \psi (0)=0.
\]
Here $\tau \in (0, \frac{1}{2}]$ is a parameter which is usually reverend to as the Delaunay parameter.

We have the validity of the following result.

\medskip
\begin{teo}\label{teo3} For each $\ve>0$ sufficiently small
there exists a domain of revolution $\Omega$, which lies within a $\ve$-neighborhood of the region
$\ve^{-1} D_\tau $, such that Problem $\equ{1}$-$\equ{2}$
with $f$ satisfying $\equ{f}$ is solvable.
\end{teo}

The Delaunay surface is compact when regarded as a submanifold of $\R^3$ with the period of the surface mod out. (See Section \ref{S2S} for explanations.)
We shall provide in the next section a more general statement, regarding a general manifold and a compact CMC surface in it,
 from which the above result follows.
We will also express in more detail the result of the nontrivial epigraph and state the result regarding a general minimal surface
with finite total curvature in $\R^3$.
In the later sections we will provide the proof of Theorems 1-3.

\begin{remark}
 In \cite{pacard}
and $f\equiv 0$, $\equ{1}$-$\equ{2}$ is found to be solvable in the domain
$$ \Omega = \{ x\in \R^2 \ /\ |x_2| < \frac \pi 2 + \cosh (x_1) \},$$
except that the solution found is unbounded. These domains are called {\em exceptional domains}.
On the other hand, for $f(u) = \la u$ a class of non-trivial domains in $\R^N$  bifurcating from the cylinder $B_{N-1}\times \R$ and periodic in the last variable, where $\equ{1}$-$\equ{2}$ is solvable, are found in \cite{scicbaldi}.

\end{remark}

\setcounter{equation}{0}
\section{More general statements}\label{S2S}

In this section we make more precise the statements that lead to Theorems \ref{teo1}-\ref{teo3}.
Concerning Theorem \ref{teo1},
we will be able to find a positive, bounded solution of \equ{1}-\equ{2} when $\Omega$ is a small perturbation
of a large dilation of the epigraph of a nontrivial minimal graph in $\R^9$, found by Bombieri, De Giorgi and Giusti in \cite{BDG}
$$ \Gamma  = \{ x\in \R^9 \ /\ x_9 = F (x_1,\ldots, x_8)\, \} . $$
 Let $\nu(y)$ denote the unit normal to $\Gamma$ with $\nu_9>0$.
We consider normal perturbations to a large dilation of $\Gamma$,
namely sets  of the form
\be   \Gamma_\ve := \ve^{-1}\Gamma, \quad  \Gamma_\ve^h  = \{ x= y + h(\ve y) \nu(\ve y) \ / \ y\in \Gamma_\ve \} \label{eps}\ee
for a small positive number  $\ve$  and a smooth function $h$ defined on $\Gamma$.
 We will prove the following result, which makes more precise
the statement of Theorem \ref{teo1}.

\begin{teo}\label{teo4}
For any sufficiently small $\ve>0$ there exists a function $h$   defined on $\Gamma$, with
a uniform $C^2$ bound independent of $\ve$,
such that $\Gamma_\ve^h $ in $\equ{eps}$ is the graph of a smooth entire function, and
letting $\Omega$ be its epigraph, then
 Problem $\equ{1}$-$\equ{2}$ admits a  solution $u_\ve$, with the property that
$$
u_\ve(x) = w(t) + O(\ve ) ,\quad x = y + (t+ h(\ve y))\nu(\ve y)
$$
uniformly for $0< t< \delta \ve^{-1}$, some $\delta  >0$.
Besides,
$$
\pp_\nu u = -w'(0)  \onn  \Gamma_\ve^h .$$
\end{teo}

As we have mentioned in the introduction, this result is analogous to that in \cite{dkwdg}. The construction in this paper is considerably more delicate and require new ideas. The linear theory required here deals with a Dirichlet to Neumann map, and it is more subtle than that in \cite{dkwdg}. As in that work, an infinite-dimensional Lyapunov-Schmidt procedure reduces the problem to a nonlinear, nonlocal equation involving the Jacobi operator.
The lack of symmetry of the seeked surface (unlike the BDG graph itself)
 induces the presence of large errors, and this is a substantial difficulty in the construction. We succeed in overcoming it, by means of a non-trivial refinement on the invertibility theory for the
 Jacobi operator.

\medskip

Next we restrict our attention to the case $N=3$.  The catenoid is the simplest example (besides the plane) of a complete, embedded minimal surface
$\Gamma$ with {\em final total curvature}\, if
$$
\int_\Gamma |K |\, dV\, < \, +\infty
$$
where $K$ denotes the Gauss curvature of the manifold $\Gamma$. Such surfaces are known to have a finite number of ends, which are either planes or
 catenoids with a common axis of rotational symmetry. The first non-trivial example of such a manifold, with genus $1$, was found in 1982 by Costa \cite{costa}. The example was later generalized by Hoffman and Meeks \cite{hm} to arbitrary genus $k\ge 1$. These minimal surfaces $\Gamma$
  are known
 to be {\em non-degenerate}, after the works by Nayatani and Morabito \cite{morabito,nayatani}, in the following sense:

{\em  The only bounded Jacobi fields, namely functions on $\Gamma$ that anhilate the Jacobi operator $\JJ_\Gamma := \Delta_\Gamma -2K$
are originated in rigid motions: rotations around the axis and translations,} namely they are linear combinations of the vector fields
$\nu(x)\cdot e_i$, $i=1,2,3$ and $\nu(x)\cdot x$, where $\nu$ is a unit normal vector field (these surfaces are orientable, they split the space into two components).

We fix such a unit normal $\nu$ for $\Gamma$ and definethe manifolds $\Gamma_\ve$ and $\Gamma_\ve^h$   as in \equ{eps}.


Given this, we have the validity of the following result, that extends Theorem \ref{teo2}.
\begin{teo}\label{teo5}
Let $\Gamma$ be a  complete, embedded minimal surface in $\R^3$
with finite total curvature and non-degenerate. Then for any sufficiently small $\ve>0$ there exists a function $h$
defined on $\Gamma$, with
a uniform $C^2$ bound independent of $\ve$,
such that $\Gamma_\ve^h $ is an embedded and orientable surface, and
letting $\Omega$ be the component of $\R^3$ in the $\nu(\ve y)$-direction, then
 Problem $\equ{1}$-$\equ{2}$ admits a positive bounded solution $u_\ve$, with the property that
$$
u_\ve(x) = w(t) + O(\ve ) ,\quad x = y + (t+ h(\ve y))\nu(\ve y)
$$
uniformly for $0< t< \delta \ve^{-1}$, some $\delta  >0$.
Besides,
$$
\pp_\nu u = -w'(0)  \onn  \Gamma_\ve^h .$$
In the case of a catenoid, the domain and the solution are axially symmetric.
\end{teo}
The corresponding analogue for entire solutions of the Allen-Cahn equation was established in \cite{jdgdkw}.

\medskip
Theorems \ref{teo4} and \ref{teo5} deal with minimal surfaces which have zero mean curvature. It is no surprising  that the right analogue of Serrin's overdetermined problem is the CMC surfaces, namely surfaces with constant mean curvatures.  For a  Riemannian manifold $M$ and a non-degenerate  CMC compact surface  $\Gamma=\pp\Omega $, we get a similar statement, which we will make precise next.

  \medskip

  Let $M$ be a Riemannian manifold and $\Gamma$ a smooth hypersurface,
the boundary of a smooth domain in $M$,
 and $\ve^{-1} M$ its canonical dilation for a small number $\ve>0$. We consider the problem of finding a domain $\Omega$ in $M_\ve$ whose boundary is close to  $\Gamma_\ve$ and encloses a domain $\Omega$
 for which the problem \equ{1}-\equ{2}.

\medskip

For some (small) function $h$ defined on $\Gamma$, the normal graph of $h$ over $\Gamma$ is a hypersurface which will be denoted by $\Gamma_h$. We have assumed that $\Gamma$ is  and we will denote by $\Omega_h$ the domain whose boundary is $\Gamma_h$. There are two choices since the complement of such a domain is also a domain whose boundary if $\Gamma_h$ and to remove the ambiguity we assume that $h \mapsto \Omega_h$ depends continuously on $h$ (in the Hausdorff topology).

\medskip

The mean curvature function of $\Gamma_h$ is denoted by $H(h)$ and its differential at $h=0$ is, by definition, the {\em Jacobi operator} about $\Gamma$. The explicit expression of the Jacobi operator about $\Gamma$ is given by \cite{BCE}
\[
J_\Gamma : =  \Delta_{\Gamma} + |A_\Gamma|^2 + \mbox{Ric} ( {\bf n}, {\bf n}),
\]
where $\Delta_\Gamma$ is the Laplace Beltrami operator on $\Gamma$,  $|A_\Gamma|^2$ is the square of the norm of $A_\Gamma$, namely the sum of the square of the principal curvatures of $\Gamma$ and $\text{Ric}$ is the Ricci tensor on $M$. We recall the following~:
\begin{definition}
A compact hypersurface $\Gamma$ is said to be {\em non degenerate} if  $J_\Gamma$ is injective.
\label{teo61}
\end{definition}

Granted all the definitions above, we now have the:
\begin{teo}
Assume that $\Omega_0 \subset M$ is a smooth bounded domain whose boundary $\partial \Omega_0$ is a non degenerate hypersurface whose mean curvature is constant. Then, for all $\varepsilon >0$ close enough to $0$, there exists $h_\varepsilon \in \mathcal C^{2, \alpha} (\Gamma)$ and  $u_\varepsilon$, a solution of $\equ{1}$-$\equ{2}$  in $\Omega_\varepsilon : = \Omega_{h_\varepsilon}$, such that the family of functions $u_\varepsilon$ tends to $1$ uniformly on compact domains of $\Omega$ as $\varepsilon$ tends to $0$. Moreover, there exists a constant $C >0$ such that
\[
\| h_\varepsilon\|_{\mathcal C^{2, \alpha} (\Gamma)} \leq C \, \varepsilon^2.
\]
\label{teo7}
\end{teo}
Hypersurfaces whose mean curvature is a constant function are know to exist in abundance and the result of \cite{Whi} (see also \cite{Maz-Pac}) shows that, for a generic choice of the ambient metric, they are non degenerate in the sense of Definition~\ref{teo61}. For example, solutions of the iso-perimetric problem, when they are smooth, give rise to hypersurfaces whose mean curvature function is constant.

\medskip
   Theorem \ref{teo7} corresponds to a parallel of the one by Pacard and Ritor\'e \cite{pr}.

Observe that Theorem~\ref{teo7} does not apply to the Delaunay surface, which is non-compact, nor does it apply    the unit ball in Euclidean space since in the case the Jacobi operator about the unit sphere $S^m$ is given by
\[
\Delta_{S^m} + m,
\]
which is not injective since the 	coordinate functions ${\tt x} \mapsto x_j$, for $j=1, \ldots, m+1$, belong to its kernel. However there is an equivariant version of Theorem~\ref{teo7}.

\begin{definition}
Let $\mathfrak G \subset \text{Isom} (M, g)$ be a discrete group of isometries. A compact hypersurface $\Gamma$ is said to be $\mathfrak G$-{\em non degenerate} if there is no nontrivial element in the kernel of $J_\Gamma$ which is invariant by the elements of $\mathfrak G$.
\label{de:2.1}
\end{definition}

We have the validity of the following result, from which Theorem \ref{teo3} follows.
\begin{teo} Assume that $\Omega_0 \subset M$ is a smooth bounded domain and $\mathfrak G \subset \text{Isom} (M, g)$ is a discrete group of isometries which leave $\Omega_0$ globally invariant, namely, $\mathfrak g (\Omega_0) = \Omega_0$ for all  $\mathfrak g \in \mathfrak G$. Further assume that  $\partial \Omega_0$ is a $\mathfrak G$-non degenerate hypersurface whose mean curvature is constant. Then, the conclusion of Theorem~\ref{teo7} hold for a domain $\Omega_\varepsilon$ and a solution $u_\varepsilon$ which are invariant under the action of the elements of $\mathfrak G$.
\label{teo8}
\end{teo}

For example, in the case of the unit ball in the Euclidean $(m+1)$-dimensional space it is enough to consider the group generated by the symmetry through the origin to apply Theorem~\ref{teo8}.  More interesting example is the Denaulay surface given by (\ref{Delaunay1})-(\ref{Delaunay2}) in Section 1.
 From the definition, $\mathcal D_\tau$ is periodic and, if $t_\tau$ denotes the fundamental period of the Delaunay surface of parameter $\tau\in (0, \frac{1}2)$, we can also understand the Delaunay surfaces as constant mean curvature surfaces in $M_\tau : = {\R }^2 \times ({\R } / t_\tau {\bf Z})$ which is endowed with the Euclidean metric $g_{eucl}$. With the  parameterization (\ref{Delaunay1})-(\ref{Delaunay2}), the Jacobi operator about a Delaunay surface reads
\[
J_\tau = \frac{1}{\varphi^2} \, \left( \partial_s^2 + \partial_\theta^2  + \left(\varphi^2 +\frac{\tau^2}{\varphi^2} \right) \right) ,
\]
it has non trivial kernel because of the invariance under the action of translations and in fact, it can be seen from  \cite{Maz-Pac} that the functions $(s, \theta) \mapsto \frac{\dot \varphi}{\varphi}$, $(s, \theta) \mapsto \left( \varphi + \frac{\tau}{\varphi}\right)  \cos \theta $ and $(s, \theta) \mapsto \left( \varphi + \frac{\tau}{\varphi}\right)  \sin \theta $ span the kernel of $J_\tau$. However, if we consider the group of isometries of $(M_\tau, g_{eucl})$ generated by the symmetry with respect to the vertical axis and also with respect to the $x_3=0$ plane, no element in the kernel of this operator is invariant with respect to the action of this group. In particular, Theorem~\ref{teo8} applies to $\mathcal D_\tau$ in $M_\tau$ and, going back to the universal cover ${\R }^{3}$, this leads to the~:

\begin{corollary}
Given $\tau \in (0, \frac{1}{2}]$, there exists for all $\varepsilon >0$ close enough to $0$, a cylindrically bounded domain $\Omega_{\tau, \varepsilon}$ which is periodic along the $x_3$-axis, of period $t_\tau$ and in which one can find positive solutions of $\equ{1}$-$\equ{2}$ . Moreover, the boundary of $\Omega_{\tau, \varepsilon}$ is a normal graph over $\mathcal D_\tau$ for some $\mathcal C^{2, \alpha}$ function whose norm is bounded by a constant times $\varepsilon^2$.
\label{teo6}
\end{corollary}

\bigskip
The proofs of Theorems \ref{teo4}-\ref{teo8} can be set up into a similar scheme, of which the case of the minimal graph Theorem \ref{teo4} is the most complicated, since the surface is noncompact.  So in the rest of the paper we concentrate mainly on the proofs of Theorem \ref{teo4}. The proofs of other theorems will be outlined only.

The organization of this paper is as follows. Section \ref{s4}-\ref{s8} contain the proofs of Theorem \ref{teo4}: in Section \ref{s4} we design a scheme to improve the error up to order ${\mathcal O} (\ve^4)$. The gluing procedure is presented in Section \ref{s5}. In Section \ref{s6} we study an important linear problem which is a Dirichlet to Neumann map.  We then solve the nonlinear projected  problem and  the reduced problem involving the Jacobi operator in Section \ref{s7} which finishes the proof of Theorem \ref{teo4}.
In Section \ref{s8} and Section \ref{s9} we explain the modifications needed to prove Theorem \ref{teo5} and Theorem \ref{teo7} respectively. We delay the solvability of Jacobi operator of the BDG graph in the appendix.

\setcounter{equation}{0}
\section{A first approximation to the nontrivial epigraph} \label{s4}

In what what follows we will denote, for $F$ and $F_0$ as above,
$$
\Gamma = \{ (x',F(x'))\mid \ x' \in \R^8\,\}, \quad  \Gamma_0 = \{ (x',F_0(x'))\mid\ x' \in \R^8\,\}. \quad
$$
By $\Gamma_\ve$ we will  denote the dilated surfaces $\Gamma_\ve= \ve^{-1}\Gamma$. Also,  we shall use the notation:
\be
\rr(x)\,:=\, \sqrt{1 + |x'|^2} , \quad \rr_\ve(x)\, :=\, \rr(\ve x) ,\quad x= (x',x_9) \in \R^8\times\R = \R^9.
\label{re}\ee

We shall refer sometimes to notation and concepts already introduced in \cite{dkwdg}.

\subsection{Local coordinates and the Laplacian near $\Gamma_\ve$}

\medskip
Let us consider the metric $g_{ij}$ of $\Gamma$ around $p$.
Then
$$ g_{ij}(\py) \ := \ \left < \pp_i Y_p,\pp_j Y_p\right> = \delta_{ij} + \theta(\py). $$
We will assume in what follows that the metric satisfies the following uniform estimates: There exists a positive number $C$ such that for all $p\in \Gamma$ we have the estimate
 \begin{align} |\theta(\py) | + |D_\py\theta(\py) | +  |D_\py^4\theta(\py) | \,\le\,  C , \quad   |\py| < 1.
\label{thetap1}\end{align}

\medskip
\subsection{The Laplace Beltrami operator}
The Laplace-Beltrami operator of $\Gamma$ is expressed in these local coordinates as
$$
\Delta_{\Gamma} = \frac 1{\sqrt{\det g( \py) }}\, \pp_{i}\left ( \sqrt{\det g(\py) } \, g^{ij}(\py)\, \pp_j\,\right)
$$
Let us set
$$ a_{ij}^0 (\py) :=  g^{ij}(\py), \quad b_j^0(\py) := \frac 1{\sqrt{\det g( \py) }}\, \pp_{i}\left ( \sqrt{\det g(\py) } \, g^{ij}(\py)\, \right ).
$$
So that
\be
 \Delta_{\Gamma} \,=\,  a^0_{ij}(\py)\,\pp_{ij} +   b^0_i(\py)\,\pp_i, \quad |\py|< \theta R,
\label{L-3}\ee
where
 \begin{align} |a^0_{ij}(\py)-\delta_{ij} | \,\le\,  c\frac {|\py|^2}{R^2} ,&\quad |D_\py a^0_{ij}(\py) | \,\le\,  c\frac {|\py|}{R^2} ,\nonumber \\
  |b_{j}^0(\py) | \,\le\,  c\frac {|\py|}{R^2} ,&\quad\  |D_\py b_{j}^0(\py) | \,\le\,  \frac {c}{R^2}
  \foral   |\py|< \theta R,\ m\ge 2.
\label{L-2}\end{align}

\subsection{The Laplacian near $\Gamma$}

For a certain $\delta>0$ the map
\be
x= X(z,y): = y + z\nu(y), \quad y\in \Gamma, \quad  |z| < \delta r(y)
\label{eer}\ee
defines diffeomorphism onto an expanding tubular neighborhood of $\Gamma$.
Let us consider the manifold
$$ \Gamma^z:= \{ y+ z\nu(y) \ /\ y\in \Gamma\} $$
 The Euclidean Laplacian in $\R^9$ near $\Gamma$ can be expressed in these coordinates
by the well-known formula
\be
\Delta_x  = \pp^2_z + \Delta_{\Gamma^z }  - H_{\Gamma^z}(y)\pp_z
\label{Lapx1}\ee
where $H_{\Gamma^z}(y)$ denotes mean curvature of $\Gamma^z$ at the point  $y + z\nu(y)$
and the operator $\Delta_{\Gamma^z }$ is understood to act on functions of the variable $y$.

\medskip
Using the local coordinates $Y_p(\py)$, \equ{eer} becomes
\be
x= X(z,\py): = Y_p(\py)  + z\nu( \py ),\quad |\py| < \theta R
\label{loccord}\ee
and then the metric tensor $g_{ij}^z$ on $\Gamma^z$ is given by
$$
g_{ij}^z(\py) = g_{ij}(\py) \,+\, z\, \big [\left < \pp_iY_p(\py),\pp_j \nu( \py )   \right > + \left  < \pp_jY_p(\py),\pp_i \nu( \py )   \right >\,\big ] +  z^2\left  < \pp_i\nu (\py),\pp_j \nu( \py )   \right > .
$$
Using that
$$
\nu(\py) =    \frac 1{\sqrt{ 1 + |D_\py G_p(\py)}|^2|}\left [\,-\sum_{j=1}^8 \pp_j G_p(\py) \Pi_j  + \nu(p)\,\right]
$$
for the computation of derivatives we get the expansion
$$
g_{ij}^z(\py) = g_{ij}(\py) + z \theta_1(\py) + z^2\theta_2(\py)
$$
where
$$
|\theta_1(\py)| \, \le \, \frac c R  ,\quad |D_\py \theta_1(\py)|   \, \le \,  \frac c {R^2} ,
$$
\be
|\theta_2(\py)|  \, \le \,  \frac c {R^2}  ,\quad |D_\py \theta_2(\py)|  \, \le \,  \frac c {R^3} .
\ee
Therefore if we let
$$ a_{ij}(\py, z) :=  g^{z ij}(\py), \quad b_j^0(\py) := \frac 1{\sqrt{\det g^z( \py) }}\, \pp_{i}\left ( \sqrt{\det g^z(\py) } \, g^{z ij}(\py)\, \right ).
$$
we get
\be
 \Delta_{\Gamma_z} \,=\,  a_{ij}(\py,z)\,\pp_{ij} +   b_i(\py,z)\,\pp_i,
\label{L1-1}\ee
with $a_{ij}^0$ is given in \equ{L-3},  where
\be
a_{ij}(\py,z) = a_{ij}^0(\py) + za^1_{ij}(\py,z), \quad b_j(\py,z) = b_j^0(\py) + zb^1_{j}(\py,z)
\label{L1-2}\ee
and
$$
|a_{ij}^1|  = O( R^{-1}), \quad  |D_\py a_{ij}^1| = O( R^{-2})\, ,
$$
\be
|b_j^1|  = O( R^{-2}),\quad | D_\py b_j^1| = O( R^{-3}).
\label{L2}\ee

On the other hand, is is well-known that if $k_1,\ldots, k_8$ denote the principal curvatures of  $\Gamma$, then
$$
H_{\Gamma^z} (y) \ =\ \sum_{i=1}^8 \frac {k_i(y)}{1- zk_i(y)}
$$
Since $\Gamma$ is a minimal surface we have that  $\sum_{i=1}^8  k_i = 0 $, therefore
\be
H_{\Gamma^z} (y) \ =\  z|A_{\Gamma} |^2  +  z^2\sum_{i=1}^8 k_i^3  +   z^3\sum_{i=1}^8 k_i^4  + z^4\theta(y,z)
\label{L3}\ee
where
\be
|A_{\Gamma} |^2 =  \sum_{i=1}^8  k_i^2 : = O(r^{-2}),
\label{L4}\ee
\be
|\theta(y,z)| = O( r(y)^{-5} ),\quad  |D\theta(y,z) | = O( r(y)^{-6} ) .
\label{L5}\ee

\subsection{Coordinates near $\Gamma_\ve$}

The previous expressions generalize by scaling to $\Gamma_\ve$ in particular
the coordinates $Y_p$ induce naturally local coordinates in $\Gamma_\ve$.
If $p_\ve  = \ve^{-1}p$, $p\in \Gamma$, we have that the map
\be
\py \in B(0,\theta R/\ve)\subset \R^8 \ \longmapsto \ Y_{p_\ve}(\py) :=  \ve^{-1} Y_p(\ve \py ) \, \in\,\Gamma_\ve.
\label{E0}\ee
defines a local parametrization. The metric on $\Gamma_\ve$ in these coordinates is simply computed as $g_{ij}(\ve\py)$. This yields the expansion

\be
\Delta_{\Gamma_\ve} =  \Delta_\py + (a_{ij}^0(\ve y)-\delta_{ij})\, \pp^2_{ij} + \ve b_i^0(\ve\py) \pp_i, \quad |\py| < \ve^{-1}\theta R.
\label{E1}\ee
We denote in what follows
$$ \rr(y):= \rr (y',y_9) := \sqrt{ 1 + |y'|^2}, \quad y\in \Gamma $$
and
$$ \rr_\ve (y):= \rr (\ve y) , \quad y\in \Gamma_\ve .$$

For some $\delta>0$,  the  following map defines coordinates for a expanding neighborhood of
$\Gamma_\ve$:
\be
x= X(y,z) := y + z\nu(\A y) ,\quad y\in \Gamma_\ve, \ |z| < \delta \ve^{-1}\, \rr( \ve y)
\ee
is computed as
\be
\Delta =  \pp^2_z + \Delta_{\Gamma_\ve^{\ve z}}    - \ve H_{\Gamma^{\ve z}}(\ve y) \, \pp_z
\label{E2}\ee
where now
\be
\Delta_{\Gamma_\ve^{\ve z} }   = \Delta_{\Gamma_\ve} + \ve z a^1_{ij}(\ve\py, \ve z) \pp^2_{ij} +
 \ve^2 z b^1_{j}(\ve \py, \ve z) \pp_{j}.
\label{E3}\ee
and
\be
\ve H_{\Gamma^{\ve z}}(\ve y) = \ve^2 z |A_\Gamma(\ve y)|^2 + \ve^3 z^2\sum_{i=1}^8 k_i(\ve y)^3 \ +
\nonumber\ee
\be
\ve^4  z^3\sum_{i=1}^8 k_i^4(\ve y)   +  \ve^5 z^4 \theta(y,z)
\label{E4}\ee

\subsection{The shifted coordinates}

We consider now a bounded smooth function $h(y)$ defined on $\Gamma$ and the coordinates near $\Gamma_\ve$,
\be
x= X^h(y,t) := y + (t +  h(\ve  y))\nu(\ve y) ,\quad y\in \Gamma_\ve, \ |t| < \delta \ve^{-1}\, \rr( \ve y)
\label{E12}\ee
We compute the Laplacian in these coordinates.  We obtain now

\begin{align}
\Delta_x \, & = \, \pp_t^2 + a_{ij}(\ve \py ,\ve z)\, \pp_{ij} \,+\, \ve b_j(\ve \py ,\ve z)\, \pp_j
\  \nonumber\\
& +\ \ve^2 a_{ij}(\ve \py ,\ve z)\, \pp_i h (\ve \py)\, \pp_j h(\ve \py) \,\pp_t^2 - 2 \ve a_{ij}(\ve \py ,\ve z)\, \pp_i h (\ve \py)  \,\pp_{jt}  \nonumber\\
& -\  \big \{\, \ve^2 \,[ a_{ij}(\ve \py ,\ve z)\,\pp_{ij}h(\ve \py)  \,+\, b_j(\ve \py ,\ve z)\,\pp_jh(\ve \py)\,]\, +\,  \ve \,H_{\Gamma^{\ve z}}(\ve \py)\,\big \} \, \pp_t
\label{euclidean2}\end{align}
where $ z= \ve( t+ h(\ve \py))$.

Since
\be
\Delta_{\Gamma_\ve} = a_{ij}(\ve \py,0) \pp_{ij} + \ve b_i(\A \py,0) \pp_i ,
\label{la}\ee
we can also decompose

\be
\Delta_x = \partial_{tt} + \Delta_{\Gamma_\A}  + B
\label{BB1}\ee
 where the small operator $B$, acting on functions of $(y,t)$ is given in local coordinates by
\begin{align}
B \, & = \,  \ve z a^1_{ij}(\ve \py ,\ve z)\, \pp_{ij} \,+\, \ve^2 z b_j^1 (\ve \py ,\ve z)\, \pp_j
\  \nonumber\\
& +\ \ve^2 a_{ij}(\ve \py ,\ve z)\, \pp_i h (\ve \py)\, \pp_j h(\ve \py) \,\pp_t^2 - 2 \ve a_{ij}(\ve \py ,\ve z)\, \pp_i h (\ve \py)  \,\pp_{jt}  \nonumber\\
& -\  \big \{\, \ve^2 \,[ a_{ij}(\ve \py ,\ve z)\,\pp_{ij}h(\ve \py)  \,+\, b_j(\ve \py ,\ve z)\,\pp_jh(\ve \py)\,]\, +\,  \ve \,H_{\Gamma^{\ve z}}(\ve \py)\,\big \} \, \pp_t
\label{BB2}\end{align}


\subsection{ The perturbed epigraph}
We fix  a positive number $M$ and assume for the moment that $h$ is a smooth function such that

\be
\| D_\Gamma ^2 h\|_{L^\infty(\Gamma)}\ +\ \| D_\Gamma  h\|_{L^\infty(\Gamma)} + \|h\|_{L^\infty(\Gamma)} \ \le \ M
\label{ass0}\ee
uniformly in small $\ve$
and set
$$ \Gamma^h_\ve = \{ y +  h(\ve  y)\nu(\ve y) \ /\  y\in \Gamma_\ve \}. $$
$\Gamma^h_\ve$ is an embedded manifold provided that $\ve$ is sufficiently small, that separates $\R^N$ into two components. We call
$\Omega_\ve^h $ the upper component. Under suitable smallness of $h$,  the implicit function theorem yields that this set
is the epigraph of an entire smooth function $F_\ve^h(x')$,
\be \Omega_\ve^h := \{ (x',x_9) \ /\ x_9> F_\ve^h(x')\} \label{epi1}\ee
whose boundary is of course $\Gamma_\ve^h$.

\subsection{The problem and a first approximation}

We want to solve the problem
$$ S[u]:= \Delta u  + f(u) = 0 \inn \Omega_\ve^h $$
$$u = 0, \quad \pp_\nu u = constant  \onn  \Gamma_\ve^h $$
for a small function $h$ and prove later on that $\Omega_\ve^h $ has the form \equ{epi1}
We observe that in the coordinates
$$
x= y + (t+h(\ve y))\, \nu(\ve y) , \quad y\in \Gamma_\ve,\quad  |t| < \delta\ve^{-1} \rr_\ve(y).
$$
we have that $x\in \Omega_\ve^h$ if and only if $t>0$. The problem for  then becomes

\begin{align}
S[u]\, =\, \Delta_x u \  +\ f(u) \ = \ 0\inn \Omega_{\ve}^h,
\nonumber\\
u(y,0)\, =\,\ 0\ \  \quad \foral y\in \Gamma_\ve,
\nonumber\\
\pp_t u(y,0)\, =\, constant  \foral y\in \Gamma_\ve.
\label{uuuu}\end{align}
We have the existence of a unique solution $w(t)$ to the problem
$$
w''\ + \ f(w) = 0 \inn (0,\infty),
$$
$$
w(0)= 0 , \quad  w(+\infty) = 1.
$$

As a first approximation, close to $\Gamma_\ve^h$ we then take
$$
u_0(x) := w(t).
$$
Using formula \equ{euclidean2}, we find that the error of approximation is then given by
\begin{align}
S[u_0] \, & = \,
  \ve^2 a_{ij}(\ve \py ,\ve z)\, \pp_i h (\ve \py)\, \pp_j h(\ve \py) \,w''(t)
   \nonumber\\
& -\  \big \{\, \ve^2 \,[ a_{ij}(\ve \py ,\ve z)\,\pp_{ij}h(\ve \py)  \,+\, b_j(\ve \py ,\ve z)\,\pp_jh(\ve \py)\,]\, +\,  \ve \,H_{\Gamma^{\ve z}}(\ve \py)\,\big \} \, w'(t) .
\label{error0}\end{align}
where $ z= ( t+ h(\ve \py))$.
Recalling that
$$ a_{ij}(\py,z) = a_{ij}^0(\py) + z a_{ij}^1(\py, z), \quad  b_j(\py,z) =  b_j^0(\py) + z  b_j^1(\py, z),\quad
\Delta_\Gamma = a_{ij}^0 \pp_{ij} + b_j^0\pp_j,
$$
and using the expansion \equ{E4} for the mean curvature, we then write
\begin{align}
S[u_0] \,  = \, -\ve^2 [\Delta_\Gamma h  + |A_\Gamma|^2\,h\,] \, w'\ +   \  \ve^2
a_{ij}\, \pp_i h \, \pp_j h \,w''
   \nonumber\\
-\ [  \ve^2 t |A_\Gamma|^2 w'+ \ve^3 \sum_{i=1}^8 k_i^3 (t+h)^2 w' +  \ve^4 \sum_{i=1}^8 k_i^4  (t+h)^3  w']
\nonumber\\
 -\  \big \{\, \ve^3 (t+ h)\,[   a_{ij}^1\,\pp_{ij}h  \,+\,  b_j^1\,\pp_jh \,]\,w' \, +\,
   \ve^5 \,(t+h)^4 \theta  \,\big \} \, w'
\label{error01}\end{align}
where all coefficients are evaluated at $(\ve \py , \ve(t+ h(\ve\py ))$ or $\ve \py$.

\medskip
What we will do  is to improve this first approximation by choosing $h$ in such a way that
at main order
the relation
$$
\int_\R S[u_0]\, (y, t)\, w'(t) \, dt = 0 \foral y\in \Gamma_\ve
$$
is satisfied.
Under this condition the addition to $u_0$ of a suitable, explicitly computed small term,
reduces the error. We will actually carry out this procedure in successive steps that we describe next.
Let us take the function $h(y)$ to have the following form
\be
h(y) = h_0 + \ve h_1(y) + \ve^2 h_2(y) + \ve^3 \hh (y)
\label{formah}\ee
where the functions $h_0$, $h_1$, $h_2$ will be explicitly chosen

\subsection{First improvement of approximation}
We will next add to $u_0$ a convenient function $\phi_1(y,t)$ of size $O(\ve^2)$ that does not change the
boundary conditions and eliminates the quadratic term in $\ve$  in the new error $S[u_0 +\phi_1]$.

To this end we consider the linear one dimensional problem
\be
p'' + f'(w(t))p = q(t), \quad t\in (0,\infty),\quad p(0) = p'(0) =0
\label{eqpsi}\ee

The solution to this equation is given by
\be
p(t) \ =\ w'(t) \int_0^t \frac {d\tau}{{w'(\tau)}^2}\int_0^\tau w'(s)\,q(s)\, ds \ .
\label{formula p}\ee
If $q$ is a bounded function, $p$ will be bounded if and only if
$$
\int_0^\infty q(t)\,w'(t)\, dt = 0 . $$

Let $c_0$ be the number such that
$$
\int_0^\infty  (t + c_0) w'(t)^2\, dt \ =\ 0,
$$
and let $p_0(t)$ be the solution  of
\be
p_0'' + f'(w(t))p_0  \ = \ (t+c_0)w'(t), \quad t\in (0,\infty),\quad p_0(0) = p_0'(0) =0,
\ee
given by formula $\equ{formula p}$.
We observe that $$ p_0(t) \sim t^3 e^{-t}\quad \hbox{as } t\to +\infty\ .$$


Let
$$
\phi_0(y,t) :=  \ve^2 |A_\Gamma (\ve y )|^2 \,p_0(t).
$$
We observe that
\be
S[u +\phi] = \Delta \phi + f'(u)\phi +  S[u]+ N(u,\phi)
\label{sincrem}\ee
where
$$
N(u,\phi) = f(u + \phi) - f(u) -f'(u)\phi .
$$

Next we estimate the new error of approximation, $S(u_0 + \phi_0)$.
To do this and for later computations, it is
useful  the following lemma, that follows directly  from  formula \equ{euclidean2}.

\begin{lemma}\label{lemin0}
Let $\psi (y)$, $p(t)$ be smooth functions  defined respectively on $\Gamma$  and on $(0,\infty)$. Let us set
$$ \phi(x) \ =\  \psi(\ve y) \, p(t),\quad x= y + (t+ h(\ve y) )\, \nu(\ve y).  $$
Then
\begin{align}
\Delta_x\phi  \, & = \, \psi\, p'' \, +\, \ve^2 \big[ a_{ij}\, \pp_{ij}\psi   \,+\, \ve \,b_j \pp_j\psi \,\big]\, p'
\  \nonumber\\
& +\ \ve^2\psi\, a_{ij}\, \pp_i h \, \pp_j h  \,p'' - 2 \ve^2 a_{ij}\, \pp_i h   \,\pp_{j}\psi\,p'   \nonumber\\
& -\  \big \{\, \ve^2 \,[ a_{ij}\,\pp_{ij}h \,+\, b_j\,\pp_jh\,]\, +\,  \ve \,H_{\Gamma^{\ve(t+h)}}\,\big \} \, \psi \, p'
\label{euclidean3}\end{align}
where the coefficients are evaluated at $\ve \py$ or $(\ve\py, \ve( t+ h(\ve \py))$, and we recall
$$
\ve \,H_{\Gamma^{\ve(t+h)}} = \ve^2(t+ h)|A_\Gamma|^2 + \ve^3(t+ h)^2\sum_{i=1}^8 k_i^3 +  \ve^4(t+ h)^3\sum_{i=1}^8 k_i^4  + \ve^5 (t+ h)^3\theta,
$$
$$
\theta = O( \rr_\ve^{-5}).
$$
\end{lemma}

Using formula \equ{euclidean3} we then get
\begin{align}
\Delta_x\phi_0 + f'(u_0)\phi_0  \, & =      \ve^2 \,|A_\Gamma|^2(t+ c_0)\, w'  + \, \ve^4 \big[ a_{ij}\, \pp_{ij} |A_\Gamma|^2   \,+\, \,b_j \pp_j |A_\Gamma|^2  \,\big]\, p
\  \nonumber\\
& +\ \ve^4 |A_\Gamma|^2 \, a_{ij}\, \pp_i h \, \pp_j h  \,p'' - 2 \ve^3 a_{ij}\, \pp_i h   \,\pp_j |A_\Gamma|^2 \,p'   \nonumber\\
& -\  \big \{\, \ve^4 \,[ a_{ij}\,\pp_{ij}h \,+\, b_j\,\pp_jh\,]\, +\,  \ve^3 \,H_{\Gamma^{\ve(t+h)}}\,\big \} \, |A_\Gamma|^2 \, p'
\end{align}

From formula \equ{sincrem} we have
\be
S[u_0 +\phi_0]  = \Delta \phi_0 + f'(u_0)\phi_0 +  S[u_0] +  \frac 12 f''(u_0)\phi_0^2 + \frac { \phi_0^3 }6 \int_0^1 f'''(u_0+ s\phi_0)\, ds \, .
\label{sin}\ee
Using \equ{error01} and \equ{euclidean3} we then get

\begin{align}
S[u_0 +\phi_0] \, & =
  \, -\,\ve^2\, [\Delta_\Gamma  h  + |A_\Gamma|^2\,h\,] \, w'\ + \ve^2 \,|A_\Gamma|^2(t+ c_0)\, w' -\ve^2 |A_\Gamma|^2t w'
   \nonumber\\
&  \,+   \,  \ve^2
a_{ij}\, \pp_i h \, \pp_j h \,w'' -\ [   \ve^3 \sum_{i=1}^8 k_i^3 (t+h)^2 w' +  \ve^4 \sum_{i=1}^8 k_i^4  (t+h)^3  w']
\nonumber\\
&+ \, \ve^2 \big[ a_{ij}\, \pp_{ij} |A_\Gamma|^2   \,+\, \,b_j \pp_j |A_\Gamma|^2  \,\big]\, p
\nonumber\\
& -\  \big \{\, \ve^3 (t+ h)\,[   a_{ij}^1\,\pp_{ij}h  \,+\,  b_j^1\,\pp_jh \,]\,w' \, +\,
  \ve^5 \,(t+h)^4 \theta  \,\big \} \, w'
\  \nonumber\\
& +\ \ve^4 |A_\Gamma|^2 \, a_{ij}\, \pp_i h \, \pp_j h  \,p'' - 2 \ve^3 a_{ij}\, \pp_i h   \,\pp_j |A_\Gamma|^2 \,p'   \nonumber\\
& -\  \ve^4 \,\big \{\,[ a_{ij}\,\pp_{ij}h \,+\, b_j\,\pp_jh\,]\,  + \, |A_\Gamma|^2(t+ h)
\,+\,   \ve \sum_{i=1}^8 k_i^3 (t+ h)^2 + \cdots \,\big \} \, |A_\Gamma|^2 \, p' \  \nonumber\\
& + \,   \frac {\ve^4}2 f''(w) |A_\Gamma|^4 p^2  +  \,\frac {\ve^6  }6 p^3\, |A_\Gamma|^6  \int_0^1  \,f'''(w+ s\phi_0) \, ds\,.
\label{euclidean5}\end{align}

Next we proceed to make a choice at main order of the parameter function $h$
by writing
$$
h= h_0 + \ve h_1 + \ve^2 h_2 +  \hh
$$
We choose $h_0 \equiv c_0$ and replace in the above expression. We get

\begin{align}
S[u_0 +\phi_0] \, & =
  \, -\,\ve^2\, [\Delta_\Gamma  \hh  + |A_\Gamma|^2\,\hh\,] \, w'\
  \nonumber\\
&  \, - \,\ve^3\, [\Delta_\Gamma h_1  + |A_\Gamma|^2\, h_1\,]\, w' -   \ve^3 \sum_{i=1}^8 k_i^3 (t+h_0)^2 w'
\nonumber\\
&  \,- \ve^4\, [\Delta_\Gamma h_2  + |A_\Gamma|^2\, h_2\,]\, w' - \ve^4\sum_{i=1}^8 k_i^4  (t+h_0)^3  w'
\nonumber\\
&  \,   + \, \ve^4 \Delta_\Gamma\,|A_\Gamma|^2\, p_0
   + \,   \frac {\ve^4}2 f''(w) |A_\Gamma|^4 p_0^2   - \ve^4 \, |A_\Gamma|^4(t+ h_0) \, p'
   \nonumber\\
&  \, -\,  2\ve^4 \sum_{i=1}^8 k_i^3h_1\,(t+h_0)  w'
\nonumber\\&\,
\, + \, {\mathcal R}_1( h),
\label{pl}\end{align}
where
\begin{align}
\, {\mathcal R}_1( h)\, = &\,   \,\  \ve^2
a_{ij}\, \pp_i h \, \pp_j h\,w'' \nonumber\\
&  \,  -\, \ve^3 \sum_{i=1}^8 k_i^3 [\,(t+h)^2-(t+h_0)^2 -2(t+h_0)\ve h_1 \,]  w'
\nonumber\\
&
- \, \ve^4 \sum_{i=1}^8 k_i^4  [(t+h)^3-(t+h_0)^3]  w'
\nonumber\\
&+ \, \ve^5(t+h) \big[ a_{ij}^1\, \pp_{ij} |A_\Gamma|^2   \,+\, \,b_j^1 \pp_j |A_\Gamma|^2  \,\big]\, p_0
\nonumber\\
& -\,   \ve^3 (t+ h)\,[   a_{ij}^1\,\pp_{ij}h  \,+\,  b_j^1\,\pp_jh \,]\,w' \, +\,
  \ve^5 \,(t+h)^4 \theta \, w'
\  \nonumber\\
& +\ \ve^4 |A_\Gamma|^2 \, a_{ij}\, \pp_i h \, \pp_j h  \,p_0'' - 2 \ve^3 a_{ij}\, \pp_i h   \,\pp_j |A_\Gamma|^2 \,p_0'   \nonumber\\
& -\  \ve^4 \,\big \{\,[ a_{ij}\,\pp_{ij}h \,+\, b_j\,\pp_jh\,]\,  + \, |A_\Gamma|^2(h-h_0)
\,+\,   \ve \sum_{i=1}^8 k_i^3 (t+ h)^2 + \cdots \,\big \} \, |A_\Gamma|^2 \, p' \  \nonumber\\
& \,+  \,\frac {\ve^6  }6 p_0^3\, |A_\Gamma|^6  \int_0^1  \,f'''(w+ s\phi_0) \, ds\,.
\label{euclidean7}\end{align}

\medskip
\subsection{Second improvement of approximation}
Similarly to the introduction of $\phi_0$ and $h_0$, a convenient choice of the functions
$h_1$  and $h_2$ will allow us to eliminate the largest terms in the error above.
To this end, we need to achieve at main order the orthogonality
$$
\int_0^\infty  S[ u_0+ \phi_0]\, w'\, dt\, = 0.
$$
Thus we require first that at main order $h_1$ and $h_2$ are such that
$$
\int_0^\infty  \big(\, [\Delta_\Gamma h_1  + |A_\Gamma|^2\, h_1\,]\,w'  +  \sum_{i=1}^8 k_i^3  (t+h_0)^2 w'\,\big)\,w'\, dt \,=\, 0 ,
$$
and
$$
\int_0^\infty  \big(\,  [\Delta_\Gamma h_2  + |A_\Gamma|^2\, h_2\,]\, w'  +  \sum_{i=1}^8 k_i^4  (t+h_0)^3  w'\qquad\qquad
 $$
 $$
\qquad \,+\   \Delta_\Gamma\,|A_\Gamma|^2\, p_0
- \frac {1}2 f''(w) |A_\Gamma|^4 p_0^2   +  |A_\Gamma|^4(t+ h_0) \, p'\,\big)\,w'\, dt \,=\, 0 .
$$

According to Proposition \ref{prophc}  we can find functions $\hh_1$, $\hh_2$, $\hh_3$
such that
$$
 \Delta_\Gamma  \hh_1  + |A_\Gamma|^2\,\hh_1 = \sum_{i=1}^8 k_i^3 ,\quad
$$
$$
 \Delta_\Gamma  \hh_2  + |A_\Gamma|^2\,\hh_2 = \sum_{i=1}^8 k_i^4  ,\quad
$$
$$
 \Delta_\Gamma  \hh_3  + |A_\Gamma|^2\,\hh_3 =  |A_\Gamma|^4   ,\quad
$$
Then we let

$$
h_1\, := \,  c_1 \hh_1, \quad h_2\, := \,  c_2 \hh_2 +  c_3|A_\Gamma|^2 + (c_4- c_3)\hh_3 ,
$$
 where
 $$
c_1\ :=\  -\frac{\int_0^\infty  (t+h_0)^2 {w'}^2\, dt}{\  \int_0^\infty   {w'}^2\, dt }\ ,\quad
 c_2 = -\frac{ \int_0^\infty   (t+h_0)^3 \, {w'}^2 \, dt }{  \int_0^\infty   {w'}^2\, dt  }\, ,\quad
 $$
 $$
 c_3 = \frac{ \int_0^\infty   p_0\, {w'} \, dt} { \int_0^\infty   {w'}^2\, dt }\, ,\quad
  c_4 = \frac{\int_0^\infty   \big [ \frac {1}2 f''(w)p_0^2 - (t+ h_0)p_0' \, \big]\,w' \, dt} { \int_0^\infty   {w'}^2\, dt }\, .
 $$

 Then expression
 \equ{pl} becomes

 \begin{align}
S(u_0 +\phi_0) \,= &
  \, -\,\ve^2\, [\Delta_\Gamma  \hh  + |A_\Gamma|^2\,\hh\,] \, w'\
  \nonumber\\
&  \,  + \, \ve^3 \, \psi_1(\ve y)\, q_1(t)
\,
  +  \ve^4\, \sum_{\ell =2}^5 \psi_\ell (\ve y)\, q_\ell (t)
\nonumber\\& \, +\,  {\mathcal R}_1( h),
\label{plr}\end{align}
where
 \begin{align}
\psi_1(\ve y)\, q_1(t)\, :=&\,  -\sum_{i=1}^8 k_i^3 (\ve y)\,[(t+h_0)^2 w'(t) + c_1w'(t)\,]\,,
\nonumber\\
\psi_2(\ve y)\, q_2(t)\, :=&\,  \ \quad \Delta_\Gamma\,|A_\Gamma(\ve y)|^2\,[\, p_0(t) - c_3 w'(t)\,]\,,
\nonumber\\
\psi_3(\ve y)\, q_3(t)\, :=&\ -  \sum_{i=1}^8 k_i^4(\ve y) [(t+h_0)^3 w'(t) + c_2w'(t)\,]\,,
\nonumber\\
\psi_4(\ve y)\, q_4(t)\, :=&\,  \,\ \quad  |A_\Gamma(\ve y)|^4\, [\,  \frac {1}2 f''(w(t))\,  p_0^2(t)   -  (t+ h_0) \, p_0'(t) - c_4 w'(t)\,],
\nonumber\\
\psi_5(\ve y)\, q_5(t)\, := &\,   -\,  2\sum_{i=1}^8 k_i^3(\ve y)h_1(\ve y)\,(t+h_0)  w'(t)\, .
\label{plr1}\end{align}
The constants $c_\ell$ have been chosen so that
$$
\int_0^\infty q_\ell (t)\, w'(t)\, dt \, =\, 0, \quad \ell =1,\ldots,5.
$$
Hence the solution $p_\ell(t)$ to the problem
\be
p_\ell''(t)\, + \,f'(w(t))\,p_\ell(t) \, = \,q_\ell (t), \quad t\in (0,\infty),\quad p_\ell(0) = p_\ell'(0) =0
\label{eqpsi2}\ee
is bounded. In fact, for all $\ell$ we have
$$ p_\ell (t) = O( t^8e^{-t})\quad\hbox{as } t\to +\infty. $$
Let us set
$$
\phi_1(y,t)\, :=\, \ve^3 \psi_1(\ve y) \,p_1(t) \, +\, \ve^4 \sum_{\ell =2}^5  \psi_\ell(\ve y) \,p_\ell(t)  $$
and consider as a new approximation the function
$$
u_0 + \phi_0 + \phi_1
$$
Then, according to formula \equ{sincrem}, we have that
$$
S[u_0 + \phi_0 + \phi_1] \ =\ S[u_0 + \phi_0] \ + \ \Delta \phi_1 + f'(u_0)\phi_1  \qquad
$$
$$
\qquad \qquad \qquad \qquad \qquad \  + \ [f'(u_0 + \phi_0)-f'(u_0)\,]\,\phi_1  + N(u_0+\phi_0, \phi_1) .
$$
Hence,
 \begin{align}
S[u_0 +\phi_0+ \phi_1]\,=
  \, -\,\ve^2\, [\Delta_\Gamma  \hh  + |A_\Gamma|^2\,\hh\,] \, w'\
  \,  + \,  {\mathcal R}_2( h),
\label{plr3}\end{align}
where
\begin{align}
{\mathcal R}_2( h) \,= & \,
 \ve^3 (\Delta_x -\pp_t^2)\,[ \psi_1 p_1 ]\, + \, \ve^4 \sum_{\ell =2}^5(\Delta_x -\pp_t^2)[ \psi_\ell p_\ell ]
\,
\nonumber\\&\,
\,  +\,  {\mathcal R}_1( h)   \nonumber\\
&  \,    + \,  [f'(u_0 + \phi_0)-f'(u_0)\,]\,\phi_1 \, +\, N(u_0+\phi_0, \phi_1).
\label{plr4}\end{align}

Now, we can estimate with the aid of Lemma \ref{lemin0} the quantities $(\Delta_x -\pp_t^2)[ \psi_\ell p_\ell ]$, for instance when $\hh=0$. We see that in all these functions the action of the operator is roughly that of adding two powers of $\ve$ in smallness and two powers of $\rr_\ve$ in decay.
Thus we have that
 $$ (\Delta_x -\pp_t^2)[ \psi_\ell p_\ell ] = O( \ve^2\rr_\ve^{-4-\mu} e^{-\gamma t}), $$
sizes that do not change with the introduction of $\hh$ that decays as $O(\rr^{-1})$. The size of
${\mathcal R}_2( h)$ is thus globally estimated as   $O( \ve^3\rr_\ve^{-4-\mu} e^{-\gamma t})$.
We will make precise these statements in terms of norms that we introduce  next.

\subsection{Norms}
We introduce here several norms that will be used in the rest of the paper.
Let $\Lambda$ be an open set of $\R^N$ or of an embedded submanifold. For a function $g$ defined on $\Lambda$ we denote, as usual,
$$
\|g\|_{L^2(\Lambda)}^2 := \int_\Lambda |g|^2, \quad
\|g\|_{H^m(\Lambda )} := \int_\Lambda |D^m g|^2 + \int_\Lambda |g|^2,\quad m\ge 1.
$$
We consider also the following local-uniform norms.
\be
\|g \|_{L^2_{l.u.}(\Lambda )} := \sup_{x\in \Lambda} \|g\|_{L^2( B(x,1)\cap \Lambda)} ,\quad \|g \|_{H^m_{l.u.}(\Lambda )}:= \sup_{x\in \Lambda} \|g\|_{L^2( B(x,1)\cap \Lambda)}
\label{lu}\ee

 For a number $0<\sigma < 1$ we denote, as customary,
\be
[\, g \,]_{\sigma , \Lambda}  :=   \sup \big \{ \frac {|g (y_1) - g(y_2 ) |}{ |y_1-y_2|^\sigma }\ /\ y_1,y_2\in \Lambda , \ y_1\ne y_2 \big \}.
\label{semi}\ee
We let
\be
\|g\|_{C^{0,\sigma}(\Lambda)} :=  \| g\|_{L^\infty (\Lambda )} + [\, g \,]_{\sigma , \Lambda}
\label{holder0}\ee
and for $k\ge 1$,
\be
\| g\|_{C^{k,\sigma}_0(\Lambda ) }:=   \| g\|_{C^{0,\sigma}(\Lambda ) } + \| D^k g\|_{C^{0,\sigma}(\Lambda ) }  .   \label{holder1}\ee

\medskip
We consider now a submanifold $\Gamma$ in $\R^{m+1}$ and its dilation $\Gamma_\ve$.
We denote
\be
\rr (y',y_{m+1}) := \sqrt{ 1 + |y'|^2} , \quad y\in \Gamma,
\label{rr}\ee
and
\be
\rr_\ve (y) := \rr (\ve y) , \quad y\in \Gamma_\ve.
\ee
We consider local-uniform weighted H\"older norms involving powers of $\rr_\ve$.
Let $g$ be a function defined on $\Gamma_\ve$. For a number $\nu \ge 0$ we define

\be
\|g \|_{C^{0,\sigma}_\nu (\Gamma_\ve)} \, :=\,   \|\rr_\ve^\nu\,g\|_{L^\infty(\Gamma_\ve)} + \, [ \rr_\ve^\nu\,g \,]_{\sigma, \Gamma_\ve },
\label{norm0} \ee
and for $m\ge 1$,
\be
\|g \|_{C^{m,\sigma}_\nu (\Gamma_\ve)} \, :=\,   \|D^m_{\Gamma_\ve}  g \|_{C^{0,\sigma}_\nu (\Gamma_\ve)}
\, +\, \|g \|_{C^{0,\sigma}_\nu (\Gamma_\ve)}.
\label{norm1} \ee

We will use the same notation to refer to the corresponding norm in $\R^9$ rather than on $\Gamma_\ve$.

\medskip
Let us consider now the case of a function $g(y,t)$ defined on the space $\Gamma_\ve \times (0,\infty)$.
We consider a norm similar to that above, but that measures also exponential decay in the $t$-direction.
For numbers $\nu, \gamma\ge 0$ we denote

\be \|  g\|_{C_{\nu,\gamma}^{0,\sigma}(\Gamma_\ve\times (0,\infty))}\, :=\,
\|\rr_\ve^\nu \, e^{\gamma t} \,g\|_{L^\infty(\Gamma_\ve\times (0,\infty))}\ +
   \, [\, \rr_\ve^\nu \, e^{\gamma t}\,g \,]_{\sigma, \Gamma_\ve\times (0,\infty)}\ .
\label{norm2} \ee

which we observe, is also equivalent to the norm
 \be
\|\rr_\ve^\nu \, e^{\gamma t} \,g\|_{L^\infty(\Gamma_\ve\times (0,\infty))}\ +
   \ \sup_{(y,t)\in \Gamma_\ve\times (0,\infty)} e^{\gamma t}\rr_\ve^\nu(y)\, [\,g \,]_{\sigma, \{B( (x,t),1)\cap \Gamma_\ve\times (0,\infty)\}}\ .
\label{norm22} \ee
We define, correspondingly,

 \be \|  g\|_{C_{\nu,\gamma}^{m ,\sigma}(\Gamma_\ve\times (0,\infty))}\, :=\, \|  g\|_{C_{\nu,\gamma}^{0,\sigma}(\Gamma_\ve\times (0,\infty))}\ +\
\| D^m_{\Gamma_\ve} g\|_{C_{\nu,\gamma}^{0,\sigma}(\Gamma_\ve\times (0,\infty))}
\label{norm3} \ee

We shall denote, also, by simplicity
$$ \| p \|_{C_{0}^{0,\sigma}(\Gamma_\ve)} =: \| p \|_{C^{0,\sigma}(\Gamma_\ve)},\quad
\| g\|_{C_{0,0}^{0,\sigma}(\Gamma_\ve\times (0,\infty))} =: \| g \|_{C^{0,\sigma}(\Gamma_\ve)}
$$

\medskip
The weighted norms introduced above are appropriate for functions that share the same decay properties
as its derivatives. It is important for our purposed to consider a different set of norms
that are suitably adapted to a function $\tg$ defined on a subset $\Lambda$ of $\Gamma$ such that
when differentiated it gains decay in successive negative powers of $\rr(y)$
Let us assume that
\be
\|\,\rr^\nu\, \tg \|_{L^\infty(\Lambda)} <+\infty .
\label{fin}\ee
Roughly speaking, we expect that when differentiated $m$ times, also
$$
\|\,\rr^{\nu+m}\, D^m \tg \|_{L^\infty(\Lambda)} <+\infty .
$$

In this context, since the H\"older seminorm \equ{semi} corresponds roughly to differentiate $\sigma$ times, then it is natural to require that besides \equ{fin}, the following quantity be finite:

$$
[ \tg ]_{\sigma,\nu,\Gamma}\ := \ \sup_{y\in \Gamma} \rr(y)^{\nu + \sigma}  [\, \tg \,]_{\sigma , \{B(y,1)\cap \Gamma\}}
$$
We define
\be
\|\tg\|_{\sigma,\nu,\Lambda} := \| {\rr }^\nu \tg\|_{L^\infty (\Lambda )}  + [ \tg ]_{\sigma,\nu,\Gamma},
\label{norm4}\ee
which is actually equivalent to the norm
$$
\| {\rr }^\nu \tg\|_{L^\infty (\Lambda )}  + [\rr^{\nu + \sigma} \tg ]_{\sigma,\Gamma} .
$$

Let us observe that if $\tg$ is of class $C^1(\Lambda)$ then we
have
$$
\|\tg\|_{\sigma,\nu,\Lambda}\ \le \ C[\, \| {\rr }^\nu \tg\|_{L^\infty (\Lambda )} + \| {\rr }^{\nu+1} D_\Gamma \tg\|_{L^\infty (\Lambda )}\, ].
$$

We define correspondingly

\be
\|\tg\|_{k,\sigma,\nu,\Lambda} := \|\,D_\Gamma^k\tg\|_{\sigma,\nu+k,\Lambda} + \|\tg\|_{k,\sigma,\nu,\Lambda}\ .
\label{norm5}\ee

\bigskip
\subsection{Connection between norms}

A simple but very important fact is the connection existing between the norms
$\| \ \|_{C^{0,\sigma}_{\nu}(\Gamma_\ve)}$ in \equ{norm0} and $\|\ \|_{\sigma,\nu, \Gamma}
$ in \equ{norm4} as described in the following result.

\medskip
Let $p$ be a function defined on $\Gamma$, so that
$$ q(y) := p(\ve y) $$
is defined on $\Gamma_\ve$.

\begin{lemma}\label{lema8}
Let $p$ and $q$ be functions related as above, $\nu>1$. Then there exists a $C>0$ such that the following
inequalities hold.
\be
\| q \|_{C^{0,\sigma}_{\nu}(\Gamma_\ve)} \ \le \ C \|p\|_{\sigma,\nu, \Gamma}
\label{lado1}\ee
and
\be
\|p\|_{\sigma,\nu-\sigma, \Gamma}   \ \le \ C\ve^{-\sigma} \|  q \|_{C^{0,\sigma}_\nu(\Gamma_\ve)} .
\label{lado2}\ee

\end{lemma}

\proof
On the one hand, we have that
$$
\| \rr_\ve^\nu \, q\|_{L^\infty(\Gamma_\ve)} \ = \ \| \rr^\nu \, p\|_{L^\infty(\Gamma)}
$$
On the other hand, for $y_1, y_2\in \Gamma_\ve$ we have that for $\ttt y_l := \ve y_l$,
$$
\frac { (\rr_\ve^\nu q )(y_1) - (\rr_\ve^\nu q )(y_2)}{ | y_1 - y_2|^\sigma }\ =\
\ve^\sigma \frac { (\rr^\nu p)(\ttt y_1) - (\rr^\nu q )(\ttt y_2)}{ | \ttt y_1 - \ttt y_2|^\sigma }
$$

Assuming that $|\ttt y_1- \ttt y_2|< 1$ we get
$$
\frac { (\rr^\nu p)(\ttt y_1) - (\rr^\nu p )(\ttt y_2)}{ | \ttt y_1 - \ttt y_2|^\sigma } =
 \rr^\nu (\ttt y_1) \frac { p(\ttt y_1) -  p (\ttt y_2)}{ | \ttt y_1 - \ttt y_2|^\sigma } +
 p (\ttt y_2)\frac{ \rr^\nu (\ttt y_1)- \rr^\nu (\ttt y_2)}{ | \ttt y_1 - \ttt y_2|^\sigma }
$$
Hence
$$
\frac { |(\rr_\ve^\nu q )(y_1) - (\rr_\ve^\nu q )(y_2)|}{ | y_1 - y_2|^\sigma }\ \le \
C\ve^\sigma [\, \|\rr^\nu  p \|_{L^\infty ( B(\ttt y_1, 1))}  +  \rr^\nu(\ttt y_1) [ p]_{\sigma, B(\ttt y_1, 1)} ]\, \le C\ve^{\sigma} \|p\|_{\sigma,\nu, \Gamma} .
$$
If  $|\ttt y_1- \ttt y_2|\ge 1$ we get that
$$
\frac { |(\rr_\ve^\nu q )(y_1) - (\rr_\ve^\nu q )(y_2)|}{ | y_1 - y_2|^\sigma }\ \le \
C\ve^\sigma \|\rr^\nu  p \|_{L^\infty ( \Gamma )}.
$$
As a conclusion, we get the validity of inequality \equ{lado1}.

\medskip
Now, let us consider an inequality in the opposite direction. Let $\ttt y_0\in \Gamma$ and consider
$\ttt y_1, \ttt y_2\in B(\ttt y_0,1) \cap \Gamma$, and set correspondingly $y_l=\ve^{-1}\ttt y_l$.
Let us assume first that $| \ttt y_1 - \ttt y_2| \le \ve$. We have that
$$
\rr^{\nu} (\ttt y_0) \frac { |p(\ttt y_1) -  p (\ttt y_2)|}{ | \ttt y_1 - \ttt y_2|^\sigma } \ \le \
C \rr_\ve ^{\nu} ( y_0) \ve^{-\sigma} \frac { |q( y_1) -  q(y_2)|}{ | y_1 - y_2|^\sigma } \, \le\,
C \rr_\ve ^{\nu} ( y_0) \ve^{-\sigma} [ \, q\, ]_{\sigma, B(y_0,1)}.
$$

On the other hand if $|\ttt y_1 -\ttt y_2| > \ve $ we have
$$
\rr^{\nu} (\ttt y_0) \frac { |p(\ttt y_1) -  p (\ttt y_2)|}{ | \ttt y_1 - \ttt y_2|^\sigma } \ \le \
C\ve^{-\sigma} \| \rr^\nu_\ve\, q\|_{L^\infty(\Gamma_\ve)}.
$$
Combining these two inequalities yields
$$
[\,p\, ]_{\sigma, \nu-\sigma, \Gamma} \ \le \ C\ve^{-\sigma} \|  q\|_{C^{0,\sigma}_\nu} .
$$
hence, we have obtained the inequality \equ{lado2} and the proof is concluded.
\qed

\bigskip
\begin{remark}\label{rema}
A typical term to which we want to measure its size in $\Gamma_\ve \times (0,\infty)$
has the form
$$
g(y,t)\, = \,    a(y,t)\,  p(\ve y)\, \zeta(t)
$$
where  $\zeta$ is such that  $\zeta(t ) = O(e^{-\gamma t})$ as $t\to +\infty$, as well as its derivatives.
Arguing as in the proof of  Lemma \ref{lema8} we find the estimate:
$$
\| g\|_{C^{0,\sigma}_{\nu,\gamma} (\Gamma_\ve \times (0,\infty))}\ \le \ C\, \|a \|_{C^{0,\sigma}(\Gamma_\ve \times (0,\infty))} \, \|p\|_{\sigma,\nu, \Gamma }.
$$
\end{remark}

\subsection{Conclusion of the construction of the first approximation and computation of error size}\label{1stapprox}
We consider then  $h$ of the form
\be
h(y) \ = \ h_0 + \ve h_1(y) + \ve^2 h_2(y) + \hh (y),\quad y\in \Gamma.
\label{formh1}\ee
where on the $\ve$-dependent parameter function $\hh$,
we assume in what follows that for some fixed $\mu>0$,
\be
\|\hh\|_{2,2+\mu,\sigma, \Gamma}:= \| D^2_\Gamma \hh \|_{4+\mu,\sigma, \Gamma } + \| D_\Gamma \hh \|_{3+\mu,\sigma, \Gamma } + \| \hh \|_{2+\mu,\sigma, \Gamma }\ \le \  \ve .
\label{assh}\ee

We observe that from Corollary \ref{coro}, we have that the functions $h_1$ and $h_2$ satisfy
\be
\|h_1\|_{2, 1,\sigma, \Gamma}\ < \ +\infty ,\quad  \|h_2\|_{2, 2-\tau ,\sigma, \Gamma}\ < \ +\infty ,
\label{formh2}\ee
for any small $\tau>0$.
Since $h_0$ is a constant, we point out that we have in particular

\be
|\nn_\Gamma  h (y)|  \le  C\ve \rr(y)^{-2}   .
\label{formh3}\ee

\medskip
The approximation already built, $u_0 + \phi_0 + \phi_1$ is sufficient for our purposes, except that
it is only defined near $\Gamma^h_\ve$.
Since we have that
$$ w(t) =  1 - O( e^{-\gamma t}) \quad\hbox{as } t\to +\infty
$$
we consider a simple interpolation with the function $1$.
We let $\eta(s)$ be a smooth function with $\eta(s) =1$ is $s<1$ and $=0$ if $s>2$.
For a sufficiently small $\delta>0$ we let
$$
\eta_m (x) = \eta (s), \quad s = \ve t -  {m \delta},\quad x =  y+  (t+ h(\ve y) ) \, \nu(\ve y)
$$
understood this function as identically zero at any point outside its support.

\be
u_1(x) =     \eta_{10}( x)( u_0 + \phi_0 + \phi_1)  + (1- \eta_{10}( x ))\, (+1) ,\quad  s = \ve t -  {10 \delta}\, \rr_\ve (y)
\label{primo}\ee
where the function is understood to be identically equal to $+1$, all over the space, outside the support
of $\eta_{10}$.

\medskip
\subsection{Error size and Lipschitz property}
We shall investigate the size of the error in terms of the H\"older type norms introduced
in the previous section, as well as its Lipschitz dependence on $\hh$. The main part of the error of approximation is of course in the region close to $\Gamma_\ve$. In reality, we can consider the error as a function defined in the entire space $(y,t)\in \Gamma_\ve \times (0,\infty)$,  by setting
\be E(y,s)\ :=\   \eta_3 S[u_1] \ =\ \eta_3\, S[u_0 + \phi_0 + \phi_1]\ ,
\label{error7} \ee
this function understood as zero outside its support.

\medskip
Let us consider the expansion \equ{pl} of the first approximation error and the operator ${\mathcal R}_1( h),$ there appearing, defined in \equ{euclidean7} as
\begin{align}
\, {\mathcal R}_1( h)\, = &\,   \,\  \ve^2
a_{ij}\, \pp_i h \, \pp_j h\,w'' \nonumber\\
&  \,  -\, \ve^3 \sum_{i=1}^8 k_i^3 [\,(t+h)^2-(t+h_0)^2 -2(t+h_0)\ve h_1 \,]  w'
\nonumber\\
&
- \, \ve^4 \sum_{i=1}^8 k_i^4  [(t+h)^3-(t+h_0)^3]  w'
\nonumber\\
&+ \, \ve^5(t+h) \big[ a_{ij}^1\, \pp_{ij} |A_\Gamma|^2   \,+\, \,b_j^1 \pp_j |A_\Gamma|^2  \,\big]\, p_0
\nonumber\\
& -\,   \ve^3 (t+ h)\,[   a_{ij}^1\,\pp_{ij}h  \,+\,  b_j^1\,\pp_jh \,]\,w' \, +\,
  \ve^5 \,(t+h)^4 \theta \, w'
\  \nonumber\\
& +\ \ve^4 |A_\Gamma|^2 \, a_{ij}\, \pp_i h \, \pp_j h  \,p_0'' - 2 \ve^3 a_{ij}\, \pp_i h   \,\pp_j |A_\Gamma|^2 \,p_0'   \nonumber\\
& -\  \ve^4 \,\big \{\,[ a_{ij}\,\pp_{ij}h \,+\, b_j\,\pp_jh\,]\,  + \, |A_\Gamma|^2(h-h_0)
\,+\,   \ve \sum_{i=1}^8 k_i^3 (t+ h)^2 + \cdots \,\big \} \, |A_\Gamma|^2 \, p' \  \nonumber\\
& \,+  \,\frac {\ve^6  }6 p_0^3\, |A_\Gamma|^6  \int_0^1  \,f'''(w+ s\phi_0) \, ds\,.
\label{euclidean8}\end{align}

Let us consider for instance the term,

$$
{\mathcal R}_{11} ( h) =  \ve^2\, a_{ij} ( \ve y , \ve( t+ h)) \, \pp_i h \, \pp_j h\,w''
$$
where we recall
$$ h = h_0 + \ve h_1 + \ve^2 h_2 + \hh . $$
with $\hh$ satisfying \equ{assh}.
Then we have
$$
|\eta_3 {\mathcal R}_{11} ( h)| \, \le \, C \,\ve^2\,  |D_\Gamma  h (\ve y) |^2  \, e^{-\gamma t} \le
 C \, \ve^4\, \rr_\ve^{-4-\mu } \, e^{-\gamma t} ,
$$
so that
$$
\| e^{\gamma t} \rr_\ve^{4 + \mu} \eta_3{\mathcal R}_1( h) \|_{L^\infty( \Gamma_\ve \times (0,\infty))} \ \le \ C \, \ve^4\, .
$$
Using Remark \ref{rema}, we find moreover that
$$
\| \eta_3{\mathcal R}_{11} ( h)\|_{C^{0,\sigma}_{4+\mu,\gamma} (\Gamma_\ve \times (0,\infty))}\ \le \ C\,\ve^4.
$$
Similar estimates are obtained for the remaining terms in  ${\mathcal R}_1( h).$
We then find
$$
\| \eta_3{\mathcal R}_{1} ( h)\|_{C^{0,\sigma}_{4+\mu,\gamma} (\Gamma_\ve \times (0,\infty))}\ \le \ C\,\ve^4.
$$
We want to investigate  next the Lipschitz character of the operator   ${\mathcal R_1}( h).$
Let us consider again our model operator ${\mathcal R_{11}}( h).$
We have that

$$
\ve^{-2}[\eta_3{\mathcal R}_{11} ( h_1)- {\mathcal R}_{11} ( h_2)]\, = \,  [\, \eta_3a_{ij}( \ve y, \ve (  t + h_1))\,-\,
\eta_3a_{ij}( \ve y, \ve (  t + h_2))\, ]\, \pp h^1_i \pp h^1_j \,w''
$$

$$
\qquad\qquad\qquad \, +\,\eta_3a_{ij}( \ve y, \ve (  t + h_2))\, [\, \pp h^1_i \pp h^1_j - \pp h^2_i \pp h^2_j \, ]\,w'',
$$
and hence from Remark \ref{rema}
$$
\ve^{-2}\| \eta_3{\mathcal R}_{11} ( h_1)- \eta_3{\mathcal R}_{11} ( h_2)\|_{C^{0,\sigma}_{4+\mu,\gamma} (\Gamma_\ve \times (0,\infty))} \, \le  \,
$$

$$
C \| \eta_3a_{ij}( \ve y, \ve (  t + h_1))\,-\,
\eta_3a_{ij}( \ve y, \ve (  t + h_2)) \,\|_{C^{0,\sigma} (\Gamma_\ve \times (0,\infty))} \, \,
\| \pp h^1_i \pp h^1_j  \,\|_{\sigma, 4 +\mu, \Gamma}
$$

$$
\ + \
C \| \eta_3a_{ij}( \ve y, \ve (  t + h_2))\,\|_{C^{0,\sigma} (\Gamma_\ve \times (0,\infty))} \, \,
\|  \pp h^1_i \pp h^1_j - \pp h^2_i \pp h^2_j   \,\|_{\sigma, 4 +\mu, \Gamma}.
$$
Thus
$$
\| \eta_3{\mathcal R}_{11} ( h_1)- \eta_3{\mathcal R}_{11} ( h_2)\|_{C^{0,\sigma}_{4+\mu,\gamma} (\Gamma_\ve \times (0,\infty))} \, \le  \, C\ve^{3}\| D_\Gamma h_1 -D_\Gamma h_1\|_{\sigma, 2 +\mu, \Gamma}\ +\ \, C\ve^{5}\| h_1 - h_1\|_{\sigma,0,\Gamma}.
$$
Similar estimates are obtained for the remaining terms. In all we have,
\be
\| \eta_3{\mathcal R}_{1} ( h_1)- \eta_3{\mathcal R}_{1} ( h_2)\|_{C^{0,\sigma}_{4+\mu,\gamma} (\Gamma_\ve \times (0,\infty))} \, \le  \, C\ve^{3}\|  h_1 - h_1\|_{2,\sigma, 2 +\mu, \Gamma}\,
\label{lip1}\ee
uniformly on $h_1$, $h_2$ satisfying \equ{assh}.

Now, let us consider expression \equ{plr3}
for the error at $u_0+ \phi_0 +\phi_1$.
 \begin{align}
S[u_0 +\phi_0+ \phi_1]\,=
  \, -\,\ve^2\, [\Delta_\Gamma  \hh  + |A_\Gamma|^2\,\hh\,] \, w'\
  \,  + \,  {\mathcal R}_2( h),
\label{plr31}\end{align}
where
\begin{align}
{\mathcal R}_2( h) \,= & \,
 \ve^3 (\Delta_x -\pp_t^2)\,[ \psi_1 p_1 ]\, + \, \ve^4 \sum_{\ell =2}^5(\Delta_x -\pp_t^2)[ \psi_\ell p_\ell ]
\,
\nonumber\\&\,
\,  +\,  {\mathcal R}_1( h)   \nonumber\\
&  \,    + \,  [f'(u_0 + \phi_0)-f'(u_0)\,]\,\phi_1 \, +\, N(u_0+\phi_0, \phi_1).
\label{plr44}\end{align}

The terms $(\Delta_x -\pp_t^2)[ \psi_\ell p_\ell ]$ are of the same nature as those in the
operator ${\mathcal R}_1( h)$, which was computed from $(\Delta_x -\pp_t^2)[ |A_\Gamma|^2 p_0 ]$.
We get in fact extra smallness and decay.  Therefore we get

$$
\| \eta_3 {\mathcal R}_{2} ( h)\|_{C^{0,\sigma}_{4+\mu,\gamma} (\Gamma_\ve \times (0,\infty))}\ \le \ C\,\ve^4.
$$
and
\be
\| \eta_3{\mathcal R}_{2} ( h_1)- \eta_3{\mathcal R}_{2} ( h_2)\|_{C^{0,\sigma}_{4+\mu,\gamma} (\Gamma_\ve \times (0,\infty))} \, \le  \, C\ve^{3}\|  h_1 - h_1\|_{2,\sigma, 2 +\mu, \Gamma}\,
\label{lip2}\ee
uniformly on $h$, $h_1$, $h_2$ satisfying \equ{assh}.

\medskip
As a conclusion of the above considerations it follows that we can write
the error \equ{error7} as
\be
E\  =\  [\,\Delta_\Gamma \hh + |A_\Gamma|^2 \hh\,]\, w'    \ + \ {\mathcal R}_3(\hh)
\label{error8}
\ee
where the operator ${\mathcal R}_3(\hh)$ satisfies
\be
\|{\mathcal R}_{3} ( h)\|_{C^{0,\sigma}_{4+\mu,\gamma} (\Gamma_\ve \times (0,\infty))}\ \le \ C\,\ve^4.
\label{r30}\ee
and
\be
\| {\mathcal R}_{3} ( h_1)- {\mathcal R}_{3} ( h_2)\|_{C^{0,\sigma}_{4+\mu,\gamma} (\Gamma_\ve \times (0,\infty))} \, \le  \, C\ve^{3}\|  h_1 - h_1\|_{2,\sigma, 2 +\mu, \Gamma}\,
\label{r31}\ee

\setcounter{equation}{0}
\section{The gluing reduction}\label{s5}
We want to solve the problem
$$S(u) := \Delta u  + f(u) = 0 \inn \Omega_\ve^h $$
$$u = 0, \quad \pp_\nu u = constant  \onn  \Gamma_\ve^h $$
where we have our first approximation $u_1$, built in \S \ref{1stapprox}.  We write
$ u = u_1 + \ttt \phi $. We recall that
$$\pp_\nu u_0 = w'(0)= \, constant \onn \Gamma_\ve^h. $$
 Thus the problem gets rewritten as
$$
\Delta \ttt \phi  + f'(u_1)\ttt \phi + N(\ttt \phi) + E = 0 \inn \Omega_\ve^h
$$
\be \ttt \phi = 0, \quad \pp_\nu \ttt \phi = 0  \onn  \Gamma_\ve^h \label{pbb1}\ee
$$  S[u_1],\quad   N(\ttt \phi)= f(u_1 +\ttt \phi) -f(u_1) -f'(u_1)\ttt \phi. $$

We recall that we set in \S \ref{1stapprox},
$$
\eta_{m}(y,t) = \eta( m^{-1} (\ve t -  \delta  \rr_\ve (y) ) )\ .
$$
We look for a solution of Problem \equ{pbb1} with the form
$$
\ttt\phi = \eta_2\phi + \psi,
$$
with $\phi(y,t)$ is defined in entire $\Gamma_\ve \times (0,\infty)$.

We obtain a solution to Problem  \equ{pbb1} if we solve
the following system.

\begin{align}
 \pp^2_t\phi + \Delta_{\Gamma_\ve } \phi  + f'(w(t))\phi \  = \quad\qquad\qquad \quad
 \nonumber \\
  -\eta_{3}[S[u_1] + B\phi+ N(\eta_1\phi +\psi) + (f'(u_0 )+1) \psi]\quad  \hbox{in } \Gamma_\ve\times (0,\infty) ,\nonumber \\
 \phi(y,0)   = 0\foral  y \in \Gamma_\ve,  \nonumber\\
 \pp_t \phi(y,0) = -  \pp_t \psi(y,0)   \foral   y \in \Gamma_\ve , \label{gl1}
 \end{align}
 \begin{align}
\Delta \psi  - \psi  \,  =\, -  [  2\nn \phi \cdot\nn \eta_{1} +  \phi \Delta\eta_{1} \, ]\qquad \qquad  \nonumber\\
 - (1-\eta_1)\,[\, (f'(u_1) +1)\, \psi + S[u_1] + N(\eta_1 \phi +\psi)\,] \,  \quad \hbox{in } \Omega_\ve  \,
\nonumber \\
 \psi  = 0 \onn  \Gamma_\ve .
 \label{gl2} \end{align}


We have the following result.

\begin{lemma}
Let $\phi$ be given with $\|\phi\|_{C^{2,\sigma}_{\nu,\gamma}(\Gamma_\ve\times (0,\infty))} <1$ and $\nu \ge 2$.
 Then equation \equ{gl2} can be solved
as $\psi = \Psi(\phi)$ where for some $a>0$,
$$
\|\Psi(\phi_1) -\Psi(\phi_2)\|_{C^{2,\sigma}_{\nu}(\Omega_\ve)} \le e^{-\frac a\ve} \|\phi_1 -\phi_2\|_{C^{2,\sigma}_{\nu,\gamma}(\Gamma_\ve\times (0,\infty)},
$$
$$
\|\Psi(0)\|_{C^{2,\sigma}_{\nu}(\Omega_\ve)} \ \le \ e^{-\frac a\ve}.
$$
\end{lemma}

\proof

Let  us consider first the linear problem
\begin{align}
\Delta \psi - \psi = g \inn \Omega_\ve,\quad \psi=0 \quad\hbox{on }\Gamma_\ve^h .
\end{align}
By a standard barrier argument, this problem has a unique bounded solution $\psi= T(g)$ if $g$ is bounded. In fact
$$\|\psi\|_{L^\infty(\Omega_\ve)}\le \|g\|_{L^\infty(\Omega_\ve)}$$
Let us further assume that
$
\|g\|_{C^{0,\sigma}(\Omega_\ve)}< +\infty $. Applying local boundary and interior Schauder estimates
we also get
\be
\|\psi\|_{C^{2,\sigma} (\Omega_\ve)}  \le C\|g\| _{C^{0,\sigma}(\Omega_\ve)}
\label{cr}\ee
for a constant $C$ uniform in all small $\ve$. Finally,
by writing $\psi = r_\ve^{-\nu}\ttt\psi $, examining the equation for $\ttt\psi$ and using estimate
\equ{cr} we obtain that
\be
\|T(g) \|_{C_\nu^{2,\sigma} (\Omega_\ve)}\ \le\ C\|g\|_{C_\nu^{0,\sigma} (\Omega_\ve)}.
\label{cotapsi}\ee
Now, we can solve Problem \equ{gl2}
as the fixed point problem
$$
\psi =  - T(\,  [  2\nn \phi \cdot\nn \eta_{1} +  \phi \Delta\eta_{1} \, ] +
  (1-\eta_1)\,[\, (f'(u_0) +1)\, \psi + E + N(\eta_1 \phi +\psi)\,]\, ).
$$
It is readily checked that the right hand side of this equation defines a contraction mapping
 in a region of the form $
\|\psi\|_{C^{2,\sigma}_{\nu}(\Omega_\ve)} \ \le \ e^{-\frac a\ve}.
$
Banach fixed point then gives a solution of this problem with the desired properties. \qed

\bigskip
Through this lemma we have reduced our original equation to solving the nonlinear, nonlocal problem for a $\phi$ with $\|\phi\|_{C^{2,\sigma}_{\nu,\gamma}(\Gamma_\ve\times (0,\infty))} <1,$

\begin{align}
 \pp^2_t\phi + \Delta_{\Gamma_\ve } \phi  + f'(w(t))\phi \  = - E - \NN (\phi)
  \quad  \hbox{in } \Gamma_\ve\times (0,\infty) ,\nonumber \\
 \phi(y,0)   = 0\foral  y \in \Gamma_\ve,  \nonumber\\
 \pp_t \phi(y,0) = -  \pp_t \Psi(\phi)(y,0)   \foral   y \in \Gamma_\ve , \label{gl11}
 \end{align}
where
\be
\NN(\phi) := \eta_{3}[\, [f'(u_1) -f'(u_0)]\, \phi +  B\phi+ N(\eta_1\phi +\Psi(\phi)) + (f'(u_1 )+1) \Psi(\phi)],
\label{NN}\ee
and as in \equ{error7},
$$
E\ :=\ \eta_3 \, S[u_1].
$$
Let us recall that we decomposed in \equ{error8}

$$
E \ =\  - (\Delta_\Gamma \hh + |A_\Gamma|^2 \hh )\, w'(t)
  + {\mathcal R}_3(\hh),
$$
where $ {\mathcal R}_3(\hh)$ is a small operator, satisfying  \equ{r30}, \equ{r31}.



\subsection{The projected problem}\label{s61}
Rather than solving Problem \ref{gl11} directly, we consider a  projected version of it, namely
the problem of finding $\phi$ and $\A$ such that
\begin{align}
 \pp^2_t\phi + \Delta_{\Gamma_\ve } \phi  + f'(w(t))\phi \,  =\,  \A(y)\, w'(t)\,
  - {\mathcal R}_3(\hh)  - \NN (\phi)\quad  \hbox{in } \Gamma_\ve\times (0,\infty) ,\nonumber \\
 \phi(y,0)   = 0\foral  y \in \Gamma_\ve,  \nonumber\\
 \pp_t \phi(y,0) = -  \pp_t \Psi(\phi)(y,0)   \foral   y \in \Gamma_\ve . \label{g113}
 \end{align}iformly on $h$, $h_1$, $h_2$ satisfying \equ{assh}.

\setcounter{equation}{0}
\section{Linearized problem in a half-space}\label{s6}

In order to solve Problem \equ{g113} we shall develop a uniform invertibility theory for the associated linear problem, so that
we later proceed just by contraction mapping principle to solve the nonlinear equation. This is also the procedure to prove the
results of Theorems \ref{teo5} and \ref{teo6} so that we consider a more general surface $\Gamma$ in Euclidean space $\R^{m+1}$, $m\ge 1$.
We consider in this section the (linear) problem in $\R^{m+1}$
of finding, for  given functions $g(y,t)$, $\beta(y)$, a solution $(\alpha, \phi)$ to the problem
\begin{align}
\Delta \phi \, + \, f'(w(t))\phi  \, = &\, \A(y)\,w'(t) + g(y,t) \quad \hbox{in } \R^{m+1}_+ ,
\nonumber\\
\phi(y,0)\, =\,&\ 0\ \  \ \foral y\in \R^m,
\nonumber\\
\pp_t \phi(y,0)\, =\, &\B(y) \foral y\in \R^m.
\label{linear2}\end{align}
Here  $\R^{m+1}_+ := \R^m\times (0,\infty)$.
We want to solve  Problem \equ{linear2}  using H\"older norms.
The main result of this section is the following

\begin{proposition}\label{lema5}
Given $\beta $ and $g$ such that
$$
\| \beta\|_{C^{1,\sigma}(\R^m)}\,+\, \| g\|_{C^{0,\sigma}(\R^{m+1}_+)}\,< \, +\infty
$$
there exists a   solution  $$(\phi,\A)\in  C^{2,\sigma}(\R^{m+1}_+)\times C^{0,\sigma}(\R^m) $$ of Problem $\equ{linear2}$ that  defines a linear operator of the pair
$(\beta,g)$,  satisfying  the estimate
\be
 \|\phi\|_{ C^{2,\sigma}(\R_+^{m+1})} + \|\A\|_{C^{0,\sigma}(\R^m)} \, \le\, C\,[\,\| \beta\|_{C^{1,\sigma}(\R^m)}\,+\, \| g\|_{C^{0,\sigma}(\R^{m+1}_+)}\,]\, .
\label{estt5}\ee
\end{proposition}

We will first solve this problem in an $L^2$ setting by means of Fourier transform and then use classical elliptic regularity theory to solve it in H\"older spaces.

\medskip
 We consider first the special case $g=0$, namely the problem of finding, for a given function $\beta(y)$ defined in
$\R^m$, functions $\alpha(y)$ and $\phi(y,t)$  that solve the problem

\begin{align}
\Delta \phi \, + \, f'(w(t))\phi  \, = &\, \A(y)\,w'(t) \quad \hbox{in } \R^{m+1}_+ ,
\nonumber\\
\phi(y,0)\, =\,&\ 0\ \  \ \foral y\in \R^m,
\nonumber\\
\pp_t \phi(y,0)\, =\, &\B(y) \foral y\in \R^m.
\label{linear}\end{align}

We will solve first this problem in $L^2$  by means of Fourier transform.
We have
the following
result.

\begin{lemma}\label{lema1}
Given $\beta\in H^1(\R^m)$, there exists a unique  solution  $$(\phi,\A)\in H^2(\R^{m+1}_+)\times L^2(\R^m) $$ of Problem $\equ{linear}$ that defines a linear operator of $\beta$. Besides, we have the estimate
\be
 \|\phi\|_{H^2(\R_+^{m+1})} + \|\A\|_{L^2(\R^m)} \, \le\, C\,\| \beta\|_{H^1(\R^m)}\, .
\label{estt}\ee

\end{lemma}

\proof

Let us assume for the moment that $b$ is a smooth, rapidly decaying function. The we can write Problem \equ{linear} in terms of Fourier transforms for its unknowns, $\h \phi(\xi,t)$, $\h \A (\xi)$ as
\begin{align}
 \pp_t^2\h \phi -|\xi|^2\h \phi\, + \, f'(w(t))\h \phi   \, = &\,  \h \A(\xi)\,w'(t) \quad \hbox{in } \R^m\times (0,\infty),
\nonumber\\
\h \phi(\xi,0)\, =\,&\ 0\ \  \ \foral \xi \in \R^m,
\nonumber\\
\pp_t \h \phi(\xi ,0) \,=\, & \h\B(\xi) \foral \xi \in \R^m.
\label{linearf}\end{align}

Let us consider first the ODE problem
\begin{equation}
\label{eq:p0}
  p_0'' (t)  + \, f'(w(t))\,p_0(t)  \, = \,  \,-w'(t) \quad \hbox{in }  (0,\infty),\ \
p_0(0) \, =\,\ 0\, .
\end{equation}
This equation has a unique bounded solution, given by
$$
p_0(t) \ =\ w'(t) \int_0^t \frac {d\tau}{{w'(\tau)}^2}\int_\tau^\infty w'(s)^2\, ds \ .
$$
In particular, $p_0'(0) =  w'(0)^{-1}\int_0^\infty w'(t)^2\, dt > 0$.
Similarly, for $\xi\ne 0$ the equation
\begin{align}
  p_\xi'' (t) -|\xi|^2p_\xi + \, f'(w(t))\,p_\xi(t)  \, = &\,  \,-w'(t) \quad \hbox{in }  (0,\infty),
\nonumber\\
p_\xi(0) \, =\,&\ 0\, .
\nonumber
\end{align}
has a unique bounded solution, which by maximum principle it is positive. Since $p_\xi''(0) = -w'(0)<0$,
we must have
$
p_\xi'(0) >0.
$
This last value define a smooth function of $\xi$. On the other hand for large values of $\xi$ we have that $p_\xi'(0) \approx q_\xi'(0)$, where $q_\xi$ solves

\begin{align}
  q_\xi'' (t) -|\xi|^2q_\xi(t)    \, = &\,  \,-w'(t) \quad \hbox{in }  (0,\infty),
\nonumber\\
q_\xi(0) \, =\,&\ 0\, .
\nonumber
\end{align}
Thus we have, for large values of $|\xi|$,
$$
p_\xi'(0) \approx q_\xi'(0) =   \int_0^\infty w'(t)\, e^{-|\xi| t} \, dt \approx  \frac {w'(0) }{ |\xi|}   $$
The solution of problem \equ{linearf} is then given by
$$
\hat \A (\xi) =  - \frac 1{p_\xi'(0) }\hat\B(\xi) ,\quad \hat\phi(\xi,t) = \hat \A (\xi)q_\xi(t).
$$
Observe in particular that
$$
\int_{\R^m} |\hat \A (\xi)|^2 \, d\xi \, \le \, c\int_{\R^m} (1+|\xi|^2)| \hat\B(\xi)|^2\, d\xi
$$
Now, from the fact that for some $a>0$  and any function $q\in H_0^1(0,\infty)$ we have that
\be
a\int_0^\infty  |q|^2\, dt  \le \int_0^\infty (\, |q'|^2 -f'(w)q^2)\, dt
\label{ee}\ee
we deduce that for some $C>0$, and any $\beta$,
$$
\int_0^\infty \int_{\R^m} [\, |\pp_t\h\phi|^2 +  (1+ |\xi|^2)\, |\hat\phi |^2\,d\xi\, dt\ \le\ C\,\int_{\R^m} (1+|\xi|^2)\,| \hat\B(\xi)|^2\, d\xi\ .
$$
Taking inverse Fourier transform of $\hat\phi (\xi, t)$ and $\h\A(\xi)$ we obtain a solution
$(\phi,\A)$ to the original problem \equ{linear}.
The above inequalities translate into
\be
\int_{\R^m} |\A|^2 \,dy +  \int_{\R^{m+1}_+ } [\,|\nn \phi|^2  + |\phi|^2\,]\, dy\,dt \ \le \ C \int_{\R^m} [\,|\nn \beta|^2  + |\beta|^2\, ]\,dy
\label{esti}\ee
By density, and the standard $L^2$ regularity theory, we get, given $\beta \in  H^1(\R^m)$ the  existence of a solution $(\phi,\A)\in H^2(\R^{m+1}_+)\times L^2(\R^m)$. This solution satisfies the desired estimate.

Finally, for uniqueness we need to prove that if $\beta=0$ in Problem \equ{linear}, then $\alpha$ and $\phi$ vanish identically. Taking Fourier transform we arrive to the family of ODE in $H^2(\R)$,

\begin{align}
 \pp_t^2\h \phi -|\xi|^2\h \phi\, + \, f'(w(t))\h \phi   \, = &\,  \h \A(\xi)\,w'(t) \quad \hbox{in } \R^m\times (0,\infty),
\nonumber\\
\h \phi(\xi,0)\, =\,&\ 0\ \  \ \foral \xi \in \R^m,
\nonumber\\
\pp_t \h \phi(\xi ,0) \,=\, & 0 \foral \xi \in \R^m.
\end{align}
For $\xi\ne 0$ we see that setting
$$
\phi =  p  - \frac {\h\A(\xi)} {|\xi|^2} w'(t),
$$
then as a function of $t$, $p$ satisfies
\begin{align}
  p'' -|\xi|^2 p \, + \, f'(w(t))p  \, = &\, 0 \quad \hbox{in }(0,\infty),
p'(0) \,=\, & 0 \foral \xi \in \R^m.
\end{align}
The function $p$ and its equation can be evenly extended to all $\R$. Then

$$
\int_\R  (|p'|^2 - f'(w)|p|^2)\, dt + \int_\R  |\xi|^2|p|^2  = 0
 $$
 Since the first integral above is always non-negative, we get that $p\equiv 0$. In particular this tells us that
  $$0= \phi(\xi,0)  = - \frac {\h\A(\xi)} {|\xi|^2} w'(0),  $$
and hence $\h\A(\xi) =0 $ almost everywhere in $\xi$. We conclude that $\A \equiv 0$.
The fact that $\hat \phi \equiv 0$ comes directly testing its  equation against $\phi$ itself. The proof is concluded.
\qed

\bigskip
 In order to solve Problem \equ{linear2}  and for later purposes, we also consider the  problem
\begin{align}
\Delta \phi \, + \, f'(w(t))\phi  \, = &\, \ g(y,t) \quad \hbox{in } \R^{m+1}_+ ,
\nonumber\\
\phi(y,0)\, =\,&\ 0\ \  \ \foral y\in \R^m,
\label{linear3}\end{align}

For the latter problem we have the following result.

\begin{lemma}\label{lema22}
Assume that $g\in L^2( \R^{m+1}_+)$. Then Problem \equ{linear3} has a unique solution $\phi\in H^2(\R_+^{m+1})$.
This solution satisfies in addition
\be
 \|\phi\|_{H^2(\R_+^{m+1})} \, \le\, C\,\| g\|_{L^2(\R^{m+1}_+)}\,
\label{estt2}\ee
and if $\phi \in H^1(\R^{m+1}_+)$,
\be
  \| \partial_t\phi (y,0) \|_{H^1(\R^m)}\,\le\, C\,\| g\|_{H^1(\R^{m+1}_+)}.
\label{estt3}\ee

\end{lemma}

\proof
As in the previous arguments, we consider the version of \equ{linear3} after Fourier transform,

\begin{align}
 \pp_t^2\h \phi -|\xi|^2\h \phi\, + \, f'(w(t))\h \phi   \, = &\,  \h g(\xi,t) \quad \hbox{in } \R^m\times (0,\infty),
\nonumber\\
\h \phi(\xi,0)\, =\,&\ 0\ \  \ \foral \xi \in \R^m.
\nonumber\end{align}
Using inequality \equ{ee}, we see that for each $\xi$ this problem can be solved uniquely  in such a way that
\be
\int_0^\infty [\, |\pp_t\h\phi|^2 +  (1+ |\xi|^2)\, |\hat\phi |^2 \,] \, dt\ \le\ C\,\int_0^\infty | \hat g|^2\,dt\,  .
\label{ii}\ee
so that
$$
\int_{\R^{m+1}_+} [\, |\pp_t\h\phi|^2 +  (1+ |\xi|^2)\, |\hat\phi |^2 \,] \,d\xi\, dt\ \le\ C\,\int_{\R^{m+1}_+} | \hat g|^2\, d\xi\,dt\,  .
$$
By taking Fourier transform back, and then using $L^2$ regularity we get a solution $\phi$
satisfying \equ{estt2}.

Now, testing equation \equ{linear3} against $e^{-|\xi|t}$ and setting $\hat\beta(\xi):= \pp_t\h\phi(\xi,0)$ we see that
$$
\h\beta(\xi) = - \int_0^\infty f'(w(t)) e^{-|\xi|t}\, \h\phi(\xi,t)\, dt +  \int_0^\infty f'(w(t)) e^{-|\xi|t}\, \h g(\xi,t)\, dt.
$$
Hence
$$
|\h\beta(\xi)|   \le  -  C \left( \int_0^\infty  e^{- 2|\xi|t}\,dt \right)^{\frac 12}
\left( \int_0^\infty [\, | \h\phi(\xi,t) |^2  +  \h g(\xi,t) |^2 \,dt \right)^{\frac 12}
$$
Using inequality \equ{ii} we then get that

$$
 (1+ |\xi|^2)|\h\beta(\xi)|^2 \ \le \  C \int_0^\infty  (1+ |\xi|)| \h g(\xi,t) |^2 \,dt.
$$
From here, estimate \equ{estt3} immediately follows. Observe that the control is in reality
stronger. In terms of fractional Sobolev spaces we have
$$
\|\beta\|_{H^1(\R^m)}\ \le\ C\|g\|_{H^{\frac 12} (\R^{m+1}_+)}
$$
but we will not need this.
The proof is concluded. \qed

\bigskip
Using Lemmas \ref{lema1}, \ref{lema22} and simple superposition we conclude the following result.

\begin{lemma}
\label{lema2}
Given $\beta\in H^1(\R^m)$, $g\in H^1(\R^{m+1}_+),$ there exists a  unique solution  $$(\phi,\A)\in H^2(\R^{m+1}_+)\times L^2(\R^m) $$ of Problem $\equ{linear2}$. This solution defines a linear operator of the pair
$(\beta,g)$, that satisfies  the estimate
\be
 \|\phi\|_{H^2(\R_+^{m+1})} + \|\A\|_{L^2(\R^m)} \, \le\, C\,[\,\| \beta\|_{H^1(\R^m)}\,+\, \| g\|_{H^1(\R^{m+1}_+)}\,]\, .
\label{estt222}\ee

\end{lemma}

\bigskip
We are interested solving Problem \equ{linear} for functions $\beta$ that are only locally in $H^1(\R^m)$ however in a uniform way, in the sense of the local uniform norms introduced in \equ{lu}.

We would like to solve Problem \equ{linear} for a $\beta$ with $\|\B \|_{H^1_{l.u.}(\R^m )}< +\infty$,
obtaining a linear operator with an estimate similar to \equ{estt} in its ``local uniform version''.
We have the following result.

\begin{lemma}\label{lema4}
Given $\beta\in H^1(\R^m)$, there exists a  solution  $$(\phi,\A)\in H_{loc}^2(\R^{m+1}_+)\times L_{loc}^2(\R^m) $$ of Problem $\equ{linear}$ that defines a linear operator of $\beta$. Besides, we have the estimate
\be
 \|\phi\|_{H^2_{l.u.}(\R_+^{m+1})} + \|\A\|_{L^2_{l.u.}(\R^m)} \, \le\, C\,\| \beta\|_{H_{l.u.}^1(\R^m)}\, .
\label{est6}\ee
\end{lemma}

\proof
For the moment, let us further assume that $\beta \in H^1(\R^m)$ and  consider
the solution $(\A,\phi)$ to Problem \equ{linear} predicted by Lemma \ref{lema1}. We will
prove that the a priori estimate \equ{estt} holds.

Let $p\in \R^m$ and  for small values $\delta$ consider the function
$$
\rho(y):=  \sqrt{1 + \delta^2 |y-p|^2}
$$

Let us write $\phi $ in the form

$$
\phi = \rho^\nu \ttt \phi.
$$
 Then Problem \equ{linear} becomes in terms of $\ttt \phi$

\begin{align}
\Delta \ttt \phi \, + \, f'(w(t))\ttt \phi  \, = &\,\ttt \A(y)\,w'(t) + B_\delta\ttt\phi \quad \hbox{in } \R^{m+1}_+ ,
\nonumber\\
\ttt\phi(y,0)\, =\,&\ 0\ \  \ \foral y\in \R^m,
\nonumber\\
\pp_t \ttt\phi(y,0)\, =\, & \ttt\B(y) \foral y\in \R^m.
\label{linear4}\end{align}
where,   $\ttt \beta =   \rho^{-\nu}\beta$,  $\ttt \A =   \rho^{-\nu}\A$
Here $B_\delta$ is a small linear operator of the form
$$
B_\delta \ttt \phi = O(\delta) \cdot \nn_y \ttt \phi  + O(\delta^2) \ttt \phi,
$$

We observe that for all small $\delta$,
$$
\| B_\delta \ttt \phi  \|_{H^1(\R^{m+1}_+)} \le C\delta \| \ttt \phi  \|_{H^2(\R^{m+1}_+)}
$$
where $C$ is independent of the point $p$ defining $\rho$.
By uniqueness of the $H^2$ solution in  Lemma \ref{lema2}, and
using the estimate \equ{estt222} we get, after fixing $\delta$ sufficiently small,

$$
 \| \ttt \alpha   \|_{L^2(\R^m)} + \| \ttt \phi  \|_{H^1(\R^{m+1}_+)} \le C\| \ttt \beta  \|_{H^1(\R^m)}.
$$

Now, if $\nu$ was chosen large, then
$$
\| \ttt \beta  \|_{H^1(\R^m)} \le C \| \beta  \|_{H_{l.u.}^1(\R^m)}
$$
while
$$
 \|\alpha   \|_{L^2( B(p ,1)} + \| \phi  \|_{H^2(B(p,1)}\ \le \  C[\, \| \ttt \alpha   \|_{L^2(\R^m)}+\| \ttt \phi  \|_{H^2(\R^{m+1}_+)}]
$$
where the  constants $C$ are uniform on the location of the origin $p$. Then  we get
\be
\| \alpha   \|_{L^2_{l.u.}(\R^m)} + \| \phi  \|_{H^2_{l.u.}(\R^{m+1}_+)} \le   C \| \beta  \|_{H_{l.u.}^1(\R^m)},
\label{eee} \ee
as desired.

Using this estimate on the solution of Lemma \ref{lema1}, we extend it by density to a solution of
Problem \equ{linear} satisfying \equ{eee} whenever
$\| \beta  \|_{H_{l.u.}^1(\R^m)} <+\infty$. The proof is concluded. \qed

Our next task is to estimate the solution thus found using local uniform H\"older norms.

\subsection{Proof of Proposition \ref{lema5}}

We argue first that it suffices to establish the above statement
replacing $ \| g\|_{C^{0,\sigma}(\R^{m+1}_+)}$ by the stronger norm $ \| g\|_{C^{1,\sigma}_0(\R^{m+1}_+)}$, namely finding a solution such that
\be
 \|\phi\|_{ C^{2,\sigma}(\R_+^{m+1})} + \|\A\|_{C^{0,\sigma}(\R^m)} \, \le\, C\,[\,\| \beta\|_{C^{1,\sigma}(\R^m)}\,+\, \| g\|_{C^{1,\sigma}(\R^{m+1}_+)}\,]\, .
\ee
To see this we let $(\alpha, \phi)$ solve Problem \equ{linear2} and write
$$
\phi = \ttt \phi  + \bar\phi
$$
where $\bar\phi$ is the unique bounded solution to the problem
\begin{align}
\Delta \bar \phi \, - \, \bar \phi  \, = &\, g \quad \hbox{in } \R^{m+1}_+ ,
\nonumber\\
\bar \phi(y,0)\, =\,&\ 0\ \  \ \foral y\in \R^m,
\nonumber
\end{align}
It is standard that $\bar\phi$ satisfies the estimate
$$
\|\bar\phi\|_{C^{2,\sigma}(\R^{m+1}_+)}\ \le \  \|g\|_{C^{0,\sigma}(\R^{m+1}_+)}.
$$
Problem \equ{linear2} can be written in the equivalent form

\begin{align}
\Delta \ttt \phi \, + \, f'(w(t))\ttt \phi  \, = &\,  \A(y)\,w'(t) +\ttt g \quad \hbox{in } \R^{m+1}_+ ,
\nonumber\\
\ttt \phi(y,0)\, =\,&\ 0\ \  \ \foral y\in \R^m,
\nonumber\\
\pp_t \phi(y,0)\, =\, &\ttt \B(y) \foral y\in \R^m.
\nonumber\end{align}
where
$$
\ttt g = -(1+f'(w)) \bar \phi , \quad \ttt\beta = \beta - \pp_t \bar \phi(y,0)
$$
so that
$$
\|\ttt g \|_{C^{1,\sigma}_0(\R^{m+1}_+)} +  \|\ttt \beta \|_{C^{1,\sigma}_0(\R^m)}\ \le \  C\,[\, \|g\|_{C^{0,\sigma}(\R^{m+1}_+)} +  \|\ttt \beta \|_{C^{1,\sigma}_0(\R^m)}\, ],
$$
and the claim follows.

The proof consists of proving the a priori estimate \equ{estt5} for the solution built in
Lemma \ref{lema4}.

Let us fix a point $p\in \R^m$.
We consider the unique solution of
the equation
$$ \Delta_y \gamma = \A \quad \hbox{in } B(p, 3) $$
$$ \gamma = 0 \quad \hbox{on } \pp B(p, 3). $$
Let us write, for $|y|< 3$
$$
\phi (y,t) =  w'(t)\, \gamma (y) + \psi(y,t)
$$
so that $\psi$ satisfies

 \begin{align}
\Delta \psi \, + \, f'(w(t))\psi  \, = &\, g \quad \hbox{in } B(p, 3)\times (0,\infty),
\nonumber\\
\pp_t \psi(y,0)\, =\, &\B(y) \foral y\in B(p, 3).
\label{linear7}\end{align}

Standard boundary regularity estimates for the Laplacian yield the estimate

\be
\|\psi\|_{ C^{2,\sigma}_0 ( B(p,1)\times (0,1)) } \ \le \ C\, [\, \|\beta\|_{ C^{1,\sigma} ( B(p,2))} + \|g\|_{ C^{0,\sigma}( B(p,2)\times (0,2))}  + \|\psi\|_{ H^2( B(p,2)\times (0,2))}\, ]\, .
\label{ppp}\ee
On the other hand, using Lemma \ref{lema4},
\begin{align}
 \|\gamma \|_{H^2( B(p,2))} \ &\  \le \
  C\,\| \A \|_{L^2(B(p,3)} \nonumber \\
  \ &\ \le \
  C\,[ \| \beta \|_{H_{l.u.}^1(\R^m)} +\| g \|_{H_{l.u.}^1(\R^{m+1}_+)}
  \nonumber \\
  \ &\ \le \   C\,[\, \|\beta\|_{ C^{1,\sigma} ( \R^m)} + \| g \|_{C^{1,\sigma}_0(\R^{m+1}_+)} \,],
\end{align}
while
\begin{align}
 \|\phi\|_{ H^2( B(p,2)\times (0,2))} \ &\  \le \
   C\,\| \phi \|_{H_{l.u.}^2(\R^{m+1}_+)} \nonumber \\
  \ &\ \le \
  C\,[\,\| \beta \|_{H_{l.u.}^1(\R^m)} + \| g \|_{H_{l.u.}^1(\R^{m+1}_+)}\,]
  \nonumber \\
  \ &\ \le \  C\,[ \| \beta \|_{C^{1,\sigma}_0(\R^m)} + \| g \|_{C^{1,\sigma}_0(\R^{m+1}_+)}\,]\,.
\end{align}
Since
$$
\|\psi\|_{ H^2( B(p,2)\times (0,2))}\, \le \, C\, [\, \|\gamma \|_{H^2( B(p,2))} + \|\phi\|_{ H^2( B(p,2)\times (0,2))}\, ]\,  $$
it follows that
$$
\|\psi\|_{ H^2( B(p,2)\times (0,2))}\, \le \, C\,[\,  \| \beta \|_{C^{1,\sigma}_0(\R^m)} + \| g \|_{C^{1,\sigma}_0(\R^{m+1}_+)}\,]\,,
$$
and hence from \equ{ppp}, we conclude
\be
\|\psi\|_{ C^{2,\sigma}_0 ( B(p,1)\times (0,1)) } \ \le \ C\, [\,  \| \beta \|_{C^{1,\sigma}_0(\R^m)} + \| g \|_{C^{1,\sigma}_0(\R^{m+1}_+)}\,]\, .
\label{p21}\ee
Since $\psi(y,0) = - w'(0)\gamma (y)$ for $y\in B(p,3)$ we find then that
$$
\|\A\|_{ C^{0,\sigma}( B(p,1) )} =  \|\Delta_y\gamma \|_{ C^{0,\sigma}( B(p,1) )} \, \le \,  C\,[\,  \| \beta \|_{C^{1,\sigma}_0(\R^m)} + \| g \|_{C^{1,\sigma}_0(\R^{m+1}_+)}\,]\,. $$
The constants $C$ accumulated above are all independent of the point $p\in \R^m$ chosen, therefore
\be
\|\A\|_{ C^{0,\sigma}( \R^m )}  \, \le \,  C\,  [\,  \| \beta \|_{C^{1,\sigma}_0(\R^m)} + \| g \|_{C^{1,\sigma}_0(\R^{m+1}_+)}\,]\,. \label{p2}\ee
Besides, the definition of $\psi$  yields
\be
\|\phi\|_{ C^{2,\sigma}_0 ( \R^m\times (0,1)) } \ \le \ C\,  [\,  \| \beta \|_{C^{1,\sigma}_0(\R^m)} + \| g \|_{C^{1,\sigma}_0(\R^{m+1}_+)}\,]\, .
\label{p3}\ee
Finally, estimate \equ{p2} and interior elliptic estimates for the equation satisfied by $\phi$
yield that for any $\tau>0$
\begin{align}
\|\phi\|_{ C^{2,\sigma}_0 ( B(p,1) \times (\tau + 1  , \tau +2 )) }  \ &\  \le \ C\,[\,
 \|\phi\|_{ H^2 ( B(p,3) \times (\tau   , \tau + 3 )) }\  \, +   \nonumber \\
 \ &\  \qquad\quad  \| g \|_{C^{0,\sigma}(B(p,3) \times (\tau   , \tau + 3 )) )} + \|\A\|_{ C^{0,\sigma}( B(p,3) )} \,]
 \nonumber \\
  \ &\ \le \  C\,[\, \|\phi\|_{H^2_{l.u.}(\R^m)} + \| g \|_{C^{1,\sigma}_0(\R^{m+1}_+)} + \|\beta\|_{ C^{1,\sigma} ( \R^m)}\,]\,
  \nonumber \\
  \ &\ \le \  C\, [\, \|\beta\|_{ C^{1,\sigma} ( \R^m)} + \| g \|_{C^{1,\sigma}_0(\R^{m+1}_+)} \,]\,
 \label{p4}\end{align}
where, again, $C$ is uniform on $p$ and $\tau$. Combining \equ{p3} and \equ{p4} we obtain
$$
 \|\phi \|_{ C^{2,\sigma}_0 ( \R^{m+1}_+)}\ \le \ C\,  [\, \|\beta\|_{ C^{1,\sigma} ( \R^m)} + \| g \|_{C^{1,\sigma}_0(\R^{m+1}_+)} \,]
 $$
and estimate \equ{estt5} has been established. The proof is concluded. \qed

\bigskip
We have the validity of a similar result for Problem \equ{linear3}

\begin{lemma}\label{lema6}
Given $g$ such that
$$
 \| g\|_{C^{0,\sigma}(\R^{m+1}_+)}\,< \, +\infty
$$
there exists a   solution  $\phi\in  C^{2,\sigma}(\R^{m+1}_+) $ of Problem $\equ{linear3}$ that  defines a linear operator of $g$,  satisfying  the estimate
\be
 \|\phi\|_{ C^{2,\sigma}(\R_+^{m+1})}  \, \le\, C\, \| g\|_{C^{0,\sigma}(\R^{m+1}_+)}\, .
\label{es5}\ee
\end{lemma}
\proof
The proof goes along the same lines as that of Lemma \ref{lema5}, being actually easier.
A solution satisfying
the estimate
\be
 \|\phi\|_{H^2_{l.u.}(\R_+^{m+1})}  \, \le\, C\,\| g\|_{L_{l.u.}^2(\R^{m+1}_+)}\, .
\label{est66}\ee
is found with the argument in Lemma \ref{lema4}, using the result of Lemma \ref{lema22}. Estimate
\equ{es5} for this solution then follows right away from local boundary and interior elliptic estimates. \qed

\setcounter{equation}{0}
\section{Linearized problem in the half-cylinder} 

Let $\Gamma$ be a smooth, complete, embedded manifold in $\R^{m+1}$ that separates the space into two components.
 For each point $p\in \Gamma$ we assume that we can find a parametrization
$$
Y_p : B(0,1)\subset \R^m \mapsto \Gamma\subset \R^{m+1}
$$
onto a neighborhood $\UU_p$ of $p$ in $\Gamma$, so that if we write
$$ g_{ij}(\py) \ := \ \left < \pp_i Y_p,\pp_j Y_p\right> = \delta_{ij} + \theta_p(\py) $$
we may assume that
$\theta_p$ is smooth with $\theta_p(0) = 0$ and with second order derivatives  bounded in $B(0,1)$, uniformly in $p$.

\medskip
In this section we want to extend the linear theory developed so far to the same problem
considered in the region $\Gamma_\ve \times (0,\infty)$. Thus we consider the problem
 of finding, for  given functions $g(y,t)$, $\beta(y)$, a solution $(\alpha, \phi)$ to the problem
\begin{align}
\pp^2_t\phi + \Delta_{\Gamma_\ve} \phi \, + \, f'(w(t))\phi  \, = &\, \A(y)\,w'(t) + g(y,t) \foral  (y,t)\in \Gamma_\ve \times (0,\infty),
\nonumber\\
\phi(y,0)\, =\,&\ 0\ \  \ \foral y\in \Gamma_\ve ,
\nonumber\\
\pp_t \phi(y,0)\, =\, &\B(y) \foral y\in  \Gamma_\ve .
\label{linear22}\end{align}

We also consider the problem

\begin{align}
\pp^2_t\phi + \Delta_{\Gamma_\ve} \phi \, + \, f'(w(t))\phi  \, = &\,  g(y,t) \foral  (y,t)\in \Gamma_\ve \times (0,\infty),
\nonumber\\
\phi(y,0)\, =\,&\ 0\ \ \quad \ \foral y\in \Gamma_\ve .
\label{linear333}\end{align}

For an open subset $\Lambda$ of a manifold embedded in $\R^N$,   we call $C^{k,\sigma}(\Lambda)$ the Banach space of functions $h\in C^{k,\sigma}_0(\Lambda)$ for which
$$\|h\|_{C^{k,\sigma}(\Lambda)} < + \infty. $$

We prove first the following result.

\begin{lemma} \label{prop1}
For all sufficiently small $\ve$ the following statement holds:
given $$(\beta, g) \in  C ^{1,\sigma}(\Gamma_\ve )\times C ^{0,\sigma}(\Gamma_\ve\times (0,\infty)) $$ such that
$$
\| \beta\|_{C^{1,\sigma}(\Gamma_\ve )}\,+\, \| g\|_{C^{0,\sigma} (\Gamma_\ve\times (0,\infty))}\,< \, +\infty
$$
there exists a   solution  $(\phi,\A) = \TT (\beta, g)$ of Problem $\equ{linear22}$, linear in its argument,
that satisfies the estimate
\be
 \|\phi\|_{ C^{2,\sigma} (\Gamma_\ve\times (0,\infty))} + \|\A\|_{C^{0,\sigma}(\Gamma_\ve )} \, \le\, C\,[\,  \| \beta\|_{C^{1,\sigma}(\Gamma_\ve )}\,+\, \| g\|_{C^{0,\sigma}} (\Gamma_\ve\times (0,\infty)) \,]\,
\ee
where $C$ is independent of $\ve$.
\end{lemma}

\proof
We shall construct a solution by gluing solutions built through Lemma \ref{lema5} to problems of the form
\equ{linear2}, associated to local Euclidean
parametrizations of $\Gamma_\ve$.

The local coordinates of the surface $\Gamma$ induce corresponding ones for the neighborhood $\ve^{-1}\UU_p$ of the point $p_\ve= \ve^{-1}p$ in $\Gamma_\ve$, by means of the map
$$
\py\in B(0, \ve^{-1})\,   \longmapsto \, \ve^{-1} Y_p(\ve \py) \in \ve^{-1}\UU_p .
$$
The Laplace-Beltrami operator is represented in these coordinates by
$$
\Delta_{\Gamma_\ve} = \frac 1{\sqrt{\det g(\ve \py) }}\, \pp_{i}\left ( \sqrt{\det g(\ve \py) } \, g^{ij}(\ve\py)\, \pp_j\,\right)
$$
where the $g^{ij}$ represent the coefficients of the inverse matrix of $g_{ij}. $
Then we can expand
\be
 \Delta_{\Gamma_\ve} =  \Delta_y + B_{p_\ve},  \quad B_{p_\ve} \,:=\,  b_{ij}(\ve \py)\,\pp_{ij} +  \ve b_i(\ve \py)\,\pp_i, \quad |\py|< \ve^{-1},
 \label{lapbel}\ee
where  the coefficients $b_{ij}$, $b_i$ have  derivatives bounded,  uniformly in $p$, and  $b_{ij}(0)=0$.
Observe in particular that $|b_{ij}(\ve \py)| \le C\ve|\py|$ with $C$ uniform in $p$, so that for $\delta>0$ small but fixed we have
$$
|b_{ij}| \le C\delta ,\quad  |D_{\py}b_{ij}| + |D_{\py}b_{j}| \le C\ve ,\quad \py\in B(0,\delta\ve^{-1}),
$$
so that the coefficients are uniformly small with $\delta$ as $\ve\to 0$, in other words
$\Delta_{\Gamma_\ve}$ differs from Euclidean Laplacian by an operator $\delta$-small, uniformly in $p$.

We fix a small number $\delta$ and choose a sequence of points $p_j$ such that
$\Gamma$ is covered by the union of the open sets \be \UU_k:= Y_{p_k} (B(0, \delta /2 ))\label{uj}\ee
 and so that each $\UU_j$ does not intersect more than a finite, uniform number of $\UU_\ell$ with $\ell\ne j$.
Let us consider a smooth cut-off function $\eta$, with $\eta(s) =1 $ for $s<1$, $=0$ for $s>2$.
We define on $\Gamma_\ve$ the smooth cut-off functions,
$$
\eta_{km} (y) : =  \eta( \ve \, |\py| / m\delta ) , \quad y=  \ve^{-1}Y_{p_k}(\ve \py), $$
extended as zero to all $\Gamma_{\ve}$ outside their supports.

We look for a solution to Problem \equ{linear22}  with the following form.

\be
\phi = \sum_{k=1}^\infty \eta_{k1}\phi_k   + \psi, \quad \A = \sum_{j=1}^\infty \eta_{k1}\A_k
\label{form}\ee
where $\phi_k$ is defined, for instance,  on $\UU_k \times (0,\infty)$ with
$$
\UU_k := \ve^{-1} Y_{p_k} ( B(0, 2\delta)),
$$
and the function $\eta_{k1}\phi_k $ is extended by zero outside its support.
Using Einstein's summation convention, equation \equ{linear22} can be written as
\begin{align}
 \eta_{k1} [\, (\pp^2_t + \Delta_{\Gamma_\ve}) \phi_k + f'(w)\phi_k ]  +  2\nn_{\Gamma_\ve}\phi_k \cdot\nn_{\Gamma_\ve}\eta_{k1} \ + \sum_{k=1}^\infty \eta_{k1}\A_kw'
 \nonumber \\
 \phi_k \Delta_{\Gamma_\ve}\eta_{k1}  -h +  (\pp^2_t + \Delta_{\Gamma_\ve})  \psi  +   f'(w) \psi\  =0.
 \nonumber\\
 \eta_{k1} \phi_k(y,0) + \psi(y,0)\, = 0,  \nonumber\\
 \eta_{k1} \pp_t \phi_k(y,0) + \pp_t \psi(y,0)\, =\, \B(y)\nonumber\\ \foral y\in  \Gamma_\ve .
 \end{align}
We separate further the above equation as
\begin{align}
 \eta_{k1} [\,(\pp^2_t + \Delta_{\Gamma_\ve}) \phi_k + f'(w)\phi_k \,]\,  + \,    (\sum_{k=1}^\infty \eta_{k1}) (f'(w)+1) \psi \, +\, \sum_{k=1}^\infty \eta_{k1}\A_k \ +
 \nonumber \\
(\pp^2_t + \Delta_{\Gamma_\ve}) \psi  - (\sum_{k=1}^\infty \eta_{k1})\psi  +  2\nn_{\Gamma_\ve}\phi_k \cdot\nn_{\Gamma_\ve}\eta_{k1} +  \phi_k \Delta_{\Gamma_\ve}\eta_{k1}  -h \, .
\end{align}

Since the sets $\UU_k$ cover $\Gamma$ and for each $j$, there is at most a uniformly bounded number of $\ell \ne j$ is such that $\UU_k\cap \UU_\ell \ne \emptyset,$ it follows that for some constant $C$ uniform in $\ve$.
\be
1\le  V := \sum_{k=1}^\infty \eta_{k1} \le C  .
\label{V}\ee

Now, defining
\be
\beta_k  \ :=\    \frac { \eta_{k2}\beta} { \sum_{\ell =1}^\infty  \eta_{\ell 1} } ,
\label{betak}\ee
and using that $\eta_{k1}\eta_{k2} = \eta_{k1}$ we get that
$$
\beta = \sum_{k=1}^\infty \eta_{k1} \beta_k, $$
Then Equation \equ{pl} will hold if  we have the following infinite system of equations satisfied.

\begin{align}
 \pp^2_t\phi_k + \Delta_{\Gamma_\ve} \phi_k + f'(w)\phi_k   =  \A_k w' -  (f'(w)+1) \psi\quad\hbox{in } \UU_k \times (0,\infty)
 \nonumber \\
 \phi_k(y,0)  = 0,  \nonumber\\
 \pp_t \phi_k(y,0) = - \pp_t \psi(y,0) + \B_k (y)  \nonumber\\ \foral y\in  \UU_k ,
\quad  k=1,2,\ldots ,   \label{g1} \\
\pp^2_t\psi  + \Delta_{\Gamma_\ve} \psi  - V(y)\, \psi\,  =\,
- \sum_{k=1}^\infty [\,  2\nn_{\Gamma_\ve}\phi_k \cdot\nn_{\Gamma_\ve}\eta_{k1} +  \phi_k \Delta_{\Gamma_\ve}\eta_{k1} \, ] + g \,  \nonumber\\ \hbox{in } \UU_k \times (0,\infty) \,
\nonumber \\
 \psi(y,0)  = 0 \foral y\in  \Gamma_\ve . \nonumber\\
 \label{g2} \end{align}

Using local coordinates $y = \ve^{-1}Y_{p_k} (\ve \py)$ in $\UU_k$, we get equation
\equ{g1} expressed as

\begin{align}
 \pp^2_t\phi_k + \Delta_{\py} \phi_k +  B_{p_k}\phi_k + f'(w)\phi_k  & =  \A_k w' -            (f'(w)+1) \psi,\nonumber\\  & \foral (\py,t)\in B(0, 2\delta\ve^{-1}) \times (0,\infty),\nonumber \\
 \phi_k(\py,0)  & = 0,  \quad  \py\in B(0, 2\delta\ve^{-1}),  \nonumber\\
 \pp_t \phi_k(\py,0) &= - \pp_t \psi(\py,0) + \B_k (\py), \quad  \py\in B(0, 2\delta\ve^{-1}) . \label{g121} \end{align}
where $B_{p_k}$ is the small operator in \equ{lapbel}, and by slight abuse of notation we denote
in the same way $h(\py)$ and $h(y)$, a function $h$ defined on $ \ve^{-1}\ttt \UU_k$ evaluated at the point $y= \ve^{-1}Y_k(\ve \py)$. Equation \equ{g121} can be extended to all of $\R^9_+$, with all its coefficients  well-defined. Now, $\eta_{k3}(\py) = \eta( \ve |\py|/ 3\delta)$, and we will have a solution of  \equ{g1}-\equ{g2} if we solve the system

\begin{align}
 \pp^2_t\phi_k + \Delta_{\py} \phi_k +  \eta_{k3} B_{p_k}\phi_k + f'(w)\phi_k  & =  \A_k w' -           \eta_{k3}(f'(w)+1) \psi,\quad  \hbox{in } \R^9_+ ,\nonumber \\
 \phi_k(\py,0)  & = 0\foral  \py\in \R^m,  \nonumber\\
 \pp_t \phi_k(\py,0) &= -  \eta_{k3}\pp_t \psi(\py,0) + \B_k (\py)  \foral  \py\in \R^m , \label{gg1}
 \end{align}
 \begin{align}
\pp^2_t\psi  + \Delta_{\Gamma_\ve} \psi  - V(y)\, \psi\,  =\,
- \sum_{j=1}^\infty [\,  2\nn_{\Gamma_\ve}\phi_k \cdot\nn_{\Gamma_\ve}\eta_{k1} +  \phi_k \Delta_{\Gamma_\ve}\eta_{k1} \, ] + g \,  \nonumber\\ \hbox{in } \Gamma_\ve \times (0,\infty) \,
\nonumber \\
 \psi(y,0)  = 0 \foral y\in  \Gamma_\ve . \nonumber\\
 \label{gg2} \end{align}

We will solve first equation \equ{gg2} for given $\phi_k$'s and $g$.

Let us consider  the Banach space
$$\ell^\infty ( C^{m,\sigma}(\Lambda) ) $$
 of  bounded sequences
in $ C^{m,\sigma}(\Lambda)$, endowed with the norm
$$
\|\hh \|_{\ell^\infty ( C^{m,\sigma}(\Lambda) )} := \sup_{k\ge 1} \| h_k\|_{C^{m,\sigma}(\Lambda)}  .
$$

Let
$$
\Phi = (\phi_k)_{k\ge 1} \in  \ell^\infty ( C^{2,\sigma}(\R^9_+) ), \quad g\in  C^{0,\sigma}(\Gamma_\ve\times (0,\infty)) $$
be given.
We write equation \equ{gg2} as

 \begin{align}
\pp^2_t\psi  + \Delta_{\Gamma_\ve} \psi  - V(y)\, \psi\,  =\, \ttt g
\,\quad \hbox{in } \Gamma_\ve \times (0,\infty) \,
\nonumber \\
 \psi(y,0)  = 0 \foral y\in  \Gamma_\ve . \nonumber\\
 \label{gg3} \end{align}
We see that $\ttt g$ defines a linear operator of the pair $(\phi,g)$ and
$$
\|\ttt g \|_{C^{2,\sigma}(\Gamma_\ve \times (0,\infty) )} \ \le\ C\,[\, \| g \|_{C^{0,\sigma}(\Gamma_\ve \times (0,\infty) )} +
\ve \|\phi  \|_{\ell^\infty ( C^{2,\sigma}( \R^9_+) )}
\,]
$$

where the constant $C$ is uniform in all small $\ve$.  Here we have used
 the fact that
$$|\nn _{\Gamma_\ve}\eta_{k1} | + |\Delta _{\Gamma_\ve}\eta_{k1} | \le C\ve .$$
Now, since $1\le V\le C$, the use of barriers yields the existence of a unique bounded solution $\psi$ to
\equ{gg3}, which satisfies
$$
\|\psi\|_{L^\infty(  \Gamma_\ve \times (0,\infty) )} \ \le \ C\|\ttt g\|_{L^\infty(  \Gamma_\ve \times (0,\infty) )}.
$$
Then the use of local interior and boundary Schauder estimates, invoking the representation
of the equation in local coordinates and the uniform H\"older character of the coefficients, yields
that $\psi = \Psi(\Phi,g)$ satisfies
  \be
\|\Psi(\Phi,g)\|_{ C^{2,\sigma}(  \Gamma_\ve \times (0,\infty) )} \ \le \  C\,[\, \| g \|_{C^{0,\sigma}(\Gamma_\ve \times (0,\infty) )} +
\ve \|\Phi  \|_{\ell^\infty ( C^{2,\sigma}( \R^9_+) )}
\,]\ .
\label{s9n}\ee

Next we consider in addition $\beta \in  C ^{1,\sigma}(  \Gamma_\ve )$ be given
and define $\beta_k$ as in \equ{betak}.
Then $ \mathfrak b:= (\beta_k)_{k\ge 1}$ satisfies
$$
\| \mathfrak b \|_{\ell^\infty ( C^{1,\sigma}( \R^m ) )}\ \le C\ \|\beta\|_{ C^{1,\sigma}(  \Gamma_\ve ) }.
$$

Let $$T: \ell^\infty ( C^{1,\sigma}( \R^m ) )\times \ell^\infty ( C^{0,\sigma}( \R^9_+ ) ) \ \to\  \ell^\infty ( C^{2,\sigma}( \R^9_+ ) )$$
$$
(\beta, g) \ \longmapsto \  T(\beta, g) := \phi
$$
be the linear operator built in Lemma \ref{lema5} as a solution of Problem \equ{linear2}, so that
\be
\|T(\beta, g) \|_{C^{2,\sigma}(\R^9_+)} \ \le\ [\, \| g \|_{C^{0,\sigma}(\R^9_+)} + \| \beta \|_{C^{1,\sigma}(\R^m)}\, ]\, .
\label{oo}\ee
Then the equation \equ{gg1} after substituting $\psi$ by $\Psi(\phi, g)$ becomes
\begin{align}
 \pp^2_t\phi_k + \Delta_{\py} \phi_k  + f'(w)\phi_k  & =  \A_k w'  + \ttt g_k ,\quad  \hbox{in } \R^9_+ ,\nonumber \\
 \phi_k(\py,0)  & = 0\foral  \py\in \R^m,  \nonumber\\
 \pp_t \phi_k(\py,0) &= \ttt \beta_k(\py)  \foral  \py\in \R^m ,
 \end{align}
where
$$
 \ttt g_k := - \eta_{k3} B_{p_k}\phi_k- \eta_{k3} (f'(w)+1) \Psi(\phi, g), \quad
\ttt \beta_k (\py) := - \eta_{k3} \pp_t \Psi(\phi, g)\,(\py,0) \, +\,  \B_k (\py)
$$
so that we find a solution if we solve the linear fixed point problem

\be
\Phi \ = \ {\mathfrak A} (\Phi) + {\mathfrak g} ,\quad \Phi \in
  \ell^\infty ( C^{2,\sigma}( \R^9_+ ))
\label{fp}\ee
where

$$ {\mathfrak A} (\Phi)_k\ :=\ T(- \chi \pp_t \Psi(\Phi, 0)\,(\cdot,0)  \, , \,
 - \eta_{k3} B_{p_k}\phi_k- \eta_{k3} (f'(w)+1) \Psi(\Phi, 0) \, ), \quad
$$
$$
{\mathfrak g}_k : =  T(- \eta_{k3} \pp_t \Psi(0, g)\,(\cdot,0) \, +\,   \B_k  \, , \,
- \eta_{k3} (f'(w)+1) \Psi(0, g) \, )
$$
From estimates \equ{s9n}, \equ{oo} and the $\delta$-smallness of the operator $B_{p_k}$
we readily get that for all small $\ve$
$$
\|{\mathfrak A} (\Phi)\|_{\ell^\infty (C^{2,\sigma}( \R^9_+ ))} \le C\delta  \|\Phi\|_{\ell^\infty (C^{2,\sigma}( \R^9_+ ))},
$$
and also
$$
\|{\mathfrak g} \|_{\ell^\infty (C^{2,\sigma}( \R^9_+ ))}\le C\ [\, \| g \|_{C^{0,\sigma}(\Gamma_\ve \times (0,\infty))} + \| \beta \|_{C^{1,\sigma}(\Gamma_\ve)}\, ]\, .
$$
It follows that if $\delta$ is fixed sufficiently small, then for all small $\ve$, Problem \equ{fp}
has a unique solution $\Phi = \Phi (\beta, g)$ which is linear in its argument and satisfies the estimate
\be
\|\Phi (\beta, g) \|_{\ell^\infty (C^{2,\sigma}( \R^9_+ ))}\le C\ [\, \| g \|_{C^{0,\sigma}(\Gamma_\ve \times (0,\infty))} + \| \beta \|_{C^{1,\sigma}(\Gamma_\ve)}\, ]\,
\ee
It is straightforward to check that the solution $(\phi,\A)$ thus obtained by formula \equ{form}, satisfies the
desired properties. The proof is concluded. \qed

\bigskip
We consider next the case in which  $g$ and $\beta$
have a uniformly weighted H\"older control.

Let us consider a positive function $\rho(y,t)$ defined on $\Gamma_\ve\times (0,\infty)$ which we
assume of class $C^{2,\sigma}$.
Let us write
$$
\phi = \rho \ttt \phi
$$
and consider Problem \equ{linear22} written in terms of $\ttt \phi$. We have

\begin{align}
\pp^2_t\ttt \phi + \Delta_{\Gamma_\ve} \ttt \phi \, +  B\ttt\phi + \, f'(w) \ttt\phi  \, = &\, \ttt\A(y)\,w'(t) +\ttt g(y,t) \foral  (y,t)\in \Gamma_\ve \times (0,\infty),
\nonumber\\
\ttt\phi(y,0)\, =\,&\ 0\ \  \ \foral y\in \Gamma_\ve ,
\nonumber\\
\pp_t \ttt\phi(y,0)  \, =\, &  \ttt\B(y) \foral y\in  \Gamma_\ve .
\label{linear222}\end{align}
where
$$
B\ttt\phi =    \rho^{-1}[  \pp^2_t\rho  + \Delta_{\Gamma_\ve}\rho  ]\,\ttt \phi +  2\rho^{-1}\,[ \nn_{\Gamma_\ve}\rho\cdot\nn_{\Gamma_\ve}\ttt\phi + \pp_t\rho\,\pp_t\ttt\phi ] \ ,
$$
$$
\ttt \A = \rho^{-1}\A, \quad \ttt g = \rho^{-1} g, \quad \ttt\beta = \rho^{-1}\beta\ .
$$
If we have that
\be
\| B\ttt\phi\|_{C^{0,\sigma}(\Gamma_\ve \times (0,\infty) } \le c \|\ttt\phi\|_{C^{2,\sigma} ( \Gamma_\ve \times (0,\infty) )}
\label{small}\ee
for a $c$ sufficiently small, then there is a small linear perturbation of the operator built in
Lemma \ref{prop1} which solves problem \equ{linear22} whenever
$$
\| \ttt\beta\|_{C^{1,\sigma}(\Gamma_\ve ) }  + \|\ttt g \|_{C^{0,\sigma}(\Gamma_\ve \times (0,\infty) }< +\infty .
$$
and a  corresponding estimate is obtained for $\ttt\phi = \rho^{-1}\phi$
Condition \equ{small} will be achieved provided that
 \be \|\rho^{-1} D^2\rho\|_{C^{0,\sigma}( \Gamma_\ve \times (0,\infty) )} + \|\rho^{-1} D\rho\|_{C^{0,\sigma}( \Gamma_\ve \times (0,\infty) )} \ \le \ c,
\label{sma}\ee
with $c$ sufficiently small.
Under this condition we have obtained a solution $\phi$ to Problem \equ{linear22} such that
\begin{align}
\|\rho^{-1} D^2\phi\|_{C^{0,\sigma}( \Gamma_\ve \times (0,\infty) )} +
\|\rho^{-1} D \phi\|_{C^{0,\sigma}( \Gamma_\ve \times (0,\infty) )} + \|\rho^{-1}\phi\|_{C^{0,\sigma}( \Gamma_\ve \times (0,\infty) )} \nonumber\\
\,\le \,  C\, \|\ttt\phi\|_{C^{2,\sigma}  ( \Gamma_\ve \times (0,\infty) )} \, \le \,\nonumber\\
  C\,[\,\|\rho^{-1} D \beta\|_{C^{0,\sigma}( \Gamma_\ve )}+ \|\rho^{-1} \beta\|_{C^{0,\sigma}( \Gamma_\ve )} + \|\rho^{-1}g \|_{C^{0,\sigma}( \Gamma_\ve \times (0,\infty) )}\,]\, .
\end{align}

For our purposes, the following weight plays a very important role.
For $y\in \Gamma$ we denote
$\rr(y',y_9) = \sqrt{ 1+ |y'|^2}$ and we write
$$
\rr_\ve (y) := \rr(\ve y) ,\quad y\in \Gamma_\ve .
$$
For positive numbers $\nu, \gamma$ let us consider

$$
\rho(y,t)\ = \  \rr_\ve ( y )^{-\nu} e^{-\gamma  t},\quad  (y,t)\in \Gamma_\ve \times (0,\infty).
$$
Then we observe that if $\gamma$ is fixed sufficiently small, then for any fixed $\nu\ge 0$ and all small $\ve$ we have the validity
of condition \equ{sma}.

Let us consider the weighted H\"older norms defined in \equ{norm1}-\equ{norm2}

Then the following result has been obtained.

\begin{proposition} \label{prop2}
If $\gamma\ge 0$ is fixed sufficiently small and $\nu\ge 0$ is arbitrary,
then for all sufficiently small $\ve$ the following statement holds:
given $(\beta, g) $ such that
$$
\|\beta \|_{C_\nu^{1,\sigma}(\Gamma_\ve  )} +
\|g\|_{C_{\nu,\gamma}^{0,\sigma}(\Gamma_\ve\times (0,\infty))} \ <\ +\infty
$$
there exists a   solution  $\phi= \TT (\beta, g) $
of Problem $\equ{linear22}$, linear in its argument,
that satisfies the estimate
\begin{align}
\|\phi \|_{C_{\nu,\gamma}^{2,\sigma}(\Gamma_\ve\times (0,\infty))}
\ \le\ C\,[\, \|\beta \|_{C_\nu^{1,\sigma}(\Gamma_\ve  )} +
\|g\|_{C_{\nu,\gamma}^{0,\sigma}(\Gamma_\ve\times (0,\infty))}\, ]\,
\label{esti5}\end{align}
where $C$ is independent of $\ve$.
\end{proposition}

\setcounter{equation}{0}
\section{Solving the projected problems} \label{s7}
In terms of the operator $\TT(\beta, g)$ defined in Proposition \ref{prop2}, we will have a solution to \equ{g113}
if we solve the fixed point problem

\be
\phi \ =\  - \TT (\, \pp_t \Psi(\phi)(y,0)\, ,\,  {\mathcal R}_3(\hh)+ \NN (\phi)\, ), \quad \phi\in C^{2,\sigma}_{\nu, \gamma} (\Gamma_\ve \times (0,\infty)).
\label{fpp1}\ee

Taking into account that for $\nu = 4+\mu$ we have

$$  \|\pp_t \Psi(0)(y,0)\|_{
C^{1,\sigma}_{\nu, \gamma} (\Gamma_\ve )}+ \| {\mathcal R}_3(\hh) + \NN (0)\|_{
C^{0,\sigma}_{\nu, \gamma} (\Gamma \times (0,\infty))}   \ \le \ C\ve^4   $$
and the fact, straightforward to check, that the operator on the right hand side of equation \equ{fpp1},
defines a contraction mapping on the  set of functions
$\phi$ with
$$\| \phi \|_{C^{2,\sigma}_{\nu, \gamma} (\Gamma \times (0,\infty))} \le M\ve^3, $$
then fixing $M$ sufficiently large, contraction mapping principle provides a unique solution of
\equ{fpp1} in that region.  Using the Lipschitz property \equ{r31} for ${\mathcal R}_3(\hh)$,
corresponding properties are inherited for the solution.
The following result holds.

\begin{lemma}\label{lemapoto}
For all small $\ve $ sufficiently small the following holds. There exists a solution $\phi = \Phi( \hh)$ of problem \equ{gl1} that defines
an operator on functions  $\hh$ satisfying constraints
\equ{assh}. For such functions and some $\mu>0$ we have
$$
\| \Phi(\hh^1) - \Phi(\hh^2) \|_{C^{2,\sigma}_{4+\mu,\gamma} (\Gamma_\ve\times (0,\infty))} \, \le\, C\,\ve^3 \|\hh_1 - \hh_2 \|_{2,4+\mu, \Gamma}
$$
$$
\| \Phi(\hh) \|_{C^{2,\sigma}_{4+\mu,\gamma} (\Gamma_\ve\times (0,\infty))} \le C\ve^4
$$
and
$$
\| \NN(\Phi(\hh^1)) - \NN(\Phi(\hh^2)) \|_{C^{0,\sigma}_{4+\mu,\gamma} (\Gamma_\ve\times (0,\infty))}\, \le\, C\,\ve^4\, \|\hh^1 - \hh^2 \|_{2,2+\mu, \Gamma}
$$
for all  $\hh$, $\hh^1$, $\hh^2$ satisfying \equ{assh}.
\end{lemma}

\subsection{The reduced problem: proof of Theorem \ref{teo4}}
We are ready to solve the full problem. Let us consider the solution $\Phi(\hh)$ to \equ{gl1}
predicted by Lemma \ref{lemapoto}, and call $\A[\hh]$ the corresponding $\A$. Then
\begin{align}
 \pp^2_t \Phi(\hh) + \Delta_{\Gamma_\ve } \Phi(\hh)  + f'(w(t))\Phi(\hh) \,  =\,  \A[\hh]\, w'(t)\,
  - {\mathcal R}_3(\hh)  - \NN (\Phi(\hh))\quad  \hbox{in } \Gamma_\ve\times (0,\infty) ,\nonumber \\
 \phi(y,0)   = 0\foral  y \in \Gamma_\ve,  \nonumber\\
 \pp_t \Phi(\hh)(y,0) = -  \pp_t \Psi(\Phi(\hh))(y,0)   \foral   y \in \Gamma_\ve . \label{gpoto}
 \end{align}

We can express the parameter function $\A[\hh]$ for the solution predicted of Problem \equ{g113}, as an operator in $\hh$, by integration against $w'$, for $y\in \Gamma_\ve$,
\be
 \A [\hh] (y)\, \int_0^\infty {w'}^2\, dt \ =\  \int_0^\infty [ \Delta_{\Gamma_\ve} \Phi(\hh) \, +\, {\mathcal R}_3(\hh)  + \NN (\Phi(\hh))]\, w'\, .  \label{alpha}\ee
Then we will have solved our original problem if we find a solution $\hh$ within constraints \equ{assh}
of the equation
\be
\Delta_{\Gamma} \hh + |A_\Gamma(y)|^2\hh =  \mathcal B  [\hh]\, (y)  \inn \Gamma .
\label{finj}\ee
where
$$\mathcal B  [h]\, (y) \ := \ \ve^{-2} \A[\hh]\, (\ve^{-1} y), \quad y\in \Gamma.
$$
and $\A$ is the function in \equ{alpha}. We see then that
$$
\|\A [\hh]\|_{C^{0,\sigma}_{4+\mu} (\Gamma_\ve)} \ \le\
C\,
\ve^4\ ,
$$
$$
\|\A [\hh_1] -\A[\hh_2]\|_{C^{0,\sigma}_{4+\mu} (\Gamma_\ve)} \ \le\
C\,
\ve^3\|\hh_1 -\hh_2\|_{C^{2,\sigma}_{4+\mu} (\Gamma_\ve)}\ .
$$
Using Lemma \ref{lema8} we then get
\begin{align}
\|\mathcal B  [\hh ]\|_{\sigma,4 +\mu -\sigma, \Gamma}   \ & \le \ C\ve^{-\sigma -2 }
\| \A[\hh]  \|_{C^{0,\sigma}_{4+\mu} (\Gamma_\ve)} \  \nonumber\\ &\le \   C\, \ve^{ 2-\sigma },
\label{B1}\\
\|\mathcal B  [\hh^1]-\mathcal B  [\hh^2] \|_{\sigma,4+\mu -\sigma, \Gamma}   \, &\le \, C\ve^{-\sigma -2 }
\| \A[\hh^1]- \A[\hh^2]  \|_{C^{0,\sigma}_{4+\mu}(\Gamma_\ve)}
\nonumber\\
\, &\le \,  C\,
\ve^{1-\sigma}\|\hh_1 -\hh_2\|_{C^{2,\sigma}_{4+\mu} (\Gamma_\ve)}\, .
\label{B2}\end{align}

\subsection{The proof of Theorem \ref{teo4}}
At this point, we make use of a  linear result.
for the problem
\be
\JJ_\Gamma[\hh]: = \Delta_{\Gamma} \hh + |A_\Gamma|^2\hh = {\tt g}  \inn \Gamma .
\label{ejq}\ee

\begin{lemma}\label{prop0}
Let $\mu'>0$. Then there exists a  positive constant $C>0$ such that
if ${\tt g}$ satisfies
$$ \|{\tt g}\|_{\sigma,4+\mu',\Gamma} \,  <\, +\infty$$
 then there is a solution  $\hh = \TT(\tg)$ of equation \equ{ejq}
that defines a linear operator of $\tg$  that satisfies

$$\|\hh\|_{2,\sigma,4+\mu',\Gamma}:=  \|D^2_\Gamma \hh \|_{\sigma,4+\mu',\Gamma} +
 \|\hh\|_{\sigma,2+\mu',\Gamma}\, \le \, C\,  \|{\tt g}\|_{\sigma,4+\mu',\Gamma}. $$
\end{lemma}

This result follows from Proposition \ref{prop11} and Corollary \ref{coro} in the next section. We will use it here to conclude the result.
We find a solution to Equation \equ{finj} if we solve the fixed point problem
\be
\hh \ =\ \TT ( \BB[\hh] ),
\label{fixedpoint}\ee
where we choose $\mu' =\frac \mu 2$, ($0<\mu <1$), and $\sigma \le \frac \mu 2$.
From the lemma, and estimates \equ{B1}, \equ{B2} we find that the operator on the right hand side of
\equ{fixedpoint} is a contraction mapping of the region
where
$$
\|\hh\|_{2,\sigma,4+\mu',\Gamma}\ \le\ C \ve^{ 2-\sigma }< \ve^{\frac 32} .
$$
hence there is a solution of \equ{fixedpoint}  and hence of \equ{finj} in this region.
Obviously constraint \equ{assh} is satisfied by this solution for all small $\ve$.

We claim that  $\Gamma_\ve^h$ is the graph of an entire function. This is equivalent  to showing that
$\Gamma^{\ve h} =  \{ y + \ve h(y) \ /\ y\in \Gamma\}$
is a graph, provided that $\ve$ is sufficiently small.
The map
$$
(y,z)  \in \Gamma \times (-\delta, \delta) \mapsto   y +     z \nu(y)
$$
defines a diffeomorphism onto a tubular neighborhood of $\Gamma$, for sufficiently small $\delta$. Moreover, since curvatures of $\Gamma$ are actually decaying at infinity, we have moreover that its inverse $x\mapsto (y(x), z(x))$ has uniformly bounded
derivatives. Now, we have that
 $\Gamma^{\ve h}$ is described by the equation
$
V(x):= z(x) - \ve h(y(x)) = 0.
$
We observe then that
$$
\pp_{x_9} V(x) =  \pp_{x_9} z - \ve \nn_\Gamma h (y) \cdot \pp_{x_9} y =   \nu_9 +  O( \ve \rr^{-2} ).
$$
thanks to estimate \equ{formh3}.
Since $$ \nu_9 = \frac 1{ \sqrt{ 1+ |\nn F|^2}} \ge  \frac c{\rr^2} $$
for some $c>0$ it follows that
$\pp_{x_9} V(x) >0$ at every point of $\Gamma^{\ve h}$, and hence this manifold can be described locally about each of its points as a graph of a function of $(x_1,\ldots, x_8)$, hence this is also globally the case. This function is clearly entire.
The proof of the Theorem is thus concluded. \qed

\bigskip
In the next section we shall conclude the proofs of Theorems \ref{teo5}-\ref{teo6} by solving the reduced problem
\equ{finj} in the situations there considered. In order not to lose the main thread of the presentation of the results, we postpone
the rather delicate analysis to  Section \ref{s10}.

\setcounter{equation}{0}
\section{Proof of Theorem \ref{teo5} }\label{s8}

The proof follows the same scheme as that for the epigraph case, so that we only exhibit the differences.
The step of the improvement of the approximation is actually identical. The coefficients of the metric, as well as the curvatures, decay faster
in $r$ than those of the minimal graph. In fact the Jacobi operator is at main order Laplacian along each of its leaves.
The slight difference is that now we will need to find sufficiently far away a solution of
$$J_\Gamma[h] = p\inn \Gamma $$ where $p= O( r^{-2-\sigma}) $
as $r\to \infty$. We do not solve this problem directly but rather
its projected version
$$J_\Gamma[h] = p - \sum_{j=1}^4 c_i\frac {z_i}{1+ r^3} \inn \Gamma $$,
$$ \int_{\Gamma}  \frac {z_i}{1+ r^3}h = 0 \foral i=1,\ldots, 4 $$
where $z_i$'s are the Jacobi fields associated to rigid motions.
At this point we refer to the theory developed in \cite{jdgdkw} that allow to solve this problem for bounded $h$ in which in addition one has fast decay of first and second derivatives. We can adapt the theory to the use of weights with Holder norms like in this paper in a straighforward way.
The final problem that is to solve is actually a projected version of \equ{finj},
\be
\Delta_{\Gamma} \hh + |A_\Gamma(y)|^2\hh =  \mathcal B  [\hh]\, (y)  - \sum_{j=1}^4 c_i\frac {z_i}{1+ r^3}  \inn \Gamma .
\label{finj1}\ee
 $$ \int_{\Gamma}  \frac {z_i}{1+ r^3}\hh = 0 \foral i=1,\ldots, 4 $$

We conclude, as in \cite{jdgdkw}, the existence of a solution $u$ and coefficients $c_i$ such that \be
\Delta u + f(u) = \sum_{i=1}^4 c_i \frac{z_i}{1+r^3} w^{'} (t) ,\ u>0  \inn \Omega_\ve, \quad u\in L^\infty (\Omega_\ve),
\label{1s10}\ee
\be
u=0, \quad \frac{\pp u}{\pp \nu} = constant  \onn \pp\Omega_\ve
\label{2s10}\ee

 An interesting, but important fact is that for Serrin's overdetermined problem Pohozaev's identity also holds and the boundary term vanishes. A slight variation of the  argument given in \cite{jdgdkw} pp. 99-102, that exploits the invariance under rotation and translations of the problem, yields that  $c_i =0$ for all $i$, and the construction is concluded. In the case of the catenoid we can further restrict ourselves to the space
of axially symmetric functions to conclude the existence of an axially symmetric solution. \qed

\setcounter{equation}{0}
\section{ Proofs of Theorems \ref{teo7} and \ref{teo8}} \label{s9}

In this section, we sketch the proofs of Theorems \ref{teo7} and \ref{teo8} and Corollary \ref{teo6} which   follow from the general scheme of proof of Theorem \ref{teo4}. The notable difference here is that the first error is ${\mathcal O} (\ve)$ only but thanks to the CMC condition the first approximate can be made to depend on the signed distance to the surface only.  As before we need to improve the error up to order ${\mathcal O} (\ve^4)$. The reduced problem--the Jacobi operator-can be solved easily in this case, thanks to the compactness and non-degeneracy condition. So we shall concentrate only on the part of improving the errors.

\subsection{Fermi coordinates and the expression of the Laplace-Beltrami operator}

In this section, we assume that $\Gamma$ is an oriented smooth hypersurface embedded in $M$. We first define the Fermi coordinates about $\Gamma$ and then, we provide some asymptotic  expansion of the Laplace-Beltrami operator in Fermi coordinates about $\Gamma$.

We denote by ${\bf n}$ a unit normal vector field on $\Gamma$ and we define
\begin{equation}
Z ({\tt x}, z) : = \text{Exp}_{\tt x}  (z \, {\bf n} ({\tt x})) ,
\label{eer1}
\end{equation}
where ${\tt x} \in \Gamma$,  $z \in {\R }$ and $\text{Exp}$ is the exponential map. The implicit function theorem implies that $Z$ is a local diffeomorphism  from a neighborhood of a point $({\tt x}, 0) \in \Gamma \times {\bf  R}$ onto a neighborhood of ${\tt x} \in M$.

\begin{remark}
In the special case where $(M, g)$ is the Euclidean space, we simply have
\[
Z ({\tt x}, z) : =  {\tt x} + z \, {\bf n} ({\tt x}).
\]
\end{remark}

Given $z \in {\R }$, we define $\Gamma_z$ by
\[
\Gamma_z := \{ Z ({\bf x} ,z)  \in M \, : \,  {\bf x} \in \Gamma \} .
\]
Observe that for $z$ small enough (depending on the point $y \in \Gamma$ where one is working), $\Gamma_z$ restricted to a neighborhood of $y$ is a smooth hypersurface which will be referred to as the
{\em hypersurface parallel to $\Gamma$ at height $z$}.  The induced metric on $\Gamma_z$ will be denoted by $g_z$.

The following result is a consequence of Gauss' Lemma. It gives the expression  of the metric  $g$ on the domain of  $M$ which is parameterized by $Z$.
\begin{lemma}
We have
\[
	Z^* \, g =  g_z + dz^2,
\]
where $g_z$ is considered as a family of metrics on $T\Gamma$, smoothly depending on $z$, which belongs to a neighborhood of $0\in \mathbb R$.
\end{lemma}
\begin{proof}
It is easier to work in local coordinates. Given $y \in \Gamma$, we fix local coordinates $x : = (x_1, \ldots, x_n)$ in a neighborhood of $0 \in {\mathbb R}^n$ to parameterize a neighborhood of $y$ in $\Gamma$ by $\Phi$, with $\Phi (0) =y$.  We consider the mapping
\[
\tilde F  \, ( x, z )=  {\rm Exp}_{\Phi (x) } ( z  \, N (\Phi(x))) ,
\]
which is a local diffeomorphism from a neighborhood of $0 \in {\mathbb R}^{n+1}$ into a neighborhood of $y$ in $M$. The corresponding coordinate vector fields are denoted by
\[
X_{0} : = \tilde F_* (\partial_{z}) \qquad {\rm and} \qquad  X_j : = \tilde F_* (\partial_{x_j}) ,
\]
for $j = 1, \ldots, n$. The curve $x_0 \longmapsto \tilde F(x_0, x)$ being a geodesic we have  $g(X_0, X_0) \equiv 1$. This also implies that  $\nabla^g_{X_0} X_0 \equiv 0$  and hence  we get
\[
\partial_{z} g(X_0, X_j) = g ( \nabla^g_{X_0} \, X_0, X_j ) + g ( \nabla^g_{X_0} \, X_j, X_0 ) = g ( \nabla^g_{X_0} \, X_j, X_0 ) .
\]
The vector fields $X_0$ and $X_j$ being coordinate vector fields we have $\nabla^g_{X_0} X_j = \nabla^g_{X_j} X_0$ and we conclude that
\[
2 \, \partial_{z} g(X_0, X_j) =  2 \, g ( \nabla^g_{X_j} \, X_0, X_0 ) = \partial_{x_j} g(X_0, X_0) = 0 .
\]
Therefore, $g(X_0, X_j)$ does not depend on $z$ and since on $\Gamma$ this quantity is $0$ for $j=1, \ldots, n$, we conclude that the metric $g$ can be written as
\[
g = g_{z}  + dz^2 ,
\]
where $g_{z}$ is a family of metrics on $\Gamma$ smoothly depending on $z$ (this is nothing but Gauss' Lemma).
\end{proof}

The next result expresses, for $z$ small, the expansion of $g_z$ in terms of geometric objects defined on $\Gamma$.  In particular, in terms of   $\mathring  g$  the induced metric on $\Gamma$,  $\mathring  h$  the second fundamental form on $\Gamma$, which is defined by
\[
	\mathring  h (t_1, t_2)  : =  - \mathring g ( \nabla^g_{t_1} N  , t_2 \, ),
\]
and in terms of the square of the second fundamental form which is the tensor defined by
\[
	\mathring   h  \otimes  \mathring  h  (t_1, t_2) : =  \mathring g (  \nabla^g_{t_1}  N ,
	\nabla^g_{t_2} N ),
\]
for all $t_1, t_2 \in T\Gamma$. Observe that, in local coordinates, we have
\[
	( \mathring h \otimes \mathring h)_{ij} =  \sum_{a,b} \mathring h_{ia} \, \mathring g^{ab}  \, \mathring h_{b j}.
\]

With these notations at hand, we have the~:
\begin{lemma}
The induced metric $g_z$ on $\Gamma_z$ can be expanded in powers of $z$ as
\[
	g_z  = \mathring g - 2 \, z \, \mathring h + z^2 \, \left( \mathring h \otimes \mathring h  +  g( R_g (\, \cdot \, , N), \, \cdot \, , N ) \right) + \mathcal O (z^3),
\]
where $R_g$ denotes the Riemannian tensor on $(M,g)$.
\label{le:2.1n}
\end{lemma}
\begin{proof}
We keep the notations introduced in the previous proof. By definition of $\mathring g$, we have
\[
g_{z} =\mathring g + {\mathcal O} (z) .
\]
We now derive the next term the expansion of $g_{z}$ in powers of $z$. To this aim, we compute
\[
\partial_{z} \, g  (X_i, X_j) =  g ( \nabla^g_{X_i} \, X_0, X_j ) +  g ( \nabla^g_{X_j} \, X_0, X_i ) \, ,
\]
for all $i,j = 1, \ldots, n$.   Since $X_0 =N$ on $\Gamma$, we get
\begin{equation}
\label{eqN1}
\partial_z \, \bar g_{z} \, _{|_{z=0}} = - 2 \, \mathring h ,
\end{equation}
by definition of the second fundamental form.  This already implies that
\[
g_{z} = \mathring g  - 2 \,  \mathring h \, z + {\mathcal O} (z^2) \, .
\]

Using the fact that the $X_0$ and $X_j$ are coordinate vector fields, we can compute
\begin{equation}
\partial_{z}^2 \, g (X_i, X_j) =  g( \nabla^g_{X_0} \, \nabla^g_{X_i}\, X_0, X_j ) +
g( \nabla^g_{X_0} \, \nabla^g_{X_j}\, X_0, X_i ) +  2 \, g( \nabla^g_{X_i} \, X_0 ,
\nabla^g_{X_j} X_0 ).
\label{eq:1-1}
\end{equation}
By definition of the curvature tensor, we can write
\[
\nabla^g_{X_0} \, \nabla^g_{X_j} = R_g (X_0 , X_j) + \nabla^g_{X_j} \, \nabla^g_{X_0}
+ \nabla^g_{[X_0, X_j]} \, ,
\]
which, using the fact that $X_0$ and $X_j$ are coordinate vector fields, simplifies into
\[
\nabla^g_{X_0} \, \nabla^g_{X_j} = R_g (X_0 , X_j) + \nabla^g_{X_j} \, \nabla^g_{X_0} \, .
\]
Since $\nabla^g_{X_0} \, X_0 \equiv 0$, we get
\[
\nabla^g_{X_0} \, \nabla^g_{X_j} X_0 = R_g (X_0 , X_j) \, X_0 \, .
\]
Inserting this into (\ref{eq:1-1}) yields
\[
\partial_{z}^2 \, g (X_i, X_j)  =  2 \, g ( R_g (X_0, X_i) \, X_0, X_j ) + 2 \, g( \nabla^g_{X_i} \, X_0,  \nabla^g_{X_j} X_0 ) \, .
\]
Evaluation at $x_0=0$ gives
\[
\partial_{z}^2 \,  g_{z} \,_{|_{z=0}} = 2 \, g ( R(N, Ê\cdot ) \, N, \cdot )  + 2 \,g( \nabla^g_{\cdot} \, N ,  \nabla^g_{\cdot} N ) .
\]
The formula then follows at once from Taylor's expansion.\end{proof}

Similarly,  the mean curvature $H_z$ of $\Gamma_z$ can be expressed in term of $\mathring g$ and $\mathring h$. We have the~:
\begin{lemma}
The following expansion holds
\[
	H_{z} =  {\rm Tr}_{\mathring g} \mathring h  + z\,  \left( {\rm Tr}_{\mathring g} \mathring h \otimes  \mathring h  + {\rm Ric}_g (N,N) \right)  + \mathcal O (z^2) ,
\]
for $z$ close to $0$.
\label{le:2.2}
\end{lemma}
\begin{proof}
The mean curvature appears in the first variation of the volume form of parallel
hypersurfaces, namely
\[
	H_z = - \frac{1}{\sqrt{{\rm det} \,  g_z}} \, \frac{d}{dz} \sqrt{\mbox{det} \, g_z}.
\]
The result then follows at once from the expansion of the metric $g_z$ in powers of $z$ together with the well known formula
\[
{\rm det} (I + A) = 1 + \mbox{Tr} A +\frac 1 2 \,  \left(  (  \mbox{Tr} A)^2 -  \mbox{Tr} A^{2}  \right) + \mathcal O (\|A\|^3),
\]
where $A \in M_n (\mathbb R)$.
\end{proof}

Recall that, in local coordinates, the Laplace Beltrami operator is given by
\[
	\Delta_g = \frac{1}{\sqrt{ |g| }} \, \partial_{x_i} \left( g^{ij} \, \sqrt{ | g |} \, \partial_{x_j}
	\, \right) .
\]
 Therefore, in a fixed  tubular neighborhood of $\Gamma$, the Euclidean Laplacian in $
\mathbb R^{n+1}$ can be expressed in Fermi coordinates by the (well-known) formula
\begin{equation}
	\Delta_{g_{e}}  = \partial^2_z  - H_z \, \partial_z +  \Delta_{g_z} .
\label{Lapx}
\end{equation}

In the  case where the ambient manifold is the Euclidean space, we get
\begin{lemma}
The induced metric $g_z$ on $\Gamma_z$ is given by
\[
g_z  = g_0 - 2 \, z\, k_0 + z^2 \, k_0 \otimes k_0  \, .
\]
in other words
\[
g_z  : =  g_0 ((I-zA) \, \cdot \, , (I-zA) \,\cdot \, )
\]
where $A$ is the shape operator defined by
\begin{equation}
\label{shapeoperator}
A= ({\mathring h}_{ij}).
\end{equation}
\label{le:2.1}
\end{lemma}
\begin{proof}
We just need to compute
\[
\begin{array}{rllll}
\partial_{x_i} X_z \cdot \partial_{x_j} X_z = \partial_{x_i} X \cdot  \partial_{x_j} X
+  z \, \left( \partial_{x_i} X \cdot \partial_{x_j} \tilde N   + \partial_{x_i} \tilde N \cdot
\partial_{x_j} X \right) + z^2 \, \partial_{x_i} \tilde N \cdot \partial_{x_j} \tilde N  ,
\end{array}
\]
where $\tilde N : = X^* N$. We can use (\ref{eqN1}) to write
\[
\partial_{x_i} \tilde N \cdot \partial_{x_j} \tilde N =  A \, \partial_{x_i} X  \cdot A \, \partial_{x_j} X
\]
And, using the definition of the first and second fundamental forms on $\Gamma$, we conclude that
\[
\begin{array}{rllll}
\partial_{x_i} X_z \cdot \partial_{x_j} X_z  =  \partial_{x_i} X \cdot  \partial_{x_j} X  -
2 \, z \,  (A \, \partial_{x_i} X ) \cdot \partial_{x_j} X  +   z^2 \,  (A \, \partial_{x_i} X) \cdot  (A \, \partial_{x_j} X)  \, .
\end{array}
\]
This completes the proof of the result. \end{proof}

Similarly,  the mean curvature $H_z$ of $\Gamma_z$ can be expressed in term of $z$ and $A$,  the shape operator about $\Gamma$ which has been defined in (\ref{shapeoperator}). We have the~:
\begin{lemma}
The following expansion holds
\[
H_{z} =  \sum_{k=0}^\infty  {\rm Tr} \, (A^{k+1})  \, z^k \, .
\]
\label{le:2.21}
\end{lemma}
\begin{proof}
The mean curvature appears in the first variation of the volume form of parallel hypersurfaces, namely
\[
H_z = - \frac{1}{\sqrt{\mbox{det} \,  g_z}} \, \frac{d}{dz} \sqrt{\mbox{det} \, g_z}
\]
Hence we find that
\[
H_z =  - \frac{1}{\mbox{det} \,  (I-z\, A)} \, \frac{d}{dz} \mbox{det} (I-z\, A) =  {\rm Tr} \, ( A \, (I - z\, A)^{-1})
\]
and the result follows.
\end{proof}

\subsection{Construction of an approximate solution}

Given any (sufficiently small) smooth function $h$ defined on $\Gamma$, we define $\Gamma_h$ to be the normal graph over $\Gamma$ for the function $h$. Namely
\[
\Gamma_h := \{ y + h(y) \, N(y) \in \mathbb R^{n+1}  \, : \, y \in \Gamma\} \, .
\]
We also define the epigraph
\[
\Omega_h : =  \{ y + t \, N(y)  \in \mathbb R^{n+1} \, : \, y \in \Gamma ,  \quad t \geq h(y) \} \, .
\]

We would like to solve the equation
\[
\Delta u + f (u) =0 \, ,
\]
in $\Omega_h$, with $u=0$ and $\partial_\nu u = {\rm constant}$ on $\partial \Omega_h$. In this section, we explain how to build a function $h$ and a function $u$ which solve this overdetermined problem to high order of accuracy. The construction makes use of an iteration scheme which can be used to determine all the orders successively.

We keep the notations of the previous section and, in a tubular neighborhood of $\Gamma$ we write
\[
u(z, y) = v \left( \frac{z-h(y)}{\ve}, y \right) \, ,
\]
where $h$ is a (sufficiently small) smooth function defined on $\Gamma$. It will be convenient to denote by $t$ the variable
\[
t : = \frac{z-h(y)}{\ve} \, .
\]
Using the expression of the Laplacian in Fermi coordinates which has been derived in (\ref{Lapx}), we find with little work that the equation we would like to solve can be rewritten as
\begin{equation}
\begin{array}{rllll}
\Big[ ( 1 +  \|Êd h\|^2_{g_\zeta}  ) \, \partial_t^2 v + \ve^2 \Delta_{g_\zeta} v  -  \ve \, \left( H_\zeta  + \Delta_{g_\zeta} h    \right) \, \partial_t v & \\[3mm]
- \ve \, ( d h , d \, \partial_t v)_{g_\zeta} \Big]_{|\zeta = \ve \, t + h}  &+  f(v) =0
 \end{array}
\label{eq:rfe}
\end{equation}
for  $t >0$ close to $0$ and $y \in \Gamma$. Some comments are due about the notations. In this equation and below  all computations of the quantities between the  square brackets $[ \quad ]$ are performed using the metric $g_\zeta$ defined in Lemma~\ref{le:2.1} and considering that $\zeta$ is a parameter, and once this is done, we set $\zeta = t+h(y)$.

\medskip

The fact that we ask that $u$ has $0$ boundary data translates into
\[
v(0, y) = 0
\]
on $\Gamma$. Finally, the Neumann data of $u$ reads
\[
\mathfrak N (v, h) : = \left(  1+ \| d h\|_{g_\zeta}^2\right)^{1/2}_{\zeta = h(y)}  \, \partial_t v
\]
where this time the expression between the square brackets is evaluated at $t=0$.

\medskip

We set
\[
v(t,y)  = w(t) + \phi (t, y )
\]
where $w$ is the solution of (\ref{eqnforw}). In this case, the equation (\ref{eq:rfe}) becomes
\[
\mathfrak M (v,h) =0 \, ,
\]
where we have defined
\begin{equation}
\begin{array}{rllll}
\mathfrak M (v,h)  & : = & \displaystyle \Big [ \left( \partial_t^2 + \ve^2 \Delta_{g_\zeta}  + f'(w) \right) \,  \phi - \ve \, ( \Delta_{g_\zeta} h + H_\zeta ) \,  ( w'  +  \partial_t \phi ) \\[3mm]
& -  &  \| d h\|^2_{g_\zeta} \, (w'' + \partial^2_t \phi)  - \ve \, ( d h , d \, \partial_t \phi)_{g_\zeta}  \Big]_{| \zeta = \ve t + h} \\[3mm]
& + & \left( f(w+\phi) - f(w) - f'(w) \, \phi\right)
\end{array}
\label{eq:nrfe}
\end{equation}

We now perform some formal computation to determine the solution $\phi$ and $h$ so that $\mathfrak N (w+\phi, h)$ is constant.  We assume that $\phi$ and $h$ can be expanded in power of $\ve$ as
\[
\phi  =\ve \phi_0+   \ve^2 \, \phi_1 +  \ve^3 \, \phi_2 + \ve^4 \, \phi_3 +  \ldots
\]
and
\[
h =  \ve \, h_0 + \ve^2 \, h_1  + \ve^3 \, h_2 +  \ldots
\]
where all functions $\phi_j$ depend on $t$ and $y$ while the functions $h_j$ only depend on $y$. We naturally assume that $\phi_j$ nor $h_j$ depend on $\ve$ and, since this will turn out to be the case and since this simplifies the computations, we also assume that $h_0$ is constant and $ \phi_0= \phi_0 (t)$.

\begin{lemma}
The following expansion holds
\begin{equation}
\begin{array}{rllll}
 \mathfrak M (w+ \phi , h)  & = & \ve \,  \Big (  ( \partial_t^2 + f'(w)) \phi_0 -  {\rm Tr} (A) \, w'   \Big)\\[3mm]
& +  & \ve^2 \,  \Big (  ( \partial_t^2 + f'(w)) \phi_1 -  {\rm Tr} (A^2) \, (t+h_0) \, w'   -{\rm Tr } (A) \partial_t \phi_0 +\frac{1}{2} f^{''} (w) \phi_0^2 \Big)\\[3mm]
& + & \ve^3 \, \Big( ( \partial_t^2 + f'(w)) \phi_2 -  ({\rm Tr} (A^3) \, (t+h_0)^2 +  J_\Gamma \, h_1 ) \, w'  \\[3mm]
&  -&    {\rm Tr } (A^2) (t+h_0) \partial_t \phi_0  +\frac{1}{2} f^{''} (w) \phi_0 \phi_1 + \frac{1}{6} f^{'''}(w) \phi_0^3  \Big)\\[3mm]
& + & \ve^4 \, \Big( ( \partial_t^2 + f'(w)) \phi_3 - ( {\rm Tr} (A^4) \, (t+h_0)^3  +  J_\Gamma \, h_2 ) \, w'   \\[3mm]
&  &  \quad - {\rm Tr} (A^2) \, (t+h_0)\, \partial_t \phi_1  + \Delta_{g_0} \, \phi_1 + \frac{1}{2} \, f''(w)  \phi_1^2   - \| dh_1\|_{g_0}^2 \, w''   \\[3mm]
&  &  \quad - \left( 2 \, {\rm Tr} (A^3) \, h_1 +  [Ê\partial_\zeta \Delta_{g_\zeta} h_1]_{|\zeta =0} \right) \, (t+h_0) \, w'  \\[3mm]
 & \ & - {\rm Tr} (A^2) h_1 \partial_t \phi_0 - {\rm Tr} (A^3) (t+h_0)^2 \partial_t \phi_0 \\[3mm]
&\  &  +f^{(2)}(w)\phi_0 \phi_2 + \frac{ f^{(3)} (w)}{3} \phi_0 \phi_1 + \frac{f^{(4)}(w)}{4!} \phi_0^4 \Big) \\[3mm]
& + & \mathcal O (\ve^5)
\end{array}
\label{eq:expM}
\end{equation}
where
\[
J_\Gamma : =  (\Delta_{g_0} + {\rm Tr} (A^2))
\]
is the Jacobi operator about $\Gamma$ and $f^{(j)} $ denotes the $j-$th order derivative of $ f$.
\label{le:3.11}
\end{lemma}
\begin{proof} Under the above assumptions, the following expansion is easy to derive
\[
\left[ \ve^2 \Delta_{g_\zeta}  \phi \right]_{| \zeta = \ve t + h}=  \ve^4 \, \Delta_{g_0}  \phi_1 + \mathcal O (\ve^5)
\]
since $\phi_0=\phi_0 (t)$ which does not depend on $y$.
Using the fact that $g_\zeta$ depends smoothly on $\zeta$ together with the facts that $h_0$ is constant and $\phi_0 =\phi_0 (t)$, we get
\[
\left[ \Delta_{g_\zeta} h  \right]_{| \zeta = \ve t + h} = \ve \, \Delta_{g_0} (h_1 + \ve \, h_2 + \ve^2 \, h_3) + \ve^3 \, \left[ \partial_\zeta \Delta_{g_\zeta} h_1  \right]_{| \zeta = 0} \, (t + h_0) + \mathcal O (\ve^4)\, ,
\]
Using the result of Lemma~\ref{le:2.2}, we obtain the expansion
\[
\begin{array}{rllll}
\left[ H_\zeta \right]_{| \zeta = \ve t + h} & = & {\rm Tr} (A) + \ve \, {\rm Tr} (A^2) \, (t +h_0 + \ve \, h_1+ \ve^2 \, h_2 )  \\[3mm]
& + & \ve^2 \, {\rm Tr} (A^3) (t +h_0 ) (t+h_0 + 2 \, \ve \, h_1) \\[3mm]
& + & \ve^3 \, {\rm Tr} (A^4) (t +h_0)^3 + \mathcal O (\ve^4)
\end{array}
\]
Next, we have
\[
\left[ \| d h\|^2_{g_\zeta}   \right]_{|\zeta = \ve t + h} =  \ve^{4} \,  \| d h_1\|^2_{g_0} + \mathcal O (\ve^5) \,
\]
since $h_0$ is assumed to be constant. Similarly, we get
\[
\left[ ( d h , d \, \partial_t \phi)_{g_\zeta}  \right]_{|\zeta = \ve t + h} = \mathcal O (\ve^4)
\]
since $h_0$ is constant and $\phi_0= \phi_0 (t)$. Finally, Taylor's expansion yields
\[
\begin{array}{rllll}
f(w+\phi) - f(w) - f'(w) \, \phi & = &   \frac{1}{2}  f''(w)  \phi^2 + \frac{1}{3!} f^{(3)} (w) \phi^3 + \frac{1}{4!} f^{(4)} (w) \phi^4+ \mathcal O (\ve^5)
\end{array}
\]
To derive the expansion, it is enough to insert these expression in (\ref{eq:nrfe}) and rearrange the result in powers of $\ve$. \end{proof}

Using similar arguments, we also get the expansion of the normal derivative of $v+ \phi$ in powers of $\ve$.
\begin{lemma}
The following expansion holds
\[
\mathfrak N (w+\phi, h)   = w'(0) + \ve  \phi_0^{'} (0)+ \left( \ve^2 \, \partial_t \phi_1+ \ve^3 \,  \partial_t \phi_2+ \ve^4 \,  \partial_t \phi_3\right)_{| t=0} +  \| dh_1\|_{g_0}^2 + \mathcal O (\ve^5)
\]
\label{le:3.22}
\end{lemma}

In order to construct the approximate solution the idea is first to find the functions $\phi_0, \phi_1, \ldots, \phi_3$ so that $\mathfrak M (w+\phi, h) = \mathcal O(\ve^5)$. Thanks to Lemma~\ref{le:3.11}, we obtain the following system of equations
\begin{eqnarray}
\left\{ \begin{array}{llll}
( \partial_t^2 + f'(w)) \phi_0& =&   {\rm Tr} (A) \,  w'  \\[3mm]
( \partial_t^2 + f'(w)) \phi_1& =&   {\rm Tr} (A^2) \, (t+h_0) \, w'  \\[3mm]
& +&{\rm Tr} (A) \partial_t \phi_0 +\frac{1}{2} f^{''} (w) \phi_0^2 \\[3mm]
( \partial_t^2 + f'(w)) \phi_2 & =&  ({\rm Tr} (A^3) \, (t+h_0)^2 -  J_\Gamma \, h_1 ) \, w' \\[3mm]
& +&   {\rm Tr } (A^2) (t+h_0) \partial_t \phi_0  +\frac{1}{2} f^{''} (w) \phi_0 \phi_1 + \frac{1}{6} f^{'''}(w) \phi_0^3  \\[3mm]
 ( \partial_t^2 + f'(w)) \phi_3 &  = & ( {\rm Tr} (A^4) \, (t+h_0)^3  -  J_\Gamma \, h_2 ) \, w' \\[3mm]
&  + &  {\rm Tr} (A^2) \, (t+h_0)\, \partial_t \phi_1  -  \Delta_{g_0} \, \phi_1   \\[3mm]
&  - & \frac{1}{2} \, f''(w)  \phi_1^2 +  \| dh_1\|_{g_0}^2 \, w''  \\[3mm]
&  + &  (2 \, {\rm Tr} (A^3)  h_1  +  [Ê\partial_\zeta \Delta_{g_\zeta} h_1]_{|\zeta =0} ) (t+h_0) \, w'  \\[3mm]
 & \ & + {\rm Tr} (A^2) h_1 \partial_t \phi_0 + {\rm Tr} (A^3) (t+h_0)^2 \partial_t \phi_0 \\[3mm]
&- &  f^{(2)}(w)\phi_0 \phi_2 - \frac{ f^{(3)} (w)}{3} \phi_0 \phi_1 + \frac{f^{(4)}(w)}{4!} \phi_0^4  \\[3mm]
\end{array}
\right.
\label{eq:sis3}
\end{eqnarray}
We consider this system of equation as a system of ordinary differential equations (of the variable $t>0$) which depends smoothly on parameters (namely $y \in \Gamma$) through the functions ${\rm Tr} (A^k)$, the metric $g_0$ on $\Gamma$ or the functions $h_0, \ldots, h_2$.

The next step relies on the solvability of a second order ordinary differential equation. We shall solve $\phi_0$ and $\phi_1, \phi_2, \phi_3$ differently. First we solve $\phi_0$.
Observe that a  crucial fact is that since $M$ is a CMC surface
\begin{equation}
{\rm Tr } (A)\equiv Constant
\end{equation}
so that  $\phi_0$ can be chosen to be  a function of $t$ only. In fact we can choose
\begin{equation}
\label{phi0}
\phi_0 (t)= -{\rm Tr } (A) p_0 (t)
\end{equation}
where $p_0 (t)$ is the unique bounded solution to (\ref{eq:p0}).
Observe that $ \phi_0 (0)=0$ and $\phi_0$ decays exponentially in $t$.

Next we solve $\phi_1, \phi_2$ and $\phi_3$ successively and at the meantime we also determine $h_0, h_1, h_2$. This is similar to the procedure done in Section \ref{s4}.

As observed in Section \ref{s4} for a bounded function $q(t)\in L^\infty (0, \infty)$ a necessary and sufficient condition to obtain a bounded solution $p(t)$ to
\begin{equation}
p^{''}+ f^{'} (w) p = q (t), \ \  p(0)=p^{'} (0)=0
\end{equation}
is the following
\begin{equation}
\label{le:3.1}
\int_0^\infty q(t) w^{'} (t) dt=0
\end{equation}
In fact the solution $p(t)$ is given by (\ref{formula p}).

\medskip

We now explain how the constant $h_0$ and the functions $h_1$ and $h_2$ are chosen.  The idea is to use (\ref{le:3.1}) in order to solve the system (\ref{eq:sis3}) for any given constant $h_0$ and any given set of functions $h_1, \ldots, h_3$ and then we determine $h_0, \ldots, h_3$ so that
\begin{equation}
\partial_t \phi_1 \, _{| t=0} =  \partial_t \phi_2 \, _{| t=0} =  \partial_t \phi_3 \, _{| t=0} +  \ve^4 \, \| dh_1\|_{g_0}^2 =0
\label{eq:ccqfr}
\end{equation}
on $\Gamma$.

To begin with, observe that, thanks to when $t=0$, we have
\[
w'(0) \, \partial_t \phi_1 \, _{| t=0} = \left( \int_0^\infty (t+h_0)\, w'(t) \, dt \right) \,  {\rm Tr} (A^2) \]
\[ + \int_0^{\infty} ({\rm Tr} (A) \partial_t \phi_0 +\frac{1}{2} f^{''} (w) \phi_0^2) w^{'}(t) dt  \,
\]
and hence, the first equation in (\ref{eq:ccqfr}) amounts to ask that the constant $h_0 \in \mathbb R$ is chosen so that
\begin{equation}
\label{eq:h0}
\int_0^\infty (t+h_0)\, {w'}^2 \, dt  =- \int_0^{\infty} ({\rm Tr} (A) \partial_t \phi_0 +\frac{1}{2} f^{''} (w) \phi_0^2) w^{'}(t) dt  \,
\end{equation}
which can be solved uniquely for $h_0$.

Next we choose $h_1$ so that  $\partial_t \phi_2|_{t=0}=0$. By  (\ref{le:3.1}) this amounts to choosing $h_1$ such that
 \begin{eqnarray*}
J_\Gamma (h_1) \int_0^{\infty} (w^{'} (t))^2 dt&=& {\rm Tr} (A^3) \,\int_0^{\infty}  (t+h_0)^2 (w^{'}(t))^2 dt \\[3mm]
& +& \int_0^{\infty} \Big( {\rm Tr } (A^2) (t+h_0) \partial_t \phi_0  +\frac{1}{2} f^{''} (w) \phi_0 \phi_1 + \frac{1}{6} f^{'''}(w) \phi_0^3\Big) w^{'} (t) dt
\end{eqnarray*}
which has a unique solution $h_1$, thanks to the nondegeneracy assumption on $M$.

A similar argument as above can be used to solve $\phi_2$ so that $\partial_t \phi_2|_{t=0}=0$ and hence a unique $h_2$ can be found.

 If we succeed in achieving these choices to determine $h_1$ and $h_2$, then according to Lemma~\ref{le:3.11} and Lemma~\ref{le:3.22}, this will ensure that
\[
\mathfrak M (w+ \bar \phi , \bar h)   =  \mathcal O (\ve^5)
\]
in a neighborhood of $\Gamma$ in $\Gamma \times [0, \infty)$  and
$$ (w+ \bar \phi)|_{t=0}=0, $$
\[
\mathfrak N (w+ \bar \phi, \bar h)   = w'(0) + \ve \phi_0^{'} (0)+ \mathcal O (\ve^5)
\]
on $\Gamma$ for
\[
\bar \phi : =\ve \phi_0+  \ve^2 \, \phi_1 +  \ve^3 \, \phi_2 + \ve^4 \, \phi_3  \qquad \mbox{and} \qquad \bar h : =   \ve \, h_0 + \ve^2 \, h_1  + \ve^3 \, h_2 \,
\]

We could use this iteration scheme to solve the equations $\mathfrak M (v,h) =0$ and $\mathfrak N(v, h) = {\rm constant}$ to any order but it turns out that the above accuracy will be sufficient for our purpose.

Proceeding as in the proof of Theorem \ref{teo4}, we look for true solutions of the form
\[
\bar \phi : =\ve \phi_0+  \ve^2 \, \phi_1 +  \ve^3 \, \phi_2 + \ve^4 \, \phi_3  +\phi \qquad \mbox{and} \qquad \bar h : =   \ve \, h_0 + \ve^2 \, h_1  + \ve^3 \, h_2 \, +h
\]
where we use $\| \cdot \|_{C^{2, \sigma}_{0, \gamma} (\Gamma_\ve \times (0, +\infty))}$ to measure $\phi$ and $\| \cdot \|_{C^{2, \sigma} (\Gamma)}$ to measure the function $h$. Since $\Gamma$ is compact and non-degenerate,  the rest of the proof goes exactly as those of Theorem \ref{teo4}. We omit the details.

\bigskip

\setcounter{equation}{0}
\section{Appendix: The BDG graph and its Jacobi operator}
\label{s10}

In this appendix we let $\Gamma$ a fixed Bombieri-De Giorgi-Giusti minimal graph \cite{BDG}, as in the statement of Theorem \ref{teo4}.
We begin by some preliminary facts in \cite{BDG} and \cite{dkwdg}.

\subsection{{The Bombieri-De Giorgi-Giusti minimal graph}}\label{bdg0}

Let us consider the minimal surface equation in entire space $\R^8$,
 \be H[F]\ :=\ \nn \cdot \left ( \frac {\nn F } { \sqrt{ 1+
|\nn F|^2 }} \right ) = 0 \inn \R^8. \label{mse}\ee
The quantity $H[F]$
corresponds to mean curvature of the hypersurface in $\R^9$,
$$ \Gamma := \{ (x',F(x'))\mid\ x'\in \R^8\}. $$
The Bombieri-De Giorgi-Giusti minimal graph \cite{BDG} is a non-trivial,  entire smooth solution of equation \equ{mse} that enjoys some simple symmetries which we describe next.
Let us write $x'\in \R^8$ as $x'=(\uu,\vv)\in \R^4\times \R^4$ and  consider the set
\be
T\ :=\  \{   (\uu,\vv)\in \R^8 \mid\ |\vv| > |\uu| \ \}.
\label{T}\ee
The solution found in \cite{BDG} is radially symmetric in both variables, namely $F= F(u,v)$. In addition, $F$ is positive in $T$ and it vanishes along $\pp T$. Moreover, it satisfies
\be
F(|\uu|, |\vv|) = -F(|\vv|, |\uu|)\foral \uu,\, \vv \ .
\label{simF}\ee

It is useful to introduce polar coordinates $(|\uu|, |\vv|) = (r\cos\theta, r\sin\theta)$. In  \cite{dkwdg}
it was found that $F$ is well approximated for large $r$ by a function that separates variables,
$F_0(x')= r^3 g(\theta), $ where
$g(\theta)$ solves the two-point
boundary value problem
\be
  \frac { 21  g \,\sin^32\theta }  { \sqrt{  9g^2  + {g'}^2} } \ + \  \left (   \frac {  g' \sin^32\theta }  { \sqrt{  9g^2  + {g'}^2} }   \right )' = 0 \inn \left (\frac \pi 4, \frac \pi 2\right ),\quad g\left(\frac \pi 4 \right ) =0= g'\left(\frac{\pi}{ 2} \right ).
\label{eqg}\ee

Problem $\equ{eqg}$ has a unique solution $g\in C^2([\frac \pi 4, \frac \pi 2])$ such that $g$ and $g'$ are positive in $(\frac \pi 4, \frac \pi 2)$ and such that  $g'(\frac \pi 4)=1$.

\begin{lemma}\label{exist fmc} \cite{dkwdg}
There exists an entire solution $F=F(|\uu|,|\vv|)$ to  equation $\equ{mse}$ which satisfies $\equ{simF}$ and such that
\begin{equation}
F_0\leq F\leq  F_0 + \frac{\tt C}{r^\sigma} \inn T, \quad r>R_0,
\end{equation}
where $0< \sigma<1 $, ${\tt C}\geq 1$, and  $R_0$, are  positive constants.
\end{lemma}

In what what follows we will denote, for $F$ and $F_0$ as above,
$$
\Gamma = \{ (x',F(x'))\mid \ x' \in \R^8\,\}, \quad  \Gamma_0 = \{ (x',F_0(x'))\mid\ x' \in \R^8\,\}. \quad
$$
By $\Gamma_\ve$ we will  denote the dilated surfaces $\Gamma_\ve= \ve^{-1}\Gamma$. Also,  we shall use the notation:
\be
\rr(x)\,:=\, \sqrt{1 + |x'|^2} , \quad \rr_\ve(x)\, :=\, \rr(\ve x) ,\quad x= (x',x_9) \in \R^8\times\R = \R^9.
\ee


\subsection{Local coordinates}
In \cite{dkwdg}
it is found a convenient family of local parametrizations of the surface $\Gamma$ which we describe next.
Given $p\in \Gamma$ with $p = (p',p_9)$, $R= r(p) = |p'|>>1$, we let $\nu(p)$ be its normal vector, and
$\Pi_1,\ldots, \Pi_8$ an orthonormal basis of its tangent space. Using the fact that
the curvatures of $\Gamma$ at $p$ are bounded  as $O( R^{-1})$ one finds that there exists a $\theta>0$ independent of $p$ and a smooth function $G_p(\py)$ defined on $\R^8$
with $ G(0) = G_p'(0) =0$, such that $\Gamma$ can be locally parametrized
around $p$ by the map
\be
\py \in B(0,\theta R)\subset \R^8 \ \longmapsto \ Y_p(\py) :=  p+ \sum_{j=1}^8 \py_j \Pi_j + G_p(\py ) \nu(p)\,\in\,\Gamma.
\label{cords}\ee
 Besides, for each $m\ge 2$ the following estimate holds:
$$ \| D_\py ^m G_p\|_{L^\infty (B(0,\theta R))}  \ \le \ \frac {c_m} { R^{m-1}}  $$
where $c_m$ is independent of $p$.

\medskip
Let us consider the metric $g_{ij}$ of $\Gamma$ around $p$ expressed in these coordinates.
Then
$$ g_{ij}(\py) \ := \ \left < \pp_i Y_p,\pp_j Y_p\right> = \delta_{ij} + \theta(\py) $$
where
 \begin{align} |\theta(\py) | \,\le\,  c\frac {|\py|^2}{R^2} ,&\nonumber \\  |D_\py\theta(\py) |\, \le \,c\frac {|\py|}{R^2} ,&\nonumber \\  |D_\py^m\theta(\py) | \,\le\,  \frac {c_m} { R^{m}} &\foral   |\py|< \theta R,\ m\ge 2.
\label{thetap}\end{align}

\medskip
\subsection{The Laplace Beltrami operator}
The Laplace-Beltrami operator of $\Gamma$ is expressed in these local coordinates as
$$
\Delta_{\Gamma} = \frac 1{\sqrt{\det g( \py) }}\, \pp_{i}\left ( \sqrt{\det g(\py) } \, g^{ij}(\py)\, \pp_j\,\right)
$$
Let us set
$$ a_{ij}^0 (\py) :=  g^{ij}(\py), \quad b_j^0(\py) := \frac 1{\sqrt{\det g( \py) }}\, \pp_{i}\left ( \sqrt{\det g(\py) } \, g^{ij}(\py)\, \right ).
$$
So that
\be
 \Delta_{\Gamma} \,=\,  a^0_{ij}(\py)\,\pp_{ij} +   b^0_i(\py)\,\pp_i, \quad |\py|< \theta R,
\ee
where
 \begin{align} |a^0_{ij}(\py)-\delta_{ij} | \,\le\,  c\frac {|\py|^2}{R^2} ,&\quad |D_\py a^0_{ij}(\py) | \,\le\,  c\frac {|\py|}{R^2} ,\nonumber \\
  |b_{j}^0(\py) | \,\le\,  c\frac {|\py|}{R^2} ,&\quad\  |D_\py b_{j}^0(\py) | \,\le\,  \frac {c}{R^2}
  \foral   |\py|< \theta R,\ m\ge 2.
\label{L-21}\end{align}

\bigskip
\subsection{Solvability for the Jacobi operator of the BDG graph }

We  consider the linear problem
\be
\JJ_\Gamma[h]: = \Delta_{\Gamma} h + |A_\Gamma(y)|^2h = {\tt g} (y) \inn \Gamma .
\label{ej}\ee

In \cite{dkwdg} the following result was established.
\red{
\begin{proposition}\label{prop11}

 Let $4<\nu <5$. There exists a  positive constant $C>0$ such that
if ${\tt g}$ satisfies
$$\|\rr^\nu\, g\|_{L^{\infty}(\Gamma)}\,  <\, +\infty$$
 then
there is a unique solution of equation $\equ{ej}$
 such that $\|\rr^{\nu-2}\, h\|_{L^{\infty}(\Gamma)}
<+\infty$. This solution satisfies
$$
\|\rr^{\nu-2}\, h\|_{L^{\infty}(\Gamma)} \ \le\ C\,
\| \rr^\nu \,g\|_{L^{\infty}(\Gamma)}\,  .
$$
\end{proposition}}

The proof of this result is based on the construction of explicit barriers, using the fact
that the surfaces $\Gamma$ and $\Gamma_0$ are uniformly close for $r$ large. Barriers constitute an appropriate tool to solve Problem \equ{ej} since $\JJ_\Gamma$ satisfies maximum principle, as it follows from the presence of a positive bounded function in its kernel. In fact,
we have that\\ $ \JJ_\Gamma[ (1+ |\nn F|^2)^{-1/2}] =0$.

\medskip
In the current setting we need to consider right hand sides with decay of order at most
$O( r^{-4})$, the prototypes being ${\tt g} =\sum_{i=1}^8 k_i^3$ and ${\tt g} =\sum_{i=1}^8 k_i^4$.
It is not possible in general to obtain a suitable barrier in the setting  of the above
proposition  when $\nu \le 4$.
We have however the validity of Proposition \ref{lemahc}  below which will suffice for our purposes.

\medskip
The closeness of the surfaces allows us to define a canonical correspondence between maps defined
on $\Gamma$ and functions on $\Gamma_0$ as follows.
Let $p\in \Gamma$ with $r(p)\gg 1$ and let $\nu(p)$ be the unit normal to $\Gamma$ at $p$. Let  $\pi(p)\in \Gamma_0$ be a point such that for some $t_p\in \R$ we have:
\begin{align}
\pi(p)=p+t_p\nu(p).
\label{def pip}
\end{align}

As shown in \cite{dkwdg}, the point $\pi(p)$ exists and  is unique when $r(p)\gg 1$,  and the map $p\longmapsto \pi(p)$ is smooth, with uniformly bounded derivatives both for $\pi$ and its inverse. The {\em approximate Jacobi operator} $\JJ_{\Gamma_0}$, corresponding to first variation of mean curvature at $\Gamma_0$, is given by
$$
\JJ_{\Gamma_0} [h] := \Delta_{\Gamma_0} h + |A_{\Gamma_0}(y)|^2h.
$$

For large $r$, $\JJ_\Gamma$ is ``close to''
 $\JJ_{\Gamma_0}$ in the sense of the following result, contained in \cite{dkwdg}.

\begin{lemma}\label{comparisonjacobi}
Assume that $h$ and $h_0$ are smooth functions defined respectively on $\Gamma$ and $\Gamma_0$ for $r$ large, and related through the formula
 $$h_0(\pi(y))= h(y), \quad y\in \Gamma,\quad r(y)> r_0.$$
\blue{There exists a $\sigma>0$ such that}
\be
\JJ_\Gamma [h](y) = [\JJ_{\Gamma_0} [h_0]  + O(r^{-2-\sigma}) D^2_{\Gamma_0}h_0 + O(r^{-3-\sigma}) D_{\Gamma_0}h_0 + O(r^{-4-\sigma})h_0\, ]\, (\pi(y))\ .
\label{comparisonjacobi1}\ee
\end{lemma}

We can compute explicitly the operator $\JJ_{\Gamma_0}$ as follows.
Let us consider the first variation of mean curvature measured along vertical perturbations of the graph $\Gamma_0$, namely the linear operator $H'(F_0)$ defined by
$$
H'(F_0)[\phi]\, := \, \frac d{dt} H( F_0+ t\phi)\,|_{t=0}\ =\ \nn \cdot \left ( \frac {\nn \phi } { \sqrt{ 1+ |\nn F_0|^2 }} -
\frac {(\nn F_0 \cdot \nn \phi)} { ( 1+ |\nn F_0|^2 )^{\frac 32}}  \nn F_0 \, \right ).
$$

Then we have the relation
\be
\JJ_{\Gamma_0} [h] = H'(F_0)[ \phi ], \quad\hbox{where}\quad\phi(x') =
\sqrt{1+ |\nn F_0(x')|^2}\, h(x', F(x')). \label{rel}\ee

 For vertical perturbations
$\phi = \phi(r,\theta)$ of $\Gamma_0$, it is straightforward to compute

\be
\label{HF1}
H'(F_0)[\phi] :=
\ttt L := \ttt L_0 + \ttt L_1,
\ee
 with
\be
\label{HF2}
 \ttt L_0(\phi)\ =
  \frac 1{ r^7\sin^3(2\theta) } \left \{ ( 9 g^2 \,\ttt
w  r^3 \phi_\theta )_\theta + ( r^5 {g'}^2\,\ttt w \phi_r )_r -
3(gg'\,\ttt w  r^4\phi_r )_\theta - 3(gg'\,\ttt wr^4\phi_\theta )_r
\right \},
\ee
and
\begin{align}
 \ttt L_1(\phi)\ &=
  \frac 1{ r^7\sin^3(2\theta) } \left \{    ( r^{-1} \,\ttt
w  \phi_\theta )_\theta + ( r\ttt w \phi_r )_r
\right \}, \label{HF3} \\
\ttt w(r,\theta) &:=\frac {\sin^32\theta}
{
(r^{-4}  + 9g^2 +  {g'}^2)^{\frac 32}} .
\end{align}
We can expand
$$
\ttt w(\theta, r) = \ttt w_0(\theta ) + r^{-4}\,w_1(r,\theta),
$$
where
$$
\ttt w_0(\theta) :=   \frac{\blue{\sin^3 (2\theta)}}{ (    9g^2 + {g'}^2
)^{\frac 32}},   \quad w_1(r,\theta ) =   -\frac 32 \frac{\blue{\sin^3 (2\theta)}}{ (    9g^2 + {g'}^2
)^{\frac 52}} + O( r^{-4} \sin^3 (2\theta) ).
$$
We set
\begin{align}
 L_0(\phi)\ = \
  \frac 1{ r^7\sin^3(2\theta) } &\big \{ ( 9 g^2 \,\ttt
w_0  r^3 \phi_\theta )_\theta \ + \ ( r^5 {g'}^2\,\ttt w_0 \phi_r )_r \nonumber\\ &
\  -
3(gg'\,\ttt w_0  r^4\phi_r )_\theta - 3(gg'\,\ttt w_{\blue 0}r^4\phi_\theta )_r
\big \}.
\label{L0}\end{align}

Crucial in the proof of Proposition \ref{prop11}, as in the arguments that follow below is the presence of explicit solutions that separate variables for the operator $L_0$. Let us consider
the equation
 \be
 L_0(r^\beta q(\theta) ) =   \frac { p(\theta) } {  r^{4 -\beta} } ,
\quad \theta\in (\frac \pi 4, \frac \pi 2),
\label{eqL0}\ee

By a direct computation we obtain
$$
\,{ r^7\sin^3(2\theta) }\, L_0( r^\beta q(\theta))\, = \,  r^{3+\beta} \,[\,9(g^2 \,\ttt w_0 \blue{q}')' - 3\beta (gg' q \,\ttt w_0 )' + \,\ttt w_0 (\beta
+ 4) \, ( \beta {g'}^2 q - 3gg'q')\,] .
$$
We see that
$q=g^{\frac \beta 3}$ annihilates the above operator. As a consequence,
 the  operator takes a divergence form in the function $g^{-\frac \beta 3} q$, namely,
$$
\,{ r^7\sin^3(2\theta) }\, L_0( r^\beta q(\theta))\, = \,  9 r^{3+\beta} \,
g^{\frac{\beta +4} 3}\, \left [\,\ttt w_\blue{0} g^{\frac 23}\, (\,g^{-\frac \beta 3} q \,) '\right ] '.
$$
Thus equation \equ{eqL0} becomes
$$
 \left [\,\ttt w_{\blue{0}} g^{\frac 23}\, (\,g^{-\frac \beta 3} q \,) '\right ] ' = \frac 19 p(\theta) g(\theta)^{ -\frac{\beta +4} 3}\,\sin^3 (2\theta ).
$$
Provided that all quantities are well-defined, we get the following explicit formula for a solution $q(\theta)$, $\theta\in (\frac\pi 4, \frac \pi 2). $\be
q(\theta) =   g^{\frac \beta 3} (\theta)\,\big[ A -\frac 19\int_{\frac
\pi 4} ^{\theta}    g^{-\frac  23} (\,9g^2 + {g'}^2 \,)^{\frac 32}\,\frac{ ds}{
\sin^3(2s)} \int_{
s}^{\frac \pi 2} p(\tau) g^{ - \frac{\beta +4} 3}(\tau)
\sin^{3}(2\tau) \, d\tau \,\big] ,\quad
\label{htheta}\ee
where $A$ is an arbitrary constant.

\begin{lemma} \label{lemahc}
$(a)$ Let $p(\theta)$ be a smooth function, even with respect to $\pi/4$, namely
$$ p(\frac \pi 2 - \theta) =  p(\theta ) \foral \theta \in (0, \frac \pi 4). $$
Then there exists a smooth function $h(r,\theta)$ with
the same symmetry,  that satisfies, for some $\mu>0$,
\be
\JJ_{\Gamma_0} [ h] = \frac {p(\theta)}{r^4} + O( r^{-4-\mu}) \quad  \hbox{as } r\to + \infty,
\label{kk}\ee
and
$$
\| \rr^2(\log \rr)\, h\,\|_{L^\infty(\Gamma_0 )} \ < \ +\infty.
$$

$(b)$ Let $p(\theta)$ be a smooth function, odd with respect to $\pi/4$, namely
$$ p(\frac \pi 2 - \theta) =  - p(\theta ) \foral \theta \in (0, \frac \pi 4). $$
Then there exists a smooth function $h(r,\theta)$ with
the same symmetry,  such that for some $\mu>0$,
\be
\JJ_{\Gamma_0} [  h] = \frac {p(\theta)}{r^3} + O( r^{-4-\mu}) \quad  \hbox{as } r\to + \infty,
\label{hjj}\ee
and
$$
\| \rr\, h\,\|_{L^\infty(\Gamma_0 )} \ < \ +\infty,
$$
and, in addition,
\be
|\nn_{\Gamma_0}h |^2 \ = \  O( r^{-4-\mu}) \quad  \hbox{as } r\to + \infty.
\label{kk1}\ee

\end{lemma}

\proof We will prove next part (a).
We consider first the case in which $p(\pi/4)=0$.
We will construct a smooth function $\phi_0(r,\theta)$  such that for all large $r$ we have
\be
\ttt L(\phi_0) =  \,\frac { p(\theta) } {  r^{4} }  + O( r^{-4-\mu})
\label{ll}\ee
for some $\mu>0$.

Using Formula \equ{htheta} with $\beta =0$ and suitable constant $A$, we see
that
$$
 L_0(q(\theta) ) =  \,\frac { p(\theta) } {  r^{4} } ,
\quad \theta\in (\frac \pi 4, \frac \pi 2),
$$
for

$$
q(\theta) =    -\frac 19 \,\int_{\frac \pi 4}^{\theta }  g^{-\frac  23} (\,9g^2 + {g'}^2 \,)^{\frac 32} \, \,\frac{ ds}{
\sin^3(2s)} \, \int_{
s}^{\frac \pi 2} p(\tau) g^{- \frac{4} 3}(\tau)
\sin^{3}(2\tau) \, d\tau \  ,
$$
Let us analyze the asymptotic behavior of $q(\theta)$ near $\theta = \pi/4$.
Setting
$$ x= \theta -\frac \pi 4 $$
 we can expand
$$
g(\theta )=   g_1 x+  O( x^3 ) , \quad g_1 = g'(\pi/4),\quad p(\theta) = p_2 x^2 + O(x^4), \quad p_2=p''(\pi/4).
$$
Hence we have
$$
\int_{
\theta}^{\frac \pi 2} p(\tau) g^{- \frac{4} 3}(\tau)
\sin^{3}(2\tau) \, d\tau \   =   A_0  +  O( x^{\frac 53})
$$
where
$$ A_0 = \int_{
\frac \pi 4 }^{\frac \pi 2} p(\tau) g^{- \frac{4} 3}(\tau)
\sin^{3}(2\tau) \, d\tau \ .
$$
Thus, we have

$$
q(\theta) =  - {g_1}^{-\frac{11} 3}A_0 \int_0^x  s^{-\frac  23}  \,ds  + O(x^2).
$$
 Hence, for  $A_2=  -3{g_1}^{-\frac{11} 3}A_0 $, we get the expansion
\be
q(\theta) \, =\, A_2 (\theta -\pi/4)^{\frac 13} + O(\theta -\pi/4)^2\,  .
\label{expq}\ee
Now, let us consider
\medskip
Let $\eta(s)$ be a smooth cut-off function such that $\eta(s) = 1$ for $s<1$ and $\eta(s) =0$ for $s>2$.
We consider the interpolation
$$
\phi_0(r,\theta)\ :=\  (1-\eta(s )) q(\theta),\quad s:= r^2g(\theta).
$$
Then, using that $p(\theta) \sim g(\theta)^2 = O( r^{-4}) $ on the support of $\eta$, we get
\begin{align}
L_0(\phi_0) \ =\  \frac{ p(\theta) }{ r^4} \, +\,  O( r^{-10})\,   \, + \,
\nonumber \\
L_0(\eta) \,\psi  \, +\,
  \frac {\ttt w_0  }{ r^4\sin^3(2\theta) } \, 3g\psi_\theta\, [3 g \,\eta_\theta - g'\,  r \eta_r]
\, , \nonumber \\ \psi =   - q(\theta)\,.
\label{L03}\end{align}

Now, we compute
$$
\eta_r =  2\eta' rg  = O( r^{-1}) , \quad \eta_\theta = \eta' r^2g'= O(r^2), \quad
$$
$$
\eta_{rr} =  4\eta''r^2 g^2 + 2\eta' g = O( r^{-2}) , \quad \eta_{r\theta} =
2\eta'' r^3gg' + 2\eta'rg' = O( r), \quad
$$
$$
\eta_{\theta\theta}= \eta'' r^4{g'}^2 + \eta' r^2 g'' = O(r^4).
$$
Substituting these expressions in \equ{L0} we then get
$$
L_0(\eta)\ =\ O( r^{-4}) , \quad [3 g \,\eta_\theta - g'\,  r \eta_r]\ =\ O( 1) ,
$$
while on the other hand in the support of the derivatives of $\eta$ we have

$$
\psi = O( g(\theta)^{\frac 13})= O( r^{-\frac 23})
$$
and also
$$
3g\psi_\theta = 3g q' = O( g(\theta)^{\frac 13}) =  O( r^{-\frac 23}).
$$
Thus, globally we get
$$
L_0(\phi_0) =  \frac{ p(\theta) }{ r^4} \,  + O( r^{-4 -\frac 23 })\,.
$$

Now, let us consider the full operator $\ttt L $ evaluated at this $\phi_0$.
On the one hand, it is straightforward to check that
$$
\ttt L_0(\phi_0)  - L_0(\phi_0) = O( r^{-8}) .
$$
Let us estimate now  $\ttt L_1(\phi_0)$ in \equ{HF3}. We have that
$$
 L_1(\phi_0) \ =\    \frac {(1-\eta) } { r^8\sin^3(2\theta) }  (\,\ttt
w_0  q_\theta )_\theta
 $$
 $$
\  +\
  \frac {\psi } { r^7\sin^3(2\theta) } \left \{    ( r^{-1} \,\ttt
w_0  \eta_\theta )_\theta + ( r\ttt w_0 \eta_r )_r
\right \}
$$
$$
\  + \  \frac {  \,\ttt
w_0  \eta_\theta \psi_\theta  } { r^8\sin^3(2\theta) } \   =\ I_1 + I_2 +I_3
$$
We observe that where $1-\eta$ is supported we have at worst
$$
  q_{\theta \theta} =  O( g^{-\frac 53}) = O( r^{\frac {10}3} )
$$
and hence we find
 $$ I_1\  = \  O( r^{-8+ \frac {10}3}) = O( r^{-4-\frac 43}) .$$
We also compute
$$ I_2\ =\  O( r^{-4 -\frac 23 }),\quad I_3 \ =\  O( r^{- 4 -\frac 43 }) $$
 Hence,
 $$
 L_1(\phi_0)\ = \ O( r^{-4 -\frac 23 }).
$$
We also readily see that $ (\ttt L_1 -L_1) \phi_0$ is even smaller than the above bound.
We conclude
\be
\ttt L( \phi_0) = \frac{ p(\theta)} {r^4}\ +\  O( r^{-4-\frac 23})\, .
\label{fin1}\ee
where $\phi_0$ is a symmetric, smooth bounded function.
We recall that we have obtained this under the assumption that $p(0) =0$.
We consider next the case $p(0) \ne 0$.

\medskip
Let us compute
$L_0(\log r)$. We get
\begin{align}
 L_0(\log r)\ = \
  \frac 1{ r^7\sin^3(2\theta) } &\big \{  4 r^3 {g'}^2(\theta)  - 3 r^3 (gg'\,\ttt w_0   )_\theta  \big\} \nonumber\\
\ =\   \frac 1{ r^4\sin^3(2\theta) } & \big \{  3 {g'}^2(\theta)  - 3 gg''\,\ttt w_0   - 3gg'
(\ttt w_0 )_\theta  \big\}.
\nonumber\\
\ =\  \frac 1{ r^4}
\big \{   \frac { {g'}^2}{ (9g^2 + {g'}^2)^{3/2}}  &
-  \frac{ 3 gg''} { (9g^2 + {g'}^2)^{3/2}} \,  -
\frac { 3 g' g }{ \sin^3(2\theta) } (\ttt w_0 )_\theta \}.
\label{L01}\end{align}

Then we observe that we can decompose
\begin{align}
 \, L_0(\log r)\ = \
\frac { g_1^{-1} }{ r^4} \,+ \,\frac { b(\theta) }{ r^4} \, , \quad g_1 = g'(\pi/4)
\label{log}\end{align}
where $b(\theta)$ is symmetric, smooth and with $b(\pi/4) =0$.
In addition, we readily check that
$$
\ttt L_1(\log r) = O( r^{-12}), \quad (\ttt L_0 -L_0)(\log r) = O( r^{-11}),
$$
hence
\be
L(\log r ) \, =\, \frac { g_1^{-1} }{ r^4} \,+ \,\frac { b(\theta) }{ r^4} \,+ \, O( r^{-11}).
\label{fin2}\ee
Hence, if we let
$$
A:=   g_1 p(\pi/4),
$$
then we have that
\be
L( A \log r ) \, =\, \frac { p(\theta) }{ r^4} \,- \,\frac { p_1(\theta) }{ r^4} \,+ \, O( r^{-11}).
\label{fin22}\ee
where
$$
 \quad p_1(\theta) := - A b(\theta) + p(\theta) -p(\pi/4) .
$$
Now, let us consider a bounded approximate solution $\phi_0 (r,\theta)$ as built above   where $p$ is replaced by $p_1$. We see then that
$$ \phi_1 := A \log r + \phi_0 $$
satisfies
\be
L( \phi_1 ) \, =\, \frac { p(\theta) }{ r^4}  \,+ \, O( r^{-4-\frac 23 }).
\label{fin22222}\ee
Observe that then the function
$$h := (1- \eta(r))\,(1+ |\nn F_0|^2)^{-1/2} \phi_1$$ is smooth, symmetric, and satisfies \equ{kk} The proof of part (a) is concluded.

\bigskip
We prove now part (b).
Let us consider Formula \equ{htheta} for $\beta =1$
We have now that
\be
\ttt L( r\,q(\theta) ) =  \,\frac { p(\theta) } {  r^{3} }, \theta \in (\frac \pi 4, \frac \pi 2)
\label{ll111}\ee
for
\be
q(\theta) =     g^{\frac 1 3} (\theta)\,\int_{\frac
\pi 4}^\theta     g^{-\frac  23} (\,9g^2 + {g'}^2 \,)^{\frac 32}\,\frac{ ds}{
\sin^3(2s)} \int_{
s}^{\frac \pi 2}  p(\tau) g^{- \frac{5}{ 3}}(\tau)
\sin^{3}(2 \tau) \, d\tau \, .
\label{hthetaaa}\ee
Since $p(0)= 0$ and $p$ is smooth, we have that
the asymptotic behavior of $q(\theta)$ near $\theta = \pi/4$ is now given by
$$
q(\theta)\ =  A_1 (\theta - \pi/ 4 )^{\frac 23} + O(\theta - \pi/ 4 )^{\frac 53}.
$$
Then we define
$$ \phi_0(r, \theta)  =  (1-\eta(s) )\, r\,q(\theta) , \quad \theta \in (\frac \pi 4, \frac \pi 2),\quad s= r^2g(\theta).
$$
Similar computations as in the proof of Lemma \ref{lemahc}  lead us now to
$$
L_0( \phi_0) = \frac {p(\theta)} {r^3 }  + O( r^{-4- \frac 13}),\quad L_1( \phi_0) =  O( r^{-4- \frac 13}), $$
and consistently to
$$
\ttt L( \phi_0) = \frac {p(\theta)} {r^3 }  + O( r^{-4- \frac 13}). $$
Finally, the function
$$ h = \frac {\phi_0} {\sqrt{ 1+ |\nn F_0|^2 } }$$
extended oddly through $\theta = \frac \pi 4$ satisfies \equ{hjj}.

\medskip
Next we want to  estimate the quantity
$$
|\nabla_{\Gamma_0}h |^2  = g^{ij}\,\partial_i h\, \partial_j h\, ,
$$
where $g^{ij}$ denotes the inverse of the matrix with the coefficients of the metric in a system of local coordinates on $\Gamma_0$.
Let us consider the parametrization in polar coordinates
$$
(\uu_1, r, \uu_2, \theta )\in  S^3\times \R_+ \times S^3\times  (0, \pi) \longmapsto   ( r\cos\theta \uu_1, r\sin\theta \uu_2, F_0( r,\theta) )
$$
where $F_0= r^3g(\theta)$.
Then the matrix $g_{ij}$  takes the form
$$
\left(\begin{array}{cccc}  r^2\cos^2\theta I_3 &0 &0 &0 \\
0& (1+ F_{0r}^2) &0&0\\
0&0&  r^2\sin^2\theta  I_3 &0\\
0&0&0& (r^2+F_{0\theta}^2)
\end{array}
\right)
$$
and its inverse is therefore
$$
\left(\begin{array}{cccc}  r^{-2}\cos^{-2}\theta \, I_3 &0 &0 &0 \\
0& (1+ F_{0r}^2)^{-1}  &0&0\\
0&0&   r^{-2}\sin^{-2}\theta \, I_3 &0\\
0&0&0& (r^2+F_{0\theta}^2)^{-1}
\end{array}
\right)
$$
where $I_3$ is the $3\times 3$ identity matrix. Hence, if we evaluate at a function $h=h(r,\theta)$
we simply get
$$
|\nn_{\Gamma} h |^2\ =\  \frac{1}{1+F_{0r}^2}
|\partial_r h |^2 \ +\ \frac{1}{r^2 +F_{0\theta}^2}
|\partial_\theta h |^2
$$
or
\be
|\nn_{\Gamma} h |^2\ =\  \frac{1}{1+ 9 g^2 r^4 }
|\partial_r h |^2 \ +\ \frac{1}{r^2 +  r^6 {g'}^2}
|\partial_\theta h |^2
\label{n2}\ee
Let us set $ \ttt h = (1-\eta(s))\,q(\theta) r^{-1}$ so that
$$ |\nn_{\Gamma} h |^2 = |\nn_{\Gamma} \ttt h |^2 + O( r^{-8}) .$$
Evaluating formula \equ{n2} at $\ttt h$  for  $ s = r^2 g> 2$,
we get globally that
$$
 \frac{1}{1+ 9 g^2 r^4 }|\partial_r \ttt h |^2  \, \sim \,  \frac 1{r^4} \frac{ q^2(\theta) }{1+ 9 g^2 r^4 } = O( r^{-4-\mu})
$$
globally, since  $q(\theta)^2 \sim g^{\frac 43}$.
On the other hand,
$$
\frac{1}{r^2 +  r^6 {g'}^2}
|\partial_\theta h |^2 =      \frac{1}{r^4 +  r^8 {g'}^2} |q'(\theta ) |^2
$$
We have that $|q'|^2 \sim g^{-\frac 23} \le C r^{\frac 43} $ and hence the above quantity
is $O( r^{ -4 -\sigma})$ at least away from $\theta = \frac \pi 2$.
Near $\frac \pi 2$ we use that $q'(\pi/2)=0$ to get the same smallness there.

\medskip
Thus
$$
|\nn_{\Gamma} h |^2\ =\ O( r^{ -4 -\mu})
$$
for $ s> 2$. Now, in the region $1<s<2$, where the cut-off acts, we take into account that
$$\eta_\theta =  O( r^2) $$
and get that the contribution of this term to the computation of
$$\frac{1}{r^2 +  r^6 {g'}^2}
|\partial_\theta \ttt h |^2$$
is like
$$
\sim r^{-4} q(\theta)^2 = O( r^{-4 -\mu}).
$$
The contribution of the derivative in $r$ yields also a small order term.
Hence, we have in the entire region that
$$
|\nn_{\Gamma} h |^2\ =\ O( r^{ -4 -\mu})
$$
and the validity of \equ{kk1} follows. The proof is concluded. \qed

\bigskip
 \begin{proposition}\label{prophc}
 $(a)$ Problem \equ{ej} has a solution $h$
 with
 $$
 \|\rr^2 (\log \rr)\, h\|_{L^\infty(\Gamma)}\ < +\infty
 $$
 if
 $$
\tg= \sum_{i=1}^8 k_i^4 \quad\hbox{or}\quad \tg=\left [\sum_{i=1}^8 k_i^2 \right ]^2.
 $$

  $(b)$
  If
 $$
\tg= \sum_{i=1}^8 k_i^3 ,
 $$ then Problem \equ{ej} has a solution $h$
 with
 $$
  \|\rr^{2+\mu} \, D_\Gamma h\|_{L^\infty(\Gamma)}\, + \,  \|\rr \, h\|_{L^\infty(\Gamma)}\ < +\infty.
 $$

 \end{proposition}

 \proof Let us prove Part (a).
 Let $k_i^0$ denote the principal curvatures of $\Gamma_0$. Then we compute directly that the functions
  $$
 \sum_{i=1}^8 |k_i^0|^4 \quad\hbox{and}\quad \left [\sum_{i=1}^8 |k_i^0|^2 \right ]^2
 $$
 are both of the form (for large $r$)
 $$
 \tg (y) = \frac {p(\theta)} {r^4}
 $$
 with $p$ symmetric and smooth.
 In addition, we have that away from the origin,
 $$
 \sum_{i=1}^8 k_i^4( y) \ =\
 \sum_{i=1}^8 |k_i^0(\pi(y))|^4 \ +\ O( \rr(y)^{-6})
 $$
 Let $h_0$ be the approximate solution predicted by Part (a) of Lemma \ref{lemahc} in $\Gamma_0$, so that for instance
 $$
 \Delta_{\Gamma_0} h_0  + |A_{\Gamma_0}|^2 h_0 =   \sum_{i=1}^8 |k_i^0|^4  + O( \rr^{-4-\mu})
$$
where
$$
\| \rr^{2}\log \rr \,  h_0\|_{L^\infty(\Gamma_0)} < +\infty. $$

Let $h_1(y) := h_0(\pi (y))$. Then, according to Lemma \ref{comparisonjacobi} and a direct computation
we find that
$$
\JJ_\Gamma [h_1] (y)  = \JJ_{\Gamma_0} [h_0] (\pi(y) ) + O( \rr(y)^{-4-\mu}) .
$$
Hence
$$
\JJ_\Gamma [h_1] (y)  =   \sum_{i=1}^8 k_i^4( y)  + \zeta (y)
$$
where $\zeta = O( \rr^{-4-\mu})$
By Proposition \ref{prop11} there exists a solution $h_2$ of
$$
\JJ_\Gamma [h_2]   =  - \zeta
$$
with $\| \rr^{2+\mu}\, h_1\|_{L^\infty(\Gamma)} < +\infty$.
The desired result follows by simply setting $ h:= h_1 + h_2$. The proof for the other right hand side is the same. For part (b) the argument is similar, taking into account Part (b)  Lemma \ref{lemahc}. \qed

\medskip
\subsection{Weighted Schauder  estimates}

\medskip
We have the following result, that controls the decay of the first two derivatives of  solutions
of  equation \equ{ej}.

\begin{lemma}
Let $\nu \ge 2$. There exists a constant $C>0$ such that the following holds.
Let $h$ be a solution of equation \equ{ej} such that
$$  \|{\tt g}\|_{\sigma,\nu,\Gamma} + \|\rr^{\nu-2} h\|_{L^\infty(\Gamma)} < +\infty. $$
Then
\be \|D^2_\Gamma h \|_{\sigma,\nu,\Gamma} +
 \|h\|_{\sigma,\nu-2,\Gamma}\, \le \, C\,[\,  \|{\tt g}\|_{\sigma,\nu,\Gamma} + \|\rr^{\nu-2} h\|_{L^\infty(\Gamma)}]. \label{res}\ee

\end{lemma}

\proof


We use the local coordinates \equ{cords}. Then, around a point $p$ with $r(p) =R$, for any sufficiently large $R$,
the equation reads on $B(0,2\theta R)$ for a  small, fixed $\theta>0$ as
$$
a_{ij}^0(\py) \pp_{ij} h + b_i^0(\py) \pp_i h  + |A_\Gamma(\py)|^2 h =  {\tt g}(\py) \quad \hbox{in }  B(0,2\theta R).
$$
Consider the scalings $$ \ttt h (\py ) = R^{\nu-2} h( R\py ) ,\quad \ttt {\tt g} (\py ) = R^{\blue{\nu}} \blue{\tt g}( R\py ) .$$
 Then we obtain the following equation.
$$
\ttt a_{ij}^0( \py) \pp_{ij} \ttt h + \ttt b_i^0( \py) \ttt \pp_i h + \ttt b_0(\py) h =  \ttt {\tt g} \quad \hbox{in }  B(0,2\theta ),
$$
where
$$
\ttt a_{ij}(\py) = a_{ij}^0(R \py), \quad  \ttt b_i(\py ) = R b_i^0(R \py), \quad \ttt b_0(\py) = R^2 |A_\Gamma(R \py)|^2.
$$
We will apply interior elliptic estimates to this equation. First, let us notice that from the estimates obtained for the metric, we have that
the coefficients above are all uniformly bounded an elliptic in $B(0,2\theta )$. Besides, we have that
their first derivatives are also bounded in this region, with bounds uniform on the point $p$ and on $R$.

Elliptic estimates then yield
\be
\| D^2_\py\ttt  h \|_{ C^{0, \sigma}( B(0, \theta))} +  \| \ttt h \|_{ C^{0, \sigma}( B(0, \theta))} \,\le\,  C [ \|{\ttt \tg}\|_{C^{0,\sigma}(B(0,2 \theta))} + \|\ttt h\|_{L^\infty (B(0,2\theta) )}].
\label{estelipt}\ee
Let us observe that for any $\py_1, \py_2\in B(0, 2\theta )$ we have

$$
| \ttt \tg (\py_1) |\ =\ |R^\nu \tg( R\py)| \ \le\  C\| \rr^{\nu } {\tt g}\|_{ L^\infty (\Gamma )},
$$
and
$$ \frac {|\ttt \tg (\py_1) - \ttt \tg(\py_2 ) |}{ |\py_1-\py_2|^\sigma } =  R^{\nu + \sigma} \frac {|{\tt \tg} (R \py_1) -   \tg(R\py_2) |}{ |R\py_1- R\py_2|^\sigma } \,\le \, C [  {\tt g}]_{\sigma,\nu, \Gamma} .
$$
Therefore, we have the inequalities
\be
\|{\ttt \tg }\|_{C^{0,\sigma}(B(0,2 \theta))}\ \le \ C\,\|\tg\|_{\sigma,\nu,\Gamma} , \quad  \|\ttt h\|_{L^\infty (B(0,2\theta) )} \,\le\, C\,\|\rr^{\nu-2} h\|_{L^\infty (\Gamma) }.
\label{dd}\ee

Now, we have that
$$
D^2\ttt h (\py) = R^\nu [D^2 h] (R\py)
$$
Hence for $\py_1, \py_2\in B(0, \theta R)$ we have
$$
R^\nu \frac { D^2h (\py_1) - D^2h (\py_2)}{ | \py_1 -\py_2|^\sigma }\ =\   \frac { D^2\ttt h ( R^{-1}\py_1) - D^2\ttt h ( R^{-1}\py_2)}{ | \py_1 -\py_2|^\sigma } \, \le\,  C\, R^{-\sigma } \,\| \ttt h \|_{C^{0,\sigma}( B(0, \theta))}
$$
It follows that
if $\Lambda = Y_p(B(0,\theta))$ then
$$
[ D^2_\Gamma h]_{\nu,\sigma,\Lambda} \ \le \ C\,\| D^2 \ttt h \|_{C^{0,\sigma}( B(0, \theta))}.
$$
Similarly we have that
$$
[ h ]_{\nu-2 ,\sigma,\Lambda} \ \le \ C\,\| \ttt h \|_{C^{0,\sigma}( B(0, \theta))},
$$
while clearly, also,
$$
\| \rr^{\nu-2} h\|_{L^\infty(\Lambda)}  + \| \rr^\nu D^2_\Gamma h\|_{L^\infty(\Lambda)} \ \le \ C\,[\, \| \ttt h \|_{C^{0,\sigma}( B(0, \theta))} +  \|D^2 \ttt h \|_{C^{0,\sigma}( B(0, \theta))}\, ]
$$
Hence from inequalities \equ{estelipt} and \equ{dd} we obtain

$$ \|D^2_\Gamma h \|_{\sigma,\nu, \Lambda } +
 \|h\|_{\sigma,\nu-2,\Gamma}\, \le \, C\,[\,  \|{\tt g}\|_{\sigma,\nu,\Gamma} + \|\rr^{\nu-2} h\|_{L^\infty(\Gamma)}], $$
where $C$ is uniform in $p$ con $\rr(p) >> 1$. Using this and an interior estimate for the equation on a bounded region, the desired estimate \equ{res} follows. \qed

\bigskip
\begin{corollary}\label{coro}

$1.$ The solution $h$ predicted by Proposition \ref{prop11}
satisfies the estimate
$$ \|D^2_\Gamma h \|_{\sigma,\nu,\Gamma} +
 \|h\|_{\sigma,\nu-2,\Gamma}\, \le \, C\,  \|{\tt g}\|_{\sigma,\nu,\Gamma}. $$

$2.$ The solution in Part $(a)$ of Proposition \ref{prophc}
satisfies that for any small $\tau >0$,
$$ \|D^2_\Gamma h \|_{\sigma, 4-\tau,\Gamma} +
 \|h\|_{\sigma,2-\tau,\Gamma}\, <\, +\infty $$

$3.$  The solution in Part $(b)$ of Proposition \ref{prophc}
satisfies
$$ \|D^2_\Gamma h \|_{\sigma, 3 ,\Gamma} +
 \|h\|_{\sigma, 1,\Gamma}\, <\, +\infty $$
 while for some $\mu>0$
$$
\|D_\Gamma h\|_{\sigma, 2+\mu ,\Gamma}\, <\,  +\infty.
$$
\end{corollary}

\bigskip

\noindent
{\bf Acknowledgments.} Manuel del Pino acknowledges
support of Fondecyt  110181 and
Fondo Basal CMM.  Juncheng Wei is partially supported by NSERC of Canada.
 We thank  A. Farina,  M. Kowalczyk and A. Ros for useful discussions.



\bigskip























































































\begin{thebibliography}{AAA}



\bibitem{alexandrov}
A. D. Alexandrov.   {\em Uniqueness theorems for surfaces in the large. I}, (Russian) Vestnik Leningrad Univ.
Math. 11, 5-17 (1956).


\bibitem{ac}
L. Ambrosio and X. Cabr\'e, {\em Entire solutions of semilinear elliptic equations
in $\R^3$ and a conjecture of De Giorgi}, J. Amer. Math. Soc. 13
(2000), 725--739.

\bibitem{bcn2}
 H. Berestycki, L.A. Caffarelli, L. Nirenberg, {\em Monotonicity for elliptic equations in unbounded Lipschitz domains.}
  Comm. Pure Appl. Math. 50(11), 1089–-1111 (1997)



\bibitem{BCE} J.L. Barbosa and M.P. do Carmo, and Jost Eschenburg, {\em Stability of hypersurfaces of constant mean curvature in Riemannian manifolds}, Math. Z. 197 (1988), no. 1, 123-138.



  \bibitem{BDG}
 E. Bombieri, E. De Giorgi, E. Giusti, Minimal cones and the Bernstein problem, Invent. Math. 7 (1969) 243–-268.


\bibitem{costa}
C.J. Costa, {\em Example of a complete minimal immersions in $\R^3$ of genus one and three embedded ends,}
Bol. Soc. Bras. Mat. 15(1-2) (1984), 47–-54.


\bibitem{dg}
E. De Giorgi, {\em Convergence problems for functionals and operators}, Proc. Int. Meeting on Recent Methods in Nonlinear Analysis (Rome, 1978), 131–-188, Pitagora, Bologna (1979).




\bibitem{dkwdg}
M. del Pino, M. Kowalczyk and J. Wei, {\em On De Giorgi's Conjecture in Dimensions $N \ge 9$,} Annals of Mathematics 174 (2011), no. 3, 1485–-1569.

\bibitem{jdgdkw}
M. del Pino, M. Kowalczyk and J. Wei, {\em
Entire solutions of the Allen-Cahn equation and complete embedded minimal surfaces of finite total curvature in $\R^3$.} J. Differential Geom. 93 (2013), no. 1, 67–131.


\bibitem{FV}
A. Farina and E. Valdinoci, {\em Flattening results for elliptic PDEs in unbounded domains with applications to overdetermined problems}. Arch. Ration. Mech. Anal. 195 (2010), no. 3, 1025–-1058.

\bibitem{FV2} A. Farina and E. Valdinoci, {\em The state of the art for a conjecture of De Giorgi and related problems.}
 Recent progress on reaction-diffusion systems and viscosity solutions, 74–-96, World Sci. Publ., Hackensack, NJ, (2009).



\bibitem{GNN}
B. Gidas, W.-M. Ni and L. Nirenberg. {\em Symmetry and related properties via the maximum principle.} Commun.
Math. Phys. 68, 209--243 (1979).

\bibitem{gg} N. Ghoussoub and C. Gui, {\em On a conjecture of De Giorgi and some related
problems}, Math. Ann. 311 (1998), 481--491.


\bibitem{pacard}
L. Hauswirth,  F. Helein and F. Pacard, {\em  On an overdetermined elliptic problem.} Pacific J. Math. 250 (2011), no. 2, 319–-334

\bibitem{hm}
D. Hoffman,  W.H. Meeks III, {\em Embedded minimal surfaces of finite topology}, Ann. of  Math. 131 (1990), 1–-34.

\bibitem{Maz-Pac} R. Mazzeo and F. Pacard, {\em Constant mean curvature surfaces with Delaunay ends}. Comm. Analysis and Geometry. 9, 1, (2001), 169-237.



\bibitem{morabito}
F. Morabito, {\em Index and nullity of the Gauss map of the Costa-Hoffman-Meeks surfaces,} Indiana Univ. Math. J. 58 (2009), no. 2, 677–-707.

\bibitem{nayatani} S. Nayatani, {\em Morse index and Gauss maps of complete minimal surfaces in Euclidean $3-$space}, { Comm. Math. Helv.}  68(4)(1993), 511--537.
\bibitem{pr} F. Pacard and M. Ritor\'e, {\em From constant mean curvature hypersurfaces to the
gradient theory of phase transitions}, J. Differential Geom. 64 (2003), 359--423.




\bibitem{savin}
O. Savin, {\em Regularity of
at level sets in phase transitions,} Ann. of Math. 169
(2009), 41--78.

\bibitem{scicbaldi}
F. Schlenk,  P. Sicbaldi, {\em  Bifurcating extremal domains for the first eigenvalue of the Laplacian.} Adv. Math. 229 (2012), no. 1, 602–-632.



\bibitem{serrin}
J. Serrin. {\em A symmetry problem in potential theory.} Arch.Rat.Mech. Anal. 43, 304-318 (1971).

\bibitem{simons}
J. Simons, {\em Minimal varieties in Riemannian manifolds,} Ann. of Math. 88 (1968).

\bibitem{Whi} B. White, {\em The space of minimal submanifolds for varying Riemannian metrics}, Indiana Univ. Math. J. 40 (1991), no. 1, 161–200


\end{thebibliography}
\end{document}